\documentclass[11pt]{amsart}
\usepackage[margin=1in]{geometry}

\usepackage{amssymb}
\usepackage{amsthm}
\usepackage{amsmath}
\usepackage{mathrsfs}
\usepackage{amsbsy}
\usepackage[all]{xy}
\usepackage{bm}
\usepackage{hyperref}
\usepackage{tikz}
\usepackage{array}
\usepackage{enumerate}
\usepackage{xcolor}
\usepackage{bbm}
\usepackage{comment}
\usepackage{dynkin-diagrams}
\usepackage{ifthen}
\usepackage{thmtools, thm-restate}

\newcommand*\centcell[1]{\omit\hfil$\displaystyle#1$\hfil\ignorespaces}

\hypersetup{colorlinks=true, citecolor=B2TFYellow, linkcolor=lavender,urlcolor=lavender}

\definecolor{lavender}{rgb}{1,0,0}
\definecolor{Back2TheBlueture}{rgb}{0,0.5,1}
\definecolor{B2TFYellow}{rgb}{1,0.8,0}
\definecolor{WickedGreen}{rgb}{0,0.9,0}

\definecolor{Traj1}{rgb}{1,0.588,0.392}
\definecolor{Traj2}{rgb}{0,0.824,1}
\definecolor{Traj3}{rgb}{0,0,0.902}
\definecolor{Traj4}{rgb}{0,0.784,0}
\definecolor{Traj5}{rgb}{1,0,1}

\definecolor{darkred}{rgb}{0.7,0,0}
\definecolor{ArrowBlue}{rgb}{0,0.392,1}

\usepackage{hhline}
\allowdisplaybreaks
\usepackage[noadjust]{cite}
\usepackage{asymptote}

\usepackage{caption}
\usepackage{tabu}
\usepackage{diagbox}
\usepackage[noabbrev,capitalise,nameinlink]{cleveref}
\crefname{conjecture}{Conjecture}{Conjectures}

\newtheorem{theorem}{Theorem}[section]
\newtheorem{proposition}[theorem]{Proposition}
\newtheorem{corollary}[theorem]{Corollary}

\newtheorem{question}[theorem]{Question}

\newtheorem{lemma}[theorem]{Lemma}

\theoremstyle{definition}
\newtheorem{definition}[theorem]{Definition}
\newtheorem{remark}[theorem]{Remark}
\newtheorem{example}[theorem]{Example}

\definecolor{NormalGreen}{RGB}{0,220,0}
\definecolor{ChillBlue}{RGB}{0,182,255}
\definecolor{MyOrange}{RGB}{255,150,0}
\definecolor{Maroon}{RGB}{150,0,0}

\newcommand{\NN}{\mathbf{D}}
\newcommand{\TT}{\mathrm{T}}
\newcommand{\quot}{\pi}
\newcommand{\pp}{\mathfrak{p}}
\newcommand{\qq}{\mathfrak{q}}
\newcommand{\rr}{\mathfrak{r}}
\newcommand{\oo}{\mathfrak{o}}
\newcommand{\id}{\mathbbm{1}}
\newcommand{\elll}{d}
\newcommand{\ao}{\mathrm{ao}}
\newcommand{\BB}{\mathbb{B}}

\newcommand{\MM}{\mathrm{M}}
\newcommand{\HH}{\mathrm{H}}
\newcommand{\Heap}{\mathrm{Heap}}
\newcommand{\Inv}{\mathrm{Inv}}
\newcommand{\x}{\mathsf{x}}
\newcommand{\y}{\mathsf{y}}

\newcommand{\sort}{\mathsf{sort}}
\newcommand{\sfc}{\mathsf{c}}

\newcommand{\wo}{w_{\circ}}
\newcommand{\LL}{\mathscr{L}}
\newcommand{\RR}{\mathrm{R}}
\newcommand{\Cay}{\mathrm{Cay}}
\newcommand{\Gyr}{\mathrm{Gyr}}
\newcommand{\Ev}{\mathrm{Ev}}
\newcommand{\Pro}{\mathrm{Pro}}
\newcommand{\BK}{\mathrm{BK}}
\newcommand{\dom}{\mathrel{\mathsf{dom}}}
\newcommand{\lab}{\mathrm{lab}}
\newcommand{\Sep}{\mathrm{Sep}}
\newcommand{\Str}{\mathrm{Str}}

\usepackage{etoolbox}

\DeclareRobustCommand{\fold}{\raisebox{0.2ex}{\resizebox{\width}{0.5\height}{\rotatebox{90}{\textsf{w}}}}}

\newcommand{\dfn}[1]{\textcolor{Back2TheBlueture}{\emph{#1}}}

\begin{document}

\title[]{Bender--Knuth Billiards in Coxeter Groups}
\subjclass[2010]{}

\author[]{Grant Barkley}
\address[]{Department of Mathematics, Harvard University, Cambridge, MA 02138, USA}
\email{gbarkley@math.harvard.edu}

\author[]{Colin Defant}
\address[]{Department of Mathematics, Harvard University, Cambridge, MA 02138, USA}
\email{colindefant@gmail.com}

\author[]{Eliot Hodges}
\address[]{Department of Mathematics, Harvard University, Cambridge, MA 02138, USA}
\email{eliothodges@college.harvard.edu}

\author[]{Noah Kravitz}
\address[]{Department of Mathematics, Princeton University, Princeton, NJ 08540, USA}
\email{nkravitz@princeton.edu}

\author[]{Mitchell Lee}
\address[]{Department of Mathematics, Harvard University, Cambridge, MA 02138, USA}
\email{mitchell@math.harvard.edu}

\maketitle

\begin{abstract}
Let $(W,S)$ be a Coxeter system, and write $S=\{s_i:i\in I\}$, where $I$ is a finite index set. Fix a nonempty convex subset $\mathscr{L}$ of $W$. If $W$ is of type $A$, then $\mathscr{L}$ is the set of linear extensions of a poset, and there are important \emph{Bender--Knuth involutions} $\mathrm{BK}_i\colon\mathscr{L}\to\mathscr{L}$ indexed by elements of $I$.  For arbitrary $W$ and for each $i\in I$, we introduce an operator $\tau_i\colon W\to W$ (depending on~$\mathscr{L}$) that we call a \emph{noninvertible Bender--Knuth toggle}; this operator restricts to an involution on $\mathscr{L}$ that coincides with $\mathrm{BK}_i$ in type~$A$. Given a Coxeter element $c=s_{i_n}\cdots s_{i_1}$, we consider the operator $\mathrm{Pro}_c=\tau_{i_n}\cdots\tau_{i_1}$. We say $W$ is \emph{futuristic} if for every nonempty finite convex set $\mathscr{L}$, every Coxeter element $c$, and every $u\in W$, there exists an integer $K\geq 0$ such that $\mathrm{Pro}_c^K(u)\in\mathscr{L}$. We prove that finite Coxeter groups, right-angled Coxeter groups, rank-3 Coxeter groups, affine Coxeter groups of types $\widetilde A$ and $\widetilde C$, and Coxeter groups whose Coxeter graphs are complete are all futuristic. When $W$ is finite, we actually prove that if $s_{i_N}\cdots s_{i_1}$ is a reduced expression for the long element of $W$, then $\tau_{i_N}\cdots\tau_{i_1}(W)=\mathscr{L}$; this allows us to determine the smallest integer $\mathrm{M}(c)$ such that $\mathrm{Pro}_c^{{\mathrm{M}}(c)}(W)=\mathscr{L}$ for all $\mathscr{L}$. We also exhibit infinitely many non-futuristic Coxeter groups, including all irreducible affine Coxeter groups that are not of type $\widetilde A$, $\widetilde C$, or $\widetilde G_2$. 
\end{abstract}

\begin{figure}[ht]
  \begin{center}{\includegraphics[height=8cm]{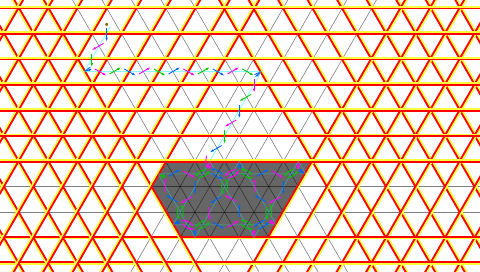}}
  \end{center}
\caption{The Coxeter arrangement of $\widetilde{A}_2$ forms a triangular grid whose unit triangles correspond to the elements of $\widetilde{A}_2$. The black convex set $\LL$ turns each hyperplane in the Coxeter arrangement into either a transparent window (indicated by a thin gray line) or a one-way mirror (indicated by a line colored yellow and red). Starting at the initial unit triangle marked with the brown dot, we apply the noninvertible Bender--Knuth toggles ${\color{ArrowBlue}\tau_0},{\color{Traj5}\tau_1},{\color{Traj4}\tau_2},{\color{ArrowBlue}\tau_0},{\color{Traj5}\tau_1},{\color{Traj4}\tau_2},\ldots$. This has the effect of following a thin cyan beam of light that eventually gets trapped in $\LL$.  }\label{fig:affineS3}
\end{figure}

\section{Introduction}\label{sec:intro}

\subsection{Linear Extensions and Noninvertible Bender--Knuth Toggles}\label{subsec:intro-posets}

Let ${P=([n+1],\leq_P)}$ be a poset whose underlying set is $[n+1]=\{1,\ldots,n+1\}$. Let $\mathfrak S_{n+1}$ denote the symmetric group whose elements are the permutations of $[n+1]$ (i.e., bijections from $[n+1]$ to itself). We can think of a permutation $w\in \mathfrak S_{n+1}$ as a labeling of $P$ in which the element $i$ receives the label $w(i)$. We say that the permutation $w$ is a \dfn{linear extension} of $P$ if $i\leq_P j$ implies $w(i)\leq w(j)$. Let $\mathcal L(P)$ denote the set of linear extensions of $P$.  

For each $i\in[n]$, the \dfn{Bender--Knuth involution} $\BK_i\colon\mathcal L(P)\to\mathcal L(P)$ is defined by \[\BK_i(u)=\begin{cases} u & \mbox{if } u^{-1}(i)\leq_Pu^{-1}(i+1); \\ (i\,\,i+1)u & \mbox{otherwise,} \end{cases}\]
where we write $(i\,\,j)$ for the transposition that swaps $i$ and $j$. These involutions are fundamental in algebraic combinatorics \cite{Ayyer, BenAndJanabel, DefantPolyurethane, DefantFriends, Edelman, Haiman, Malvenuto, Massow, Poznanovic, StanleyPromotion, StrikerWilliams}, largely due to the useful folklore fact that one can reach any linear extension of $P$ from any other linear extension of $P$ via a sequence of Bender--Knuth involutions. Haiman \cite{Haiman} and Malvenuto--Reutenauer \cite{Malvenuto} found that Sch\"utzenberger's famous \dfn{promotion} and \dfn{evacuation} operators $\Pro$ and $\Ev$ can be written as \[\Pro=\BK_{n}\cdots\BK_1\quad\text{and}\quad\Ev=(\BK_1)(\BK_2\BK_1)(\BK_3\BK_2\BK_1)\cdots(\BK_{n}\cdots\BK_2\BK_1).\]  

For $i\in[n]$, define the \dfn{noninvertible Bender--Knuth toggle} $\tau_i\colon \mathfrak S_{n+1}\to \mathfrak S_{n+1}$ by 
\begin{equation}\label{eq:toggle_def_posets}\tau_i(u)=\begin{cases} u & \mbox{if } u^{-1}(i)\leq_Pu^{-1}(i+1); \\ (i\,\,i+1)u & \mbox{otherwise.} \end{cases}
\end{equation}
Of course, $\tau_i$ is an extension of $\BK_i$ from the set of linear extensions of $P$ to the set of all permutations of $[n+1]$; it is noninvertible (unless $\mathcal L(P)=\mathfrak S_{n+1}$, i.e., $P$ is an antichain) because $\tau_i((i\,\,i+1)u)=\tau_i(u)=u$ whenever ${u^{-1}(i)\leq_P u^{-1}(i+1)}$. Defant and Kravitz introduced noninvertible Bender--Knuth toggles in \cite{DefantPromotionSorting} in order to define the \dfn{extended promotion} operator $\Pro\colon \mathfrak S_{n+1}\to \mathfrak S_{n+1}$, which is the composition \[\Pro=\tau_{n}\cdots\tau_1.\] Hodges further investigated extended promotion in \cite{Hodges}. Restricting extended promotion to $\mathcal L(P)$ recovers Sch\"utzenberger's (invertible) promotion operator; hence, we slightly abuse notation by using the same symbol $\Pro$ to denote both of these operators. In a similar vein, we define the \dfn{extended evacuation} operator $\Ev\colon \mathfrak S_{n+1}\to \mathfrak S_{n+1}$ by \[\Ev=(\tau_1)(\tau_2\tau_1)(\tau_3\tau_2\tau_1)\cdots(\tau_{n}\cdots\tau_2\tau_1).\] 

One of the main results of \cite{DefantPromotionSorting} is that 
\begin{equation}\label{eq:ProSort}
\Pro^{n}(\mathfrak S_{n+1})=\mathcal L(P).
\end{equation} In other words, if we start with a permutation in $\mathfrak S_{n+1}$ and iteratively apply extended promotion $n$ times, then we will ``sort'' the permutation to a linear extension (and every element of $\mathcal{L}(P)$ is in the image of $\Pro^n$ since $\Pro$ permutes the elements of $\mathcal{L}(P)$). In fact, the same argument shows that \begin{equation}\label{eq:EvSort}
\Ev(\mathfrak S_{n+1})=\mathcal L(P).
\end{equation}

It is natural to consider other compositions of noninvertible Bender--Knuth toggles. For example, \dfn{gyration} \cite{StrikerWilliams, Wieland} is the operator $\Gyr=\BK_{\text{even}}\BK_{\text{odd}}$, where $\BK_{\text{odd}}$ (respectively, $\BK_{\text{even}}$) denotes the product of the Bender--Knuth involutions with odd (respectively, even) indices. By replacing Bender--Knuth involutions with noninvertible Bender--Knuth toggles in the definition of gyration, we obtain the \dfn{extended gyration} operator $\Gyr\colon \mathfrak S_{n+1}\to \mathfrak S_{n+1}$. One could ask for an analogue of \eqref{eq:ProSort} for extended gyration. As a special case of \cref{thm:main2} below, we will find that 
\begin{equation}\label{eq:GyrSort}
\Gyr^{\left\lceil (n+1)/2\right\rceil}(\mathfrak S_{n+1})=\mathcal L(P)
\end{equation}
and that the quantity $\left\lceil (n+1)/2\right\rceil$ in this statement cannot be decreased. 
In fact, one of our results will generalize \eqref{eq:ProSort} and \eqref{eq:GyrSort} to operators obtained by composing $\tau_1,\ldots,\tau_{n}$ in an arbitrary order. We will also generalize \eqref{eq:EvSort} to any operator obtained by starting with a reduced word for the long element of $\mathfrak S_{n+1}$ and replacing each simple reflection by its corresponding noninvertible Bender--Knuth toggle. Moreover, we will state and prove these theorems in the more general setting where the symmetric group $\mathfrak S_{n+1}$ is replaced by any finite Coxeter group. Before stating these results, we need to discuss how to define noninvertible Bender--Knuth toggles for general Coxeter groups. 

\subsection{Noninvertible Bender--Knuth Toggles for Coxeter Groups}\label{subsec:intro_Coxeter} 

Let $(W,S)$ be a Coxeter system, and write $S=\{s_i:i\in I\}$, where $I$ is a finite index set. This means that $W$ is a group with a presentation of the form 
\begin{equation}\label{eq:presentation}
W=\langle S:(s_is_{i'})^{m(i,i')}=\mathbbm{1}\text{ for all }i,i'\in I\rangle,
\end{equation} where $m(i, i) = 1$ and $m(i,i')=m(i',i)\in\{2,3,\ldots\}\cup\{\infty\}$ for all distinct $i,i'\in I$. The \dfn{rank} of $(W,S)$ is $|S|=|I|$. We often refer to just the Coxeter group $W$, tacitly assuming that this refers to the Coxeter system $(W,S)$. 

The \dfn{Coxeter graph} of $W$ is the undirected edge-labeled graph $\Gamma_{W}$ with vertex set $I$ in which distinct vertices $i$ and $i'$ are joined by an edge labeled $m(i,i')$ whenever $m(i,i')\geq 3$. It is typical to omit the label ``$3$'' in drawings of Coxeter graphs. For $J\subseteq I$, the corresponding \dfn{(standard) parabolic subgroup} is the subgroup $W_J$ of $W$ generated by $\{s_i:i\in J\}$. The pair $(W_J,\{s_i:i\in J\})$ is a Coxeter system, and $\Gamma_{W_J}$ is the subgraph of $\Gamma_{W}$ induced by $J$.

The \dfn{root space} of $W$ is the vector space $V = \mathbb{R}^I$. Let $\{\alpha_i:i\in I\}$ be the standard basis of $V$; we call the elements of this basis the \dfn{simple roots}. For each pair $\{i,i'\}$ of indices in $I$ such that $m(i,i')=\infty$, choose a real number $\mu_{\{i,i'\}}\geq 1$. Define a symmetric bilinear form $B\colon V\times V\to\mathbb{R}$ by setting
\begin{equation}\label{eq:bilinear}
B(\alpha_i,\alpha_{i'})=\begin{cases} -\cos(\pi/m(i,i')) & \mbox{if $m(i,i')<\infty$;} \\ -\mu_{\{i,i'\}} & \mbox{if $m(i,i')=\infty$} \end{cases}
\end{equation} for all $i,i'\in I$ and extending bilinearly; we say $B$ is \dfn{induced} by $W$.\footnote{With this definition, $B$ is not uniquely determined by $W$ since we can freely choose the number $\mu_{\{i,i'\}}\geq 1$ whenever $m(i,i')=\infty$. However, all of our results and arguments outside of \cref{subsec:folding} are independent of these choices. Thus, unless explicitly stated otherwise, we will assume $\mu_{\{i,i'\}}=1$ whenever $m(i,i')=\infty$ so that we may refer to $B$ as ``the'' bilinear form induced by $W$.} The \dfn{standard geometric representation} of $W$ is the faithful representation $\rho\colon W\to\mathrm{GL}(V)$ defined by $\rho(s_i)\beta=\beta-2 B(\beta,\alpha_i)\alpha_i$. We often write $w \beta$ to mean $\rho(w) \beta$. The set $\Phi=\{w\alpha_i:w\in W,\,\, i\in I\}$ is called the \dfn{root system} of $W$; its elements are called \dfn{roots}. We say $\beta \in \Phi$ is \dfn{positive} if it is a nonnegative linear combination of the simple roots. We say $\beta$ is \dfn{negative} if $-\beta$ is positive. Let $\Phi^+$ and $\Phi^-$ denote the set of positive roots and the set of negative roots, respectively. We have $\Phi^- = -\Phi^+$ and $\Phi=\Phi^+\sqcup\Phi^-$. We say that a subset $R\subseteq\Phi$ is \dfn{closed} if whenever $\beta,\beta'\in R$ and $a,a'\geq 0$ satisfy $a\beta+a'\beta'\in\Phi$, we have $a\beta+a'\beta'\in R$. We say that $R \subseteq \Phi$ is \dfn{antisymmetric} if $R \cap (-R) = \varnothing$.

Let $V^*$ be the dual space of $V$. For $\beta\in\Phi$, consider the hyperplane 
\[\HH_\beta = \{f \in V^* : f(\beta) = 0\}\] in $V^*$. The \dfn{Coxeter arrangement} of $W$ is the collection $\mathcal H_W=\{\HH_\beta:\beta\in\Phi\}$. A \dfn{region} of $\mathcal H_W$ is the closure of a connected component of $V^* \setminus\bigcup_{\beta\in\Phi}\HH_\beta$. Consider the region \[{\BB}=\{f\in V^*:f(\alpha_i)\geq 0\text{ for all }i\in I\};\] the set of bounding walls of ${\BB}$ is $\{\HH_{\alpha_i}: i\in I\}$. There is a right action of $W$ on $V^*$ determined by the condition that $(fw)(\beta)=f(w\beta)$ for all $w\in W$, $\beta\in V$, and $f\in V^*$; this induces an action of $W$ on the set of regions of $\mathcal H_W$. The set ${\BB}W\subseteq V^*$ is called the \dfn{Tits cone}. The elements of $W$ are in bijection with the regions of $\mathcal H_W$ that are contained in the Tits cone via the bijection $u \mapsto \BB u$. The \dfn{positive projectivization} of the Tits cone is the quotient \[\mathbb P(\BB W)=(\BB W\setminus\{0\})/\mathbb R_{>0}.\] We may view each $\HH_\beta$ as a hyperplane in $\mathbb P(\BB W)$ and view $\mathcal H_W$ as a hyperplane arrangement in $\mathbb P({\BB W})$. Two distinct hyperplanes $\HH_\beta$ and $\HH_{\beta'}$ intersect in $\mathbb P(\BB W)$ if and only if $|B(\beta, \beta')| < 1$.  It is possible to equip $\mathbb P(\BB W)$ with a (not necessarily Riemannian) metric such that for all distinct $\beta, \beta' \in \Phi$ with $|B(\beta, \beta')| < 1$, the hyperplanes $\HH_\beta$ and $\HH_{\beta'}$ intersect at an angle $\theta$ that satisfies $\lvert\cos(\theta)\rvert = \lvert B(\beta, \beta')\rvert$.

For $\beta\in\Phi$, define the corresponding \dfn{half-spaces} \[H_\beta^+=\{w\in W:w\beta\in\Phi^+\}\quad\text{and}\quad H_\beta^-=\{w\in W:w\beta\in\Phi^-\}.\] Observe that these two sets are complements of one another in $W$ and that $H_\beta^- = H_{-\beta}^+$. Note that $\BB H_\beta^+$ is the part of the Tits cone that lies on one side of $\HH_\beta$, while $\BB H_\beta^-$ is the part of the Tits cone that lies on the other side.

For nonempty $\LL\subseteq W$, define \[\RR(\LL)=\{\beta\in\Phi:\LL\subseteq H_\beta^+\}.\] Note that $\RR(\LL)$ is closed and antisymmetric. The \dfn{convex hull} of $\LL$ is the intersection \[\bigcap_{\beta\in \RR(\LL)}H_\beta^+\] of all of the half-spaces that contain $\LL$ (if $\RR(\LL)=\varnothing$, this intersection is interpreted as all of $W$). We say $\LL$ is \dfn{convex} if it is equal to its own convex hull. In other words, a set is convex if it is an intersection of half-spaces.

\begin{remark}\label{rem:Cayley}
    Let $\Cay(W,S)$ be the left Cayley graph of $W$ generated by $S$. It is known (see \cite[Theorem~2.19]{Tits}) that a set $\LL\subseteq W$ is convex if and only if every minimum-length path in $\Cay(W,S)$ between two elements of $\LL$ has all of its vertices in $\LL$. This implies that the convex hull of a finite set is finite. 
\end{remark}

\begin{definition}\label{def:toggles}
Fix a nonempty convex subset $\LL$ of $W$. For each $i\in I$, define the \dfn{noninvertible Bender--Knuth toggle} $\tau_i\colon W\to W$ by \[\tau_i(u)=\begin{cases} u & \mbox{if $u^{-1}\alpha_i\in\RR(\LL)$;} \\ s_iu & \mbox{if $u^{-1}\alpha_i\not\in\RR(\LL)$.} \end{cases}
\]
\end{definition}
The preceding definition does not rely on the geometry of the Coxeter arrangement, but it has the following geometric interpretation. If we identify each element $w\in W$ with the region ${\BB w}$ of the Coxeter arrangement, then $\HH_{u^{-1}\alpha_i}$ is the unique hyperplane separating $u$ from $s_iu$. The toggle $\tau_i$ is defined so that $\tau_i(u)=u$ if and only if $\LL\cup\{u\}$ lies entirely on one side of $\HH_{u^{-1}\alpha_i}$. Equivalently, $\tau_i(u)=s_iu$ if and only if there is an element of $\LL$ on the same side of $\HH_{u^{-1}\alpha_i}$ as $s_iu$. Thus, moving from $u$ to $\tau_i(u)$ cannot take us ``strictly away from'' the convex set $\LL$.\footnote{One can also define the noninvertible Bender--Knuth toggles using the set of reflections rather than the root system (using \eqref{eq:inversions_roots} below). Although this perspective would make many of our proofs more complicated, it will be somewhat helpful in \cref{sec:ancient}.}  
 
\begin{example}\label{ex:partial-order}
Let $W=\mathfrak S_{n+1}$, and let $S = \{s_i : i \in [n]\}$, where $s_i=(i\,\,i+1)$. Then $(\mathfrak{S}_{n+1}, S)$ is the Coxeter system of type $A_n$. We can identify $V$ with the vector space \[\{(x_1,\ldots,x_{n+1})\in\mathbb R^{n+1}:x_1+\cdots+x_{n+1}=0\}\] via the isomorphism that sends the simple root $\alpha_i$ to $e_{i} - e_{i+1}$, where $e_j$ denotes the $j$-th standard basis vector of $\mathbb R^{n+1}$. We have \[\Phi^+=\{e_i-e_j:1\leq i<j\leq n+1\}\quad\text{and}\quad\Phi^-=\{e_j-e_i:1\leq i<j\leq n+1\}.\] The action of $W$ on $\Phi$ is given by $w(e_i - e_j) = e_{w(i)} - e_{w(j)}$. Hence, for any root $e_i - e_j \in \Phi$, we have \[H_{e_i-e_j}^+=\{w \in \mathfrak S_{n+1} : w(i) < w(j)\}.\]

Let $\LL \subseteq \mathfrak{S}_{n+1}$ be a nonempty set. We may define a relation $\leq_{\LL}$ on $[n+1]$ by declaring that $i \leq_{\LL} j$ if and only if either $i = j$ or $e_i - e_j \in \RR(\LL)$. Since $\RR(\LL)$ is closed and antisymmetric, the relation $\leq_{\LL}$ is a partial order. By the above discussion, the convex hull of $\LL$ is precisely the set of linear extensions of $([n+1],\leq_{\LL})$. In this way, nonempty convex subsets of $\mathfrak S_{n+1}$ correspond bijectively to partial orders on $[n+1]$, and the root-theoretic definition of $\tau_i$ in \Cref{def:toggles} agrees with the poset-theoretic definition in \eqref{eq:toggle_def_posets}. 
\end{example}

\subsection{Bender--Knuth Billiards}\label{subsec:billiards}
Fix a nonempty convex subset $\LL$ of a Coxeter group $W$, and use this convex set to define the noninvertible Bender--Knuth toggles $\tau_i$ for $i\in I$, as in \Cref{def:toggles}. 

Let $i_1,\ldots,i_n$ be an ordering of the elements of $I$. It is convenient to take $i_{k+n}=i_k$ for every positive integer $k$, thereby obtaining an infinite periodic sequence $i_1,i_2,i_3,\ldots$. If we start with an element $u_0\in W$ and apply the noninvertible Bender--Knuth toggles in the order $\tau_{i_1},\tau_{i_2},\tau_{i_3},\ldots$, we obtain a sequence $u_0,u_1,u_2,\ldots$, where $u_j=\tau_{i_j}(u_{j-1})$ for each $j\geq 1$. We call this sequence a \dfn{(Bender--Knuth) billiards trajectory}. We can draw the billiards trajectory in $\mathbb P({\BB W})$. In each of \Cref{fig:affineS3,fig:spherical,fig:affineS3_infinite,fig:hyperbolic}, we have drawn a billiards trajectory using a sequence of arrows, where each arrow points from $\BB u_{j-1}$ to $\BB u_j$ for some $j$.

We say a hyperplane $\HH_\beta$ is a \dfn{one-way mirror} if $\beta\in (\RR(\LL)\cup(-\RR(\LL))$; otherwise, we say that $\HH_\beta$ is a \dfn{window}. Somewhat informally, we imagine the billiards trajectory is the trajectory of a beam of light that starts at $u_0$. Each window allows the light to pass through it in either direction. By contrast, if $\beta\in\RR(\LL)$, then $\HH_\beta$ is a one-way mirror that allows the light to pass from $H_\beta^-$ to $H_\beta^+$ but reflects light that tries to pass in the other direction.  

\begin{example}
Let $W=\mathfrak S_4$, and consider the convex set \[\LL=\{\mathbbm{1},s_1,s_3,s_1s_3,s_2s_1s_3\}=\{(1,2,3,4),(2,1,3,4),(1,2,4,3),(2,1,4,3),(3,1,4,2)\},\] where in the last expression for $\LL$, we have represented each permutation $w\in \LL$ in its one-line notation $(w(1),w(2),w(3),w(4))$. Preserve the notation from \cref{ex:partial-order}. The Hasse diagram of $([4],\leq_{\LL})$ is \[\begin{array}{l}\includegraphics[height=1.533cm]{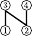}\end{array}.\] The Tits cone ${\BB W}$ is the entire $3$-dimensional space $V^*$, so ${\mathbb P({\BB W}) = ({\BB W} \setminus \{0\}) / \mathbb{R}_{> 0}}$ is a sphere. \cref{fig:spherical} depicts a stereographic projection of this sphere to a plane. Each of the $6$ hyperplanes in the Coxeter arrangement $\mathcal H_{\mathfrak S_4}=\{\HH_{e_i-e_j}:1\leq i<j\leq 4\}$ is drawn as a circle. Each of the $24$ regions of $\mathcal H_{\mathfrak S_4}$ is represented by a connected component of the space formed by removing the $6$ projected circles from the plane. Each region corresponds to a permutation $w \in \mathfrak S_4$, which we represent by the labeling
\[\begin{array}{l}\includegraphics[height=1.544cm]{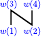}\end{array}\]
of the poset $([4],\leq_\LL)$ (for elements of $\LL$, the numbers are drawn in cyan). The regions labeled by elements of $\LL$ are shaded black. The windows---represented by thin gray circles---are $\HH_{e_1-e_2}$, $\HH_{e_3-e_4}$, and $\HH_{e_1-e_4}$, while the one-way mirrors---represented by circles colored yellow and red---are $\HH_{e_2-e_3}$, $\HH_{e_1-e_3}$, and $\HH_{e_2-e_4}$. The windows and one-way mirrors correspond (respectively) to incomparable and comparable pairs of elements of the poset $([4],\leq_\LL)$.  Fix the ordering $1,2,3$ of the index set $I$. The arrows in \cref{fig:spherical} indicate the billiards trajectory $u_0,u_1,u_2,\ldots$ that starts at the permutation ${u_0=2413}$ (marked with a brown dot). 
\end{example}

\begin{figure}
 \begin{center}{\includegraphics[width=\linewidth]{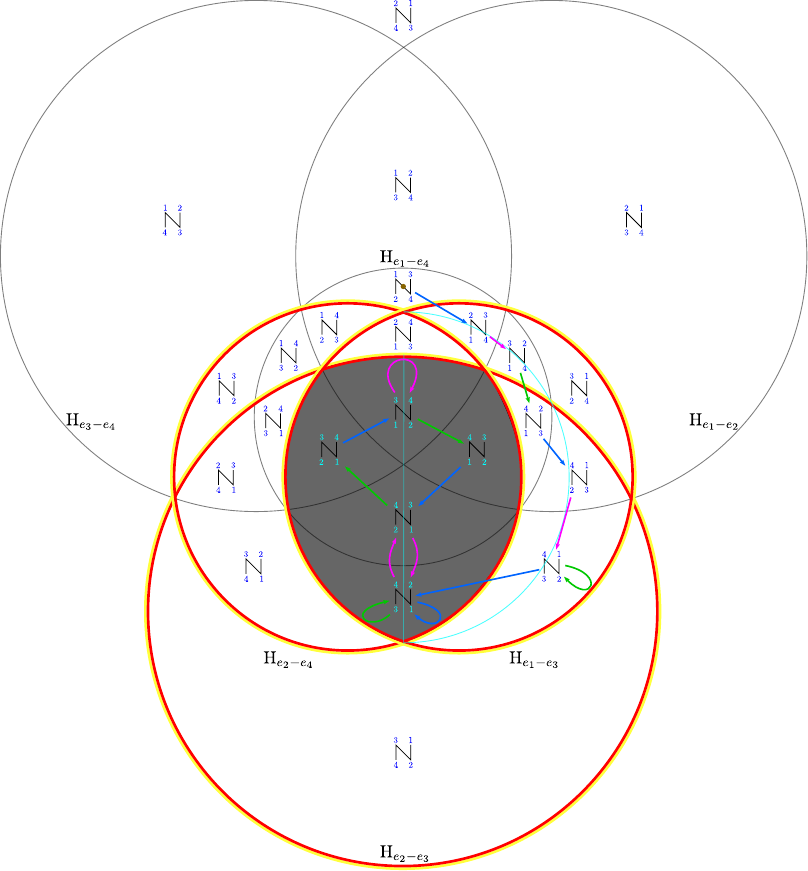}}
  \end{center}
\caption{
A stereographic projection of the Tits cone and Coxeter arrangement of $\mathfrak S_4$. The black convex set $\LL$ turns each hyperplane in the Coxeter arrangement into either a transparent window (indicated by a thin gray circle) or a one-way mirror (indicated by a circle colored yellow and red). Regions correspond to permutations in $\mathfrak S_4$, which are represented as labelings of a $4$-element ($\mathsf{N}$-shaped) poset. Arrows indicate the billiards trajectory determined by the starting permutation $u_0=2413$ (marked with a brown dot) and the ordering ${\color{ArrowBlue}1},{\color{Traj5}2},{\color{Traj4}3}$ of $I$. This billiards trajectory follows a (thin cyan) light beam. (The diagram is not to scale, so angles have been distorted.)
}\label{fig:spherical}
\end{figure}

\begin{figure}[ht]
  \begin{center}{\includegraphics[height=3.507cm]{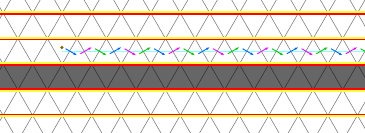}}
  \end{center}
\caption{If we choose our convex set $\LL$ to be an infinite strip (shown in black) in $\widetilde A_2$, then the billiards trajectory can ``escape to infinity'' without ever reflecting off of a mirror.}\label{fig:affineS3_infinite}
\end{figure}

\begin{figure}[ht]
 \begin{center}{\includegraphics[width=\linewidth]{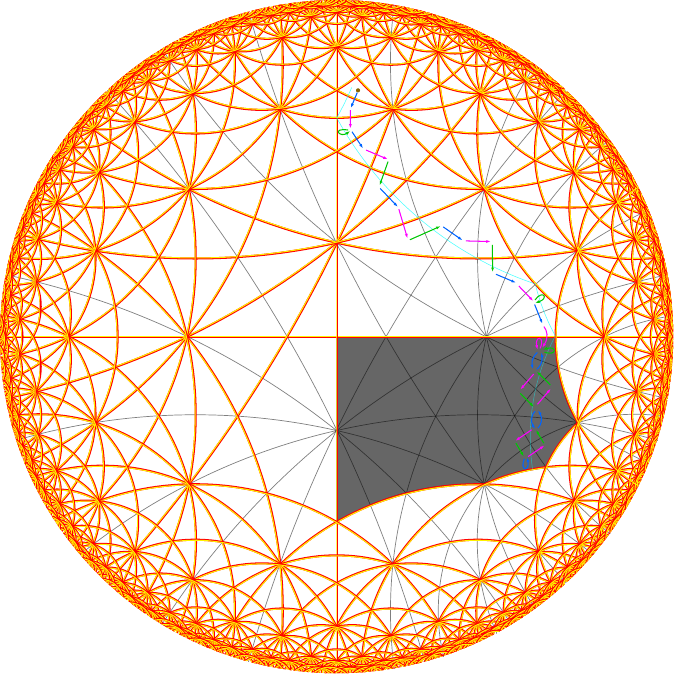}}
  \end{center}
\caption{
The Tits cone and Coxeter arrangement of the Coxeter group with Coxeter graph \!\!$\begin{array}{l} \includegraphics[height=0.373cm]{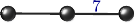}\end{array}$\!. We have passed to the positive projectivization $\mathbb{P}({\BB W})$, which is a hyperbolic plane, and then drawn the hyperbolic plane using the Poincar\'e disk model.
The black convex set $\LL$ turns each hyperplane in the Coxeter arrangement into either a transparent window (indicated by a thin gray line) or a one-way mirror (indicated by a line colored yellow and red). Arrows indicate a billiards trajectory that starts at the region marked with the brown dot. The billiards trajectory follows the (thin cyan) light beam. 
}\label{fig:hyperbolic}
\end{figure}

\begin{remark}
In each of \cref{fig:affineS3,fig:spherical,fig:affineS3_infinite,fig:hyperbolic}, we have drawn a thin, cyan, piecewise linear path that reflects off of the one-way mirrors according to the usual billiards rule (i.e., that the angle of incidence equals the angle of reflection), which has been studied vigorously and fruitfully in dynamics \cite{Hasselblatt, Kozlov, McMullen, Tabachnikov}. We call this cyan path a \dfn{light beam}. Our Bender--Knuth billiards trajectory appears to follow alongside the light beam, so it can be seen as a discretization of the light beam. It is natural to ask under what conditions a Bender--Knuth billiards trajectory will discretize a light beam (and how to even rigorously formalize this idea). We will discuss this in \cref{sec:linear}, where we will see that the Bender--Knuth billiards trajectory discretizes a light beam whenever $|I|\leq 3$ or the Coxeter element $c=s_{i_n}\cdots s_{i_1}$ is \emph{bipartite}.  
\end{remark}

Given a word $\mathsf{w}=i_M\cdots i_1$ over the alphabet $I$, it will be convenient to define the composition of toggles \[\tau_{\mathsf{w}}=\tau_{i_M}\cdots\tau_{i_1}.\] A \dfn{(standard) Coxeter element} of $W$ is an element obtained by multiplying the elements of $S$ together in some order so that each simple reflection appears exactly once in the product. Such an ordering $i_1,\ldots,i_n$ of $I$ gives rise to a reduced word $\sfc=i_n\cdots i_1$ representing a Coxeter element $c=s_{i_n}\cdots s_{i_1}$. Let \[\Pro_c=\tau_{\sfc}=\tau_{i_n}\cdots\tau_{i_1}.\] Because $\tau_i\tau_{i'}=\tau_{i'}\tau_i$ whenever $m(i,i')=2$, it follows from Matsumoto's theorem (\cref{thm:matsumoto} below) that $\Pro_c$ depends only on $c$ and not on the particular reduced word $\sfc$. Each operator $\tau_i$ restricts to a bijection from $\LL$ to itself; thus, $\Pro_c$ also restricts to a bijection from $\LL$ to itself. An element $u\in W$ is called a \dfn{periodic point} of $\Pro_c$ if there exists an integer $K\geq 1$ such that $\Pro_c^K(u)=u$.

Because noninvertible Bender--Knuth toggles never cause the billiards trajectory to move ``strictly away from'' $\LL$, one might expect (as we did for a while) that every periodic billiards trajectory must lie in $\LL$. As we will see, this is not always the case. Let us establish some terminology to describe the various scenarios that can occur.

\begin{definition}\label{def:futuristic}Let $\LL \subseteq W$ be a nonempty convex set, and define the noninvertible Bender--Knuth toggles with respect to $\LL$. Let $c$ be a Coxeter element of $W$.  
\begin{itemize}
\item We say $\LL$ is \dfn{heavy} with respect to $c$ if every periodic point of $\Pro_c$ lies in $\LL$. 
\item We say $\LL$ is \dfn{superheavy} with respect to $c$ if for every $u\in W$, there exists a nonnegative integer $K$ such that $\Pro_c^K(u)\in \LL$.\footnote{We imagine that $\LL$ is a black hole that always pulls light into it. It would be more consistent with astrophysical terminology to use the words ``massive'' and ``supermassive'' rather than ``heavy'' and ``superheavy''. However, doing so would prevent us from writing \cref{footnote:bttf}.}
\end{itemize}
\end{definition}
Note that if $\LL$ is superheavy with respect to $c$, then it is heavy with respect to $c$. The two definitions are equivalent if $W$ is finite.
\begin{remark}\label{rem:re-def}
One could rephrase \cref{def:futuristic} in the language of billiards trajectories instead of the operators $\Pro_c$. A convex set $\LL \subseteq W$ is heavy with respect to a Coxeter element $s_{i_n} \cdots s_{i_1}$ if and only if every periodic billiards trajectory (defined with respect to the convex set $\LL$ and the ordering $i_1, \ldots, i_n$ of $I$) lies in $\LL$. Similarly, $\LL$ is superheavy with respect to $s_{i_n} \cdots s_{i_1}$ if and only if every billiards trajectory eventually reaches $\LL$.
\end{remark}

As an illustration of the preceding definitions, \cref{thm:affineAC} (below) will tell us that every nonempty convex subset of $\widetilde A_2$ is heavy with respect to the Coxeter element $s_2s_1s_0$. This means that for every nonempty convex set $\LL \subseteq \widetilde A_2$ and every starting element $u_0 \in W \setminus \LL$, the billiards trajectory \[u_0, \tau_0(u_0), \tau_1\tau_0(u_0), \tau_2\tau_1\tau_0(u_0), \tau_0\tau_2\tau_1\tau_0(u_0), \ldots\] is not periodic. This does not necessarily imply that the billiards trajectory eventually reaches the set $\LL$. For example, \cref{fig:affineS3_infinite} depicts a billiards trajectory in $\widetilde A_2$ that does not reach $\LL$. Hence, the convex set depicted in \cref{fig:affineS3_infinite} is heavy but not superheavy with respect to $s_2s_1s_0$. 

\begin{restatable}{proposition}{propfinite}\label{prop:finite-to-infinite}
    Let $W$ be a Coxeter group, and let $c$ be a Coxeter element of $W$. The following are equivalent.
    \begin{enumerate}[(i)]
        \item\label{item:everything-heavy} Every nonempty convex set $\LL \subseteq W$ is heavy with respect to $c$. 
        \item\label{item:finite-heavy} Every nonempty finite convex set $\LL \subseteq W$ is heavy with respect to $c$.
        \item\label{item:sorting} Every nonempty finite convex set $\LL \subseteq W$ is superheavy with respect to $c$.
    \end{enumerate}
\end{restatable}
Note that by the paragraph preceding \cref{prop:finite-to-infinite}, the condition that every nonempty convex set $\LL \subseteq W$ is superheavy with respect to $c$ is strictly stronger than the conditions \eqref{item:everything-heavy}, \eqref{item:finite-heavy}, and \eqref{item:sorting} of \cref{prop:finite-to-infinite}.

The preceding proposition naturally motivates the following definitions, which are central to the present work.
\begin{definition}\ 
\begin{itemize}
\item We say a Coxeter element $c \in W$ is \dfn{futuristic} if every nonempty convex set $\LL \subseteq W$ is heavy with respect to $c$ (i.e., $c$ satisfies the three equivalent conditions of \cref{prop:finite-to-infinite}).
\item We say a Coxeter element $c \in W$ is \dfn{superfuturistic} if every nonempty convex set $\LL \subseteq W$ is superheavy with respect to $c$.
\item We say the Coxeter group $W$ is \dfn{futuristic} if every Coxeter element of $W$ is futuristic \linebreak(equivalently, if every nonempty convex subset of $W$ is heavy with respect to every Coxeter element of $W$).\footnote{\label{footnote:bttf}In the movie \emph{Back to the Future}, Doc asks Marty, ``Why are things so heavy in the future?''}
\item We say the Coxeter group $W$ is \dfn{superfuturistic} if every Coxeter element of $W$ is superfuturistic (equivalently, if every nonempty convex subset of $W$ is superheavy with respect to every Coxeter element of $W$).
\item We say the Coxeter group $W$ is \dfn{ancient} if no Coxeter element of $W$ is futuristic.  
\end{itemize}
\end{definition}

\begin{example}
    Let $W = \mathfrak{S}_{n+1} = A_n$, and let $c = s_n \cdots s_1$. By \cref{ex:partial-order}, each nonempty convex set $\LL \subseteq W$ has an associated partial order $P = ([n+1], \leq_{\LL})$ such that $\LL = \mathcal{L}(P)$. Additionally, the promotion operator $\Pro_c$ is equal to the extended promotion operator $\Pro$ defined in \cref{subsec:intro-posets}. By \eqref{eq:ProSort}, the convex set $\LL$ is superheavy with respect to $c$. As $\LL$ was arbitrary, $c$ is superfuturistic. Likewise, \eqref{eq:GyrSort} implies that the Coxeter element $(s_2s_4s_6\cdots)(s_1s_3s_5\cdots)$ is superfuturistic. We will vastly generalize these observations in \cref{thm:finite_futuristic}, which states that every Coxeter element of every finite Coxeter group is superfuturistic.
\end{example}

\begin{example}\label{ex:periodic-billiards-path}
The Coxeter group $\widetilde{D}_4$ has Coxeter graph 
\[\begin{array}{l}
\includegraphics[height=1.56cm]{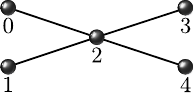}
\end{array}.\] 
Fix the ordering $0,1,2,3,4$ of $I$, and let $c=s_4s_3s_2s_1s_0$ be the corresponding Coxeter element. Let $\LL$ be the convex hull of $\{s_2,s_0s_2s_0,s_1s_2s_1,s_3s_2s_3,s_4s_2s_4\}$. It turns out that the billiards trajectory starting at $u_0=\id$ begins with 
\begin{alignat*}{6}
\mathbbm{1}&\xrightarrow{\tau_0}{}& \centcell{s_0} 
 &\xrightarrow{\tau_1}{}&\centcell{s_1s_0}&\xrightarrow{\tau_2}{}&\centcell{s_1s_0}&\xrightarrow{\tau_3}{}&\centcell{s_3s_1s_0}&\xrightarrow{\tau_4}{}&\centcell{s_4s_3s_1s_0}& \\ 
&\xrightarrow{\tau_0}{}&\centcell{s_4s_3s_1} &\xrightarrow{\tau_1}{}&\centcell{s_4s_3}&\xrightarrow{\tau_2}{}&\centcell{s_4s_3}&\xrightarrow{\tau_3}{}&\centcell{s_4}&\xrightarrow{\tau_4}{}&\centcell{\mathbbm{1}}&.
\end{alignat*} This shows that $\Pro_c^2(\id)=\id$, even though $\id\not\in\LL$. It follows that $\LL$ is not heavy with respect to $c$, so $c$ is not futuristic. In fact, we will see later (\cref{thm:ancient}) that $\widetilde D_4$ is ancient. 
\end{example}

\subsection{Main Results}
Our overarching goal is to understand which Coxeter groups are futuristic and which are ancient.
We begin with the following propositions, which state that the class of futuristic Coxeter groups and the class of non-ancient Coxeter groups are in some sense \emph{hereditary}. 

\begin{samepage}
\begin{proposition}\label{prop:hereditary} Let $W_J$ be a standard parabolic subgroup of a Coxeter group $W$. 
\begin{enumerate}[(i)]
\item If $W$ is futuristic, then $W_J$ is futuristic.  
\item If $W_J$ is ancient, then $W$ is ancient. 
\end{enumerate}
\end{proposition}
\end{samepage}

Let us say a Coxeter group is \dfn{minimally non-futuristic} if it is not futuristic and all of its proper standard parabolic subgroups are futuristic. Likewise, say a Coxeter group is \dfn{minimally ancient} if it is ancient and none of its proper standard parabolic subgroups are ancient. \cref{prop:hereditary} reduces the problem of characterizing futuristic Coxeter groups and ancient Coxeter groups to the problem of characterizing minimally non-futuristic Coxeter groups and minimally ancient Coxeter groups.

Another tool that we can use to determine whether or not a Coxeter group is futuristic comes from the technique of \emph{folding}, defined in \cref{subsec:folding}. 

\begin{proposition}\label{prop:folding}
Let $W^{\fold}$ be a folding of a Coxeter group $W$. If $W$ is futuristic, then $W^{\fold}$ is futuristic.  
\end{proposition}

Although we will not obtain a complete characterization of futuristic Coxeter groups, we will prove that several notable Coxeter groups are futuristic. 

\begin{theorem}\label{thm:finite_futuristic}
Every finite Coxeter group is superfuturistic.
\end{theorem}

\begin{theorem}\label{thm:affineAC}
Every affine Coxeter group of type $\widetilde A$, $\widetilde C$, or $\widetilde G_2$ is futuristic.
\end{theorem}

\begin{theorem}\label{thm:right-angled}
Every right-angled Coxeter group is superfuturistic. 
\end{theorem}

\begin{theorem}\label{thm:complete}
Every Coxeter group with a complete Coxeter graph is futuristic. 
\end{theorem}

\begin{theorem}\label{thm:rank_3}
Every Coxeter group of rank at most $3$ is futuristic. 
\end{theorem}

The next theorem exhibits a vast (infinite) collection of ancient Coxeter groups. 

\begin{theorem}\label{thm:ancient}
Fix integers $a,a',b,b'\geq 3$. If $W$ is a Coxeter group whose Coxeter graph has one of the following Coxeter graphs as an induced subgraph, then $W$ is ancient: 
\[\includegraphics[height=1.273cm]{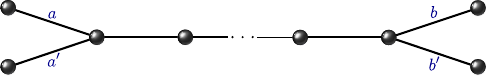}\,\,,\] 

\[\includegraphics[height=1.281cm]{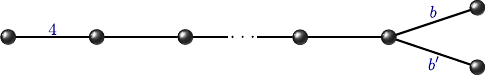}\,\,,\] 

\[\includegraphics[height=1.281cm]{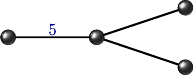}\,\, ,\]

\[\includegraphics[height=1.273cm]{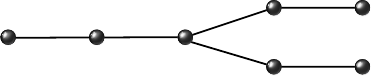}\,\, ,\]

\[\includegraphics[height=1.273cm]{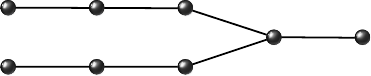}\,\, ,\]

\[\includegraphics[height=1.273cm]{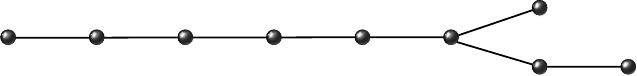}\,\,,\]

\[\includegraphics[height=0.373cm]{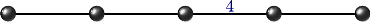}\,\, .\]
(The number of vertices in the first graph in this list can be any integer that is at least $5$, while the number of vertices in the second graph can be any integer that is at least $4$.)
\end{theorem}

\cref{thm:affineAC,thm:ancient} allow us to completely describe which affine Coxeter groups are futuristic and which are ancient. By \cref{rem:irreducible_futuristic}, it suffices to consider the irreducible affine Coxeter groups. The irreducible affine Coxeter groups not listed in \cref{thm:affineAC} are those of types $\widetilde B$, $\widetilde D$, $\widetilde E_6$, $\widetilde E_7$, $\widetilde E_8$, and $\widetilde F_4$, and these all appear in the list in \cref{thm:ancient}. 

\begin{corollary}\label{cor:affineBD}
The irreducible affine Coxeter groups of types $\widetilde A$, $\widetilde C$, and $\widetilde G_2$ are futuristic; all other irreducible affine Coxeter groups are ancient.     
\end{corollary}


A Coxeter graph is called \dfn{simply laced} if all of its edge labels are $3$. It is worth noting that \cref{prop:hereditary,thm:finite_futuristic,thm:ancient} allow us to determine the futuristicity/ancientness of all Coxeter groups whose Coxeter graphs are simply laced trees. 

\begin{corollary}
Let $W$ be a Coxeter group whose Coxeter graph is a simply laced tree. If $W$ is finite, then it is futuristic; otherwise, it is ancient.  
\end{corollary}

\subsection{Finite Coxeter Groups}

Let us now assume that $W$ is a finite Coxeter group. \Cref{thm:finite_futuristic} tells us that $W$ is superfuturistic; this implies that for every Coxeter element $c$ of $W$, there exists a nonnegative integer $K$ satisfying $\Pro_c^K(W)=\LL$ for every nonempty convex subset $\LL$ of $W$.  Our proof of \Cref{thm:finite_futuristic} will actually determine the smallest such integer $K$ (in terms of $c$).

Because $W$ is finite, it has a unique element of maximum length; this element is called the \dfn{long element} and is denoted by $\wo$. The following result generalizes \eqref{eq:EvSort} and is new even in type~$A$. 

\begin{theorem}\label{thm:main1}
If $W$ is finite and $i_N\cdots i_1$ is a reduced word for $\wo$, then \[\tau_{i_N}\cdots\tau_{i_1}(W)=\LL.\] 
\end{theorem}

\Cref{thm:main1} is tight in the sense that there exists a choice of $\LL$ such that \[\tau_{i_{N-1}}\cdots\tau_{i_{1}}(W)\neq\LL.\] For example, this is the case if $\LL=\{\wo\}$, because then $\tau_{i_{N-1}}\cdots\tau_{i_1}(\id)=s_{i_{N-1}}\cdots s_{i_1}\neq\wo$ (in fact, the noninvertible Bender--Knuth toggles generate the 0-Hecke monoid of $W$ \cite{Hivert,Kenney} in this particular case). 

\begin{remark}\label{rem:commute}
If $i,i'\in I$ satisfy $m(i,i')=2$, then $\tau_i$ and $\tau_{i'}$ commute. What makes \Cref{thm:main1} nontrivial is the fact that the noninvertible Bender--Knuth toggles do not (in general) satisfy the other braid relations of $W$. For example, when $W=\mathfrak S_{n+1}$ and $i\in[n]$, the operators $\tau_i\tau_{i+1}\tau_i$ and $\tau_{i+1}\tau_i\tau_{i+1}$ are generally not equal. 
\end{remark}

Let $\sfc$ be a reduced word for a Coxeter element $c$ of $W$. For $j\geq 1$, let $\sfc^j$ denote the concatenation of $j$ copies of $\sfc$. Let $\MM(c)$ be the smallest positive integer $k$ such that $\sfc^k$ contains a reduced word for $\wo$ as a subword. (The number $\MM(c)$ depends only on $c$, not on the particular reduced word $\sfc$.) In \cref{sec:finite}, we will describe how to compute $\MM(c)$ explicitly using diagrams called \emph{combinatorial AR quivers}. 
The next result, which we obtain as a corollary of \cref{thm:main1}, generalizes \eqref{eq:ProSort}; as discussed in \Cref{exam:bipartite}, it also implies the new result \eqref{eq:GyrSort}. 

\begin{corollary}\label{thm:main2}
If $c$ is a Coxeter element of a finite Coxeter group $W$, then \[\Pro_c^{\MM(c)}(W)=\LL.\] 
\end{corollary} 

\Cref{thm:main2} is tight in the sense that there exists a choice of $\LL$ such that $\Pro_c^{\MM(c)-1}(W)\neq\LL$. For example, this is the case if $\LL=\{\wo\}$.  

\begin{example}\label{exam:bipartite}
If $W$ is a (not necessarily finite) Coxeter group whose Coxeter graph is bipartite, then we can consider a bipartition $I=Q\sqcup Q'$ and write $c_Q=\prod_{i\in Q}s_i$ and $c_{Q'}=\prod_{i\in Q'}s_i$. Then $c_{Q'}c_Q$ is called a \dfn{bipartite Coxeter element} of $W$. 

Now, assume $W$ is finite and irreducible. The \dfn{Coxeter number} of $W$ is the quantity $h=|\Phi|/|S|$. It is known that $\Gamma_{W}$ must be bipartite, so it has a bipartite Coxeter element $c=c_{Q'}c_{Q}$. It is known \cite[$\S$V.6~Exercise~2]{Bourbaki} that $\MM(c_{Q'}c_Q)=\left\lceil h/2\right\rceil$, so \Cref{thm:main2} tells us that $\Pro_{c_{Q'}c_Q}^{\left\lceil h/2\right\rceil}(W)=\LL$. 

When $W=\mathfrak S_{n+1}$ and $S=\{s_1,\ldots,s_n\}$, where $s_i=(i\,\, i+1)$, the Coxeter graph $\Gamma_{W}$ is a path in which $i$ is adjacent to $i+1$ for each $i\in[n-1]$. In this case, we can take $Q$ (respectively, $Q'$) to be the set of odd-indexed (respectively, even-indexed) simple reflections. Then $\Pro_{c_{Q'}c_Q}=\Gyr$. The Coxeter number of $\mathfrak S_{n+1}$ is $n+1$, so we obtain the identity \eqref{eq:GyrSort}. 
\end{example}

\begin{remark}
The proofs of \eqref{eq:ProSort} and \eqref{eq:EvSort} from \cite{DefantPromotionSorting} are quite specific to the extended promotion and extended evacuation operators (they rely on the notion of a \emph{promotion chain}), and they do not generalize to other operators such as extended gyration. Thus, even in type~$A$, our proofs of \Cref{thm:main1,thm:main2} are new. Our proofs are also type-uniform. 
\end{remark}

\subsection{Futuristic Directions} 
In \cref{sec:future}, we will discuss several ideas for future research. Let us briefly highlight some of them here. 

\begin{itemize}
\item We will define a Coxeter group to be \dfn{contemporary} if it is neither futuristic nor ancient, and we will ask whether or not such Coxeter groups exist. 
\item We will pose the problem of determining the periodic points of $\Pro_c$ when $c$ is not futuristic. 
\item We will consider the algorithmic question of deciding whether a Coxeter element is futuristic. 
\item We will contemplate billiards trajectories arising from infinite sequences of elements of $I$ that might not be obtained by repeating some fixed ordering of $I$. 
\item We will suggest a more general definition of noninvertible Bender--Knuth toggles in which the set $\RR(\LL)$ is replaced by an arbitrary closed set of roots $\RR$. When $\RR=\Phi^-$, this yields the $0$-Hecke monoid of $W$. 
\end{itemize}

\subsection{Outline} 
\Cref{sec:preliminaries} provides additional necessary background on Coxeter groups and root systems. \Cref{sec:basics} establishes some general results about Bender--Knuth billiards, including \cref{prop:finite-to-infinite,prop:hereditary,prop:folding}. In \cref{sec:finite}, we focus on finite Coxeter groups and prove \cref{thm:main1,thm:main2}; note that \cref{thm:main2} implies \cref{thm:finite_futuristic}. \cref{sec:affine} establishes a general result about Bender--Knuth billiards in affine Coxeter groups, which we then apply in order to prove \cref{thm:affineAC}. Our proofs of \cref{thm:right-angled,thm:complete,thm:rank_3} use the theory of \emph{small roots} developed by Brink and Howlett \cite{BrinkHowlett}; in \cref{sec:small} we discuss small roots, introduce the \emph{small-root billiards graph}, and explain how this graph can be used to prove that certain Coxeter groups are futuristic. In \cref{sec:right-angled,sec:complete,sec:rank3}, we apply this method to prove \cref{thm:right-angled,thm:complete,thm:rank_3}. \cref{sec:ancient} is about ancient Coxeter groups; it is in this section that we establish \Cref{thm:ancient}. \cref{sec:linear} discusses how Bender--Knuth billiards relate to classical billiards from dynamics. In \cref{sec:future}, we come back to the future directions briefly mentioned above.  

\section{Preliminaries}\label{sec:preliminaries}
This section discusses relevant background information about Coxeter groups, largely drawing from the standard reference \cite{BjornerBrenti}. We assume that the reader is familiar with the classifications of finite Coxeter groups and affine Coxeter groups and the standard notation used to refer to such groups (\cite[Appendix~A1]{BjornerBrenti}). 

Let $(W,S)$ be a Coxeter system whose set of simple reflections $S = \{s_i : i \in I\}$ is indexed by a finite set $I$. A \dfn{word} is a sequence $\mathsf{w} = i_M\cdots i_1$ of elements of $I$; the integer $M$ is called the \dfn{length} of $\mathsf{w}$. We say the word $\mathsf{w} = i_M\cdots i_1$ \dfn{represents} the product $w = s_{i_M} \cdots s_{i_1}$ of the corresponding simple reflections. A \dfn{reduced word} for an element $w\in W$ is a word that represents $w$ and has minimum possible length. If $i_M\cdots i_1$ is a reduced word for $w$, then we call the expression $s_{i_M}\cdots s_{i_1}$ a \dfn{reduced expression} for $w$. The \dfn{length} of $w$, denoted $\ell(w)$, is defined to be the length of a reduced word for $w$. A \dfn{reflection} is an element of $W$ that is conjugate to a simple reflection. We write $T=\{wsw^{-1}:w\in W,\,\, s\in S\}$ for the set of reflections of $W$. For $w \in W$, the set of \dfn{right inversions} of $w$ is $\Inv(w)=\{t\in T:\ell(wt)<\ell(w)\}$. We have $\left\lvert\Inv(w)\right\rvert=\ell(w)$. 

Given a word $\mathsf w=i_M\cdots i_1$, we write $\tau_{\mathsf w}=\tau_{i_M}\cdots \tau_{i_1}$. 
One can apply a \dfn{commutation move} to $\mathsf w$ by swapping two consecutive letters $i$ and $i'$ that satisfy $m(i,i')=2$. We say that two words $\mathsf{x}$ and $\mathsf{y}$ are \dfn{commutation equivalent}, denoted $\mathsf{x}\equiv\mathsf{y}$, if $\mathsf{y}$ can be obtained from $\mathsf{x}$ via a sequence of commutation moves. Note that two commutation equivalent words represent the same element of $W$. Moreover, if $\x\equiv\y$, then $\tau_\x=\tau_\y$.  

Given symbols $\xi$ and $\xi'$ and a positive integer $d$, let us write 
\begin{equation}\label{eq:mid_notation}
[\xi\mid\xi']_d=\underbrace{\cdots\xi'\xi\xi'}_{d}
\end{equation} for the string of length $d$ that alternates between $\xi$ and $\xi'$ and ends with $\xi'$. For example, we have ${[\xi\mid\xi']_4=\xi\xi'\xi\xi'}$ and ${[\xi\mid\xi']_5=\xi'\xi\xi'\xi\xi'}$. 

For $i,i'\in I$, it follows from the defining presentation of $W$ in \eqref{eq:presentation} that the words \begin{equation}\label{eq:alternating_words}
[i\mid i']_{m(i,i')}\quad \text{and}\quad[i'\mid i]_{m(i,i')}
\end{equation} 
represent the same element of $W$. If $\mathsf{w}$ is a word that contains one of the two alternating words in \eqref{eq:alternating_words} as a consecutive factor, then we can perform a \dfn{braid move} on $\mathsf w$ by replacing that consecutive factor with the other alternating word in \eqref{eq:alternating_words} (if $m(i,i')=2$, then this braid move is a commutation move). If $\mathsf{w}$ contains two consecutive occurrences of the same letter, then we can perform a \dfn{nil move} on $\mathsf w$ by deleting those two consecutive letters. If a word $\mathsf w$ can be obtained from a word $\mathsf v$ via a sequence of braid moves and nil moves, then $\mathsf v$ and $\mathsf w$ represent the same element of $W$. 

\begin{theorem}[{Matsumoto's Theorem \cite[Theorem~3.3.1]{BjornerBrenti}}]\label{thm:matsumoto}
Let $\mathsf v$ and $\mathsf w$ be words that represent the same element of $W$. If $\mathsf w$ is reduced, then $\mathsf w$ can be obtained from $\mathsf v$ via a sequence of braid moves and nil moves. 
\end{theorem}

As in \Cref{subsec:intro_Coxeter}, let $V$ denote the root space of $W$. For $i\in I$, let $\alpha_i$ denote the simple root indexed by $i$. Then $V$ is a $W$-module under the standard geometric representation, and there is a $W$-invariant symmetric bilinear form $B\colon V\times V\to \mathbb R$ induced by $W$. Let $\Phi=\{w\alpha_i:w\in W,\,\, i\in I\}$ be the corresponding root system; let $\Phi^+$ and $\Phi^-$ denote the sets of positive roots and negative roots, respectively. We have $\mathbb{R}\beta \cap \Phi = \{\pm \beta\}$ for each $\beta \in \Phi$.

To each root $\beta = w\alpha_i$, we associate a reflection $r_\beta = w s_i w^{-1} \in T$; this is well defined in the sense that it depends only on $\beta$ and not on the particular choices of $w$ and $\alpha_i$. The map $\beta \mapsto r_\beta$ is a bijection from $\Phi^+$ to $T$. We have $r_{\beta} = r_{-\beta}$ for all $\beta \in \Phi$.

The action of $r_\beta$ on $V$ under the standard geometric representation is given by the formula
\begin{equation}\label{eq:reflection}
r_\beta \gamma=\gamma-2B(\gamma,\beta)\beta.
\end{equation} 
For each reflection $t\in T$, let $\beta_t$ denote the unique positive root such that $r_{\beta_t}=t$. Then, for every $w \in W$, we have
\begin{equation}\label{eq:inversions_roots}
\Inv(w)=\{t\in T:w\beta_t\in\Phi^-\}.
\end{equation}
Equivalently, the half-space $H^+_{\beta_t}$ is equal to the set of all $w \in W$ that do not have $t$ as a right inversion. If $i_M\cdots i_1$ is reduced word for $w$, then 
\begin{equation}\label{eq:inversions_roots_referee}
\{\alpha_{i_1},s_{i_2}\alpha_{i_1},s_{i_3}s_{i_2}\alpha_{i_1},\ldots,s_{i_M}\cdots s_{i_3}s_{i_2}\alpha_{i_1}\}=\{\beta_t:t\in\Inv(w)\}. 
\end{equation} 

If $W$ is a finite Coxeter group, then we denote the long element of $W$ (the unique element of maximum length) by $\wo$. It is known that $\wo$ is an involution and that $\wo \Phi^+ = \Phi^-$. 

Recall that for $J \subseteq I$, the \textit{(standard) parabolic subgroup} $W_J$ is the subgroup of $W$ generated by $\{s_i : i \in J\}$. Also recall that $(W_J, \{s_i : i \in J\})$ is a Coxeter system. Let $V_J = \mathbb{R}^J \subseteq V$ denote the root space of $W_J$, and let $\Phi_J = \{w\alpha_i:w\in W_J,\,\, i\in J\} \subseteq V_J$ denote the root system of $W_J$. By \cite[Theorem~3.3]{Dyer}, we have $\Phi_J = \Phi \cap V_J$. 

The \dfn{irreducible factors} of $W$ are the parabolic subgroups $W_{J_1},\ldots, W_{J_k}$, where $J_1,\ldots,J_k\subseteq I$ are the vertex sets of the connected components of the Coxeter graph $\Gamma_W$. We have the direct product decomposition $W\cong W_{J_1}\times\cdots\times W_{J_k}$. We say that $W$ is \dfn{irreducible} if $\Gamma_W$ is connected. 

\begin{remark}\label{rem:irreducible_futuristic}
It is straightforward to see that a Coxeter group is futuristic if and only if all of its irreducible factors are futuristic.  
\end{remark}

Let $\beta, \beta' \in \Phi$. We say $\beta$ \dfn{dominates} $\beta'$, written $\beta \dom \beta'$, if $H_\beta^+ \supseteq H_{\beta'}^+$. We say $\beta$ is \dfn{small} if $\beta \in \Phi^+$ and $\beta$ does not dominate any other positive root.

\begin{remark}
    Brink and Howlett \cite{BrinkHowlett} defined the dominance relation only for positive roots, but many of our results become easier to state if the relation is extended to all roots as above. The use of the word \emph{small} comes from Bj\"orner and Brenti \cite[Chapter~4.7]{BjornerBrenti}.
\end{remark}

Note that if $\beta \dom \beta'$, then $-\beta' \dom -\beta$ and $u\beta \dom u\beta'$ for all $u \in W$.

\begin{proposition}[{\cite[Theorem~4.7.6]{BjornerBrenti}}]\label{prop:small-roots-description}
    The set of small roots is the smallest set $\Sigma \subseteq \Phi^+$ that satisfies the following properties:
    \begin{enumerate}[(i)]
        \item All simple roots are in $\Sigma$.
        \item For $i \in I$ and $\gamma \in \Sigma \setminus \{\alpha_i\}$, if $0 > B(\gamma, \alpha_i) > -1$, then $s_i \gamma \in \Sigma$. 
    \end{enumerate}
    Moreover, the set of small roots is finite.
\end{proposition}
\begin{corollary}\label{cor:small-roots-of-parabolic-subgroup}
    Let $\Sigma$ denote the set of small roots of $W$. For $J \subseteq I$, the set of small roots of the parabolic subgroup $W_J$ is $\Sigma \cap V_J$.
\end{corollary}

\section{Basics of Bender--Knuth Billiards}\label{sec:basics}

In this section, we establish more terminology and collect important facts concerning our general Bender--Knuth billiards systems. 

\subsection{Separators}\label{subsec:strata}

As before, $(W, S)$ is a Coxeter system, and $S=\{s_i : i \in I\}$, where $I$ is a finite index set. Fix a nonempty convex set $\LL \subseteq W$. 

In \Cref{subsec:intro_Coxeter}, we stated that a noninvertible Bender--Knuth toggle $\tau_i$ cannot move an element ``strictly away from'' the convex set $\LL$. It is now time to formalize that observation. We define a \dfn{separator} for an element $u\in W$ to be an element of the set
\[\Sep(u) = \{\beta \in \RR(\LL) : u \beta \in \Phi^-\} \subseteq \RR(\LL).\]  (Note that $\Sep(u)$ depends on the convex set $\LL$.) The separators for $u$ correspond to (oriented) hyperplanes of the Coxeter arrangement that separate $\BB u$ from $\BB\LL$. Notice that since $\LL$ is convex, we have $u \in \LL$ if and only if $\Sep(u) = \varnothing$.
\begin{lemma}\label{lem:sep-decreasing}
    Let $u \in W$ and $i \in I$. Then $\Sep(\tau_i(u)) \subseteq \Sep(u)$. In fact,
    \begin{align*}
        \tau_i(u) = \begin{cases} s_iu & \mbox{if $\Sep(s_i u) \subseteq \Sep(u)$;} \\ u & \mbox{otherwise.} \end{cases}
    \end{align*}
\end{lemma}
\begin{proof}
A root $\gamma$ satisfies $\gamma \in \Phi^+$ and $s_i \gamma \in \Phi^-$ if and only if $\gamma=\beta_{s_i}=\alpha_i$.  Also, by the definition of separators,
\[\Sep(s_i u) \setminus \Sep(u)=\{\beta \in \RR(\LL): s_i u \beta \in \Phi^- \text{ and } u\beta \in \Phi^+\}.\] This set is
$\{u^{-1} \alpha_i\}$ if $u^{-1} \alpha_i \in \RR(\LL)$ and is $\varnothing$ otherwise.  The lemma now follows from the definition of $\tau_i$ (\cref{def:toggles}).
\end{proof}

One important consequence of \Cref{lem:sep-decreasing} is that $\Sep(u_0) = \Sep(u_1) = \Sep(u_2) = \cdots$ for any periodic billiards trajectory $u_0, u_1, u_2, \ldots$.

\begin{example}
    This example is a continuation of \Cref{ex:partial-order}. Let $\LL \subseteq \mathfrak S_{n+1}$ be a convex set. As before, there is an associated partial order $\leq_{\LL}$ on $[n+1]$. For each $u \in \mathfrak S_{n+1}$, we have $e_i - e_j \in \Sep(u)$ if and only if $i \leq_{\LL} j$ and $u(i) > u(j)$. In other words, if we think of $u$ as a labeling of the poset $([n+1], \leq_{\LL})$, then the separators for $u$ correspond to the pairs of comparable elements whose labels are ``in the wrong order''. \Cref{lem:sep-decreasing} says that applying a noninvertible Bender--Knuth toggle cannot create any new such pairs.
\end{example}

\begin{lemma}\label{lem:sep-finite}
    For every $u\in W$, the set $\Sep(u)$ is finite.
\end{lemma}
\begin{proof}
Fix some $w \in \LL$. For every $\beta \in \Sep(u)$, we have
$w\beta \in \Phi^+$ and $u\beta \in \Phi^-$.
Writing $u\beta=(uw^{-1})(w\beta)=(uw^{-1})\beta_{r_{w\beta}}$, we see that $r_{w\beta}$ is a right inversion of $uw^{-1}$. It follows that ${|\Sep(u)|\leq|\Inv(uw^{-1})|=\ell(uw^{-1})<\infty}$. 
\end{proof}

Recall that if $\mathsf w = i_M\cdots i_1$ is a word, then $\tau_{\mathsf w}$ denotes the composition $\tau_{i_M} \cdots \tau_{i_1}$.

\begin{lemma}\label{lem:sep-containment}
    Let $w \in W$, and let $\mathsf{w}$ be a reduced word for $w$. Then
    \[\Sep(\tau_{\mathsf{w}}(u)) \subseteq \Sep(w^{-1} \tau_{\mathsf{w}}(u))\] for all $u \in W$.
\end{lemma}
\begin{proof}
We proceed by induction on $M = \ell(w)$. The base case $M = 0$ is trivial, so assume that $M > 0$. Write $\mathsf{w} = i_M\cdots i_1$. Let $\mathsf{w}' = i_{M-1}\cdots i_{1}$ be the word obtained by deleting the leftmost letter of $\mathsf{w}$, and let $w' = s_{i_M} w \in W$ be the element of $W$ represented by the word $\mathsf{w}'$. Then $\tau_{\mathsf w}(u) = \tau_{i_M}(\tau_{\mathsf w'}(u))$.

First, suppose that $\tau_{\mathsf w}(u) = s_{i_M}\tau_{\mathsf w'}(u)$. Then $(w')^{-1} \tau_{\mathsf{w}'}(u) = w^{-1} \tau_{\mathsf{w}}(u)$. Hence, by \Cref{lem:sep-decreasing} and the inductive hypothesis, we have \[\Sep(\tau_{\mathsf{w}}(u)) \subseteq \Sep(\tau_{\mathsf{w}'}(u)) \subseteq \Sep((w')^{-1} \tau_{\mathsf{w}'}(u)) = \Sep(w^{-1} \tau_{\mathsf{w}}(u)),\] as desired.

Next, suppose that $\tau_{\mathsf w}(u) = \tau_{\mathsf w'}(u)$. Let $v = \tau_{\mathsf w}(u) = \tau_{\mathsf w'}(u)$. We wish to prove that \[{\Sep(v) \subseteq \Sep(w^{-1}v)}.\] Choose any $\beta \in \Sep(v)$; we will show that $\beta \in \Sep(w^{-1} v)$, i.e., that $w^{-1} v \beta \in \Phi^-$. Recall that
    \begin{equation}\label{eq:reflection-by-si1}
        s_{i_M} v \beta = v\beta - 2B(v \beta, \alpha_{i_M}) \alpha_{i_M}.
    \end{equation} We now consider two cases. 

\medskip 

   \noindent {\bf Case 1.} Assume $B(v\beta, \alpha_{i_M}) \geq 0$. Then applying $(w')^{-1}$ to each side of \eqref{eq:reflection-by-si1} yields \[w^{-1} v \beta = (w')^{-1} v \beta - 2B(v\beta, \alpha_{i_M}) (w')^{-1} \alpha_{i_M}.\] By the inductive hypothesis, we have $\beta \in \Sep((w')^{-1}v)$, so $(w')^{-1} v \beta \in \Phi^-$. Moreover, $s_{i_M}$ is not a right inversion of $(w')^{-1}$, so by \eqref{eq:inversions_roots}, we have $(w')^{-1} \alpha_{i_M} \in \Phi^+$. It follows that $w^{-1} v \beta \in \Phi^-$, as desired.

\medskip 

\noindent {\bf Case 2.} Assume $B(v\beta, \alpha_{i_M}) < 0$. We will prove that $v^{-1} s_{i_M} v \beta \in \Sep((w')^{-1} v)$; this will imply that \[w^{-1} v \beta = (w')^{-1} v (v^{-1} s_{i_M} v \beta) \in \Phi^-,\] as desired. Our inductive hypothesis tells us that $\Sep(v)\subseteq\Sep((w')^{-1} v)$, so it suffices to show that $v^{-1} s_{i_M} v \beta \in \Sep(v)$; that is, we want to show that $v^{-1} s_{i_M} v \beta\in\RR(\LL)$ and $v(v^{-1} s_{i_M} v) \beta \in \Phi^-$. 
   
Applying $v^{-1}$ to each side of \eqref{eq:reflection-by-si1} yields \[v^{-1} s_{i_M} v \beta = \beta - 2B(v\beta, \alpha_{i_M}) v^{-1} \alpha_{i_M}.\]
We know that $\beta \in \RR(\LL)$. Since $\tau_{\mathsf w}(u) = \tau_{\mathsf w'}(u)=v$, we also have $v^{-1} \alpha_{i_M} \in \RR(\LL)$. Because $\RR(\LL)$ is closed, this implies that $v^{-1}s_{i_M} v \beta \in \RR(\LL)$. Since $-\beta\not\in\RR(\LL)$, we also know that $-\beta\neq v^{-1}\alpha_{i_M}$, so $v\beta\neq-\alpha_{i_M}$. Since $v\beta\in\Phi^-\setminus\{-\alpha_{i_M}\}$, we also conclude that $v(v^{-1} s_{i_M} v) \beta = s_{i_M} v \beta \in \Phi^-$.
\end{proof}
\begin{remark}
\Cref{lem:sep-containment} is equivalent to the statement that the function $W \to W$ given by $u\mapsto  w^{-1} \tau_{\mathsf w}(u)$ is idempotent. It is unclear whether this formulation is useful.
\end{remark}

\subsection{Strata}

It will be useful to classify the elements of $W$ by their separators.  We define the \dfn{stratum} corresponding to a finite subset $R \subseteq \RR(\LL)$ to be the set \[\Str(R) = \{u \in W : \Sep(u) = R\}.\] Note that the nonempty strata form a partition of the set $W$. By \cref{lem:sep-decreasing}, any periodic billiards trajectory is contained in a single stratum. The stratum corresponding to the empty set is $\LL$ itself; we call the other nonempty strata the \dfn{proper strata}.

Strata have the following geometric interpretation. Recall that the hyperplanes $\HH_\beta$ for $\beta \in \RR(\LL)$ are called \emph{one-way mirrors} (see \cref{fig:affineS3,fig:spherical,fig:affineS3_infinite,fig:hyperbolic}, where one-way mirrors are drawn as yellow-and-red lines or curves). These hyperplanes form a subarrangement $\mathcal H^\LL_W$ of the Coxeter arrangement $\mathcal H_W$. The nonempty strata correspond to the intersections of the regions of $\mathcal H^\LL_W$ with the Tits cone. 

\begin{lemma}\label{lem:referee}
If $\LL$ is finite, then $|\Str(R)| \leq 2^{|\Phi \setminus (\RR(\LL) \cup (-\RR(\LL)))|}$ for all $R \subseteq \RR(\LL)$. 
\end{lemma}
\begin{proof}
Fix an arbitrary $v \in \Str(R)$. Each element $u\in\Str(R)$ is uniquely determined by the set of hyperplanes in the Coxeter arrangement $\mathcal H_W$ that separate $\BB v$ from $\BB u$. Each such hyperplane must be of the form $\HH_\beta$ for some root $\beta$ in the finite set $\Phi\setminus(\RR(\LL)\cup(-\RR(\LL)))$.
\end{proof}

Using strata, we may prove \cref{prop:finite-to-infinite}, which we now restate.
\propfinite*

\begin{proof}
It is clear that \eqref{item:everything-heavy} implies \eqref{item:finite-heavy} and that \eqref{item:sorting} implies \eqref{item:finite-heavy}. Therefore, it is enough to show that \eqref{item:finite-heavy} implies both \eqref{item:everything-heavy} and \eqref{item:sorting}. Write $c = s_{i_n} \cdots s_{i_1}$, where $i_1, \ldots, i_n$ is an ordering of $I$.

Assume \eqref{item:finite-heavy}; we will prove \eqref{item:everything-heavy}. Let $\LL \subseteq W$ be a (possibly infinite) convex set, and let $u_0, u_1, u_2, \ldots$ be a periodic billiards trajectory (defined with respect to the convex set $\LL$ and the ordering $i_1, \ldots, i_n$ of $I$). We wish to show that $u_0 \in \LL$. Our strategy will be to construct a finite convex subset $\LL' \subseteq \LL$ such that the billiards trajectory remains the same if the convex set $\LL$ is replaced by $\LL'$. It will then follow from \eqref{item:finite-heavy} that $u_0 \in \LL'$ and, consequently, that $u_0 \in \LL$.

Choose a positive integer $K$ such that the period of $u_0, u_1, u_2, \ldots$ divides $Kn$. For each $j \in [Kn]$ such that $u_j = s_{i_j}u_{j-1}$, we have $u_{j-1}^{-1}\alpha_{i_j} \not \in \RR(\LL)$ by the definition of the noninvertible Bender--Knuth toggles, so we may choose an element \[w_j \in \LL \setminus H^+_{u_{j-1}^{-1}\alpha_{i_j}}.\] Let $\LL'$ be the convex hull of $\{w_j : j \in [Kn],\, u_j = s_{i_j}u_{j-1}\}$. Then $\LL'$ is a finite convex subset of $\LL$ (see \cref{rem:Cayley}).

For $i \in I$, let $\tau_i'$ denote the noninvertible Bender--Knuth toggle defined with respect to the convex set $\LL'$. Let $u'_0 = u_0$, and define a billiards trajectory $u'_0, u'_1, u'_2, \ldots$ using the convex set $\LL'$ and the ordering $i_1, \ldots, i_n$ of $I$. (That is, define this sequence via the recurrence relation $u'_j = \tau'_{i_j}(u'_{j-1})$.) We claim that $u_j = u'_{j}$ for all nonnegative integers $j$. It suffices to prove that $u_j = u'_j$ for all $j \in [Kn]$. To do so, we use induction on $j$.

Suppose first that $u_j = u_{j-1}$. Then $u_{j-1}^{-1}\alpha_{i_j} \in \RR(\LL)$. Since $\LL' \subseteq \LL$, this implies that ${u_{j-1}^{-1}\alpha_{i_j} \in \RR(\LL')}$, so \[u'_j = \tau'_{i_j}(u'_{j-1}) = \tau'_{i_j}(u_{j-1}) = u_{j-1} = u_j.\] Next, suppose that $u_j = s_{i_j} u_{j-1}$. Then \[w_j \in \LL' \setminus H^+_{u_{j-1}^{-1}\alpha_{i_j}},\] so $u_{j-1}^{-1}\alpha_{i_j} \not \in \RR(\LL')$. Therefore, \[u'_j = \tau'_{i_j}(u'_{j-1}) = \tau'_{i_j}(u_{j-1}) = s_{i_j}u_{j-1} = u_{j}.\] This proves our claim that $u_j = u'_{j}$ for all nonnegative integers $j$. By \eqref{item:finite-heavy}, we have $u_0 \in \LL'\subseteq \LL$, as desired.

Now, still assuming \eqref{item:finite-heavy}, we will prove \eqref{item:sorting}. Let $\LL \subseteq W$ be a finite convex set, and let $u_0, u_1, u_2, \ldots \in W$ be a billiards trajectory (defined with respect to the convex set $\LL$ and the ordering $i_1, \ldots, i_n$ of $I$). We wish to show that the billiards trajectory eventually reaches $\LL$.


By \cref{lem:sep-decreasing}, we have 
\[\Sep(u_0) \supseteq \Sep(u_1) \supseteq \Sep(u_2) \supseteq \cdots.\] By \cref{lem:sep-finite}, the sequence $\Sep(u_0), \Sep(u_1), \Sep(u_2), \ldots$ is eventually constant; equivalently, there exists a single stratum $\Str(R)$ that contains $u_k$ for all sufficiently large positive integers $k$. The stratum $\Str(R)$ is finite by \cref{lem:referee}, so the billiards trajectory $u_0, u_1, u_2, \ldots$ is eventually periodic. According to \eqref{item:finite-heavy}, we have $u_k \in \LL$ for all sufficiently large positive integers $k$. This proves \eqref{item:sorting}.
\end{proof}

Let $R \subseteq \RR(\LL)$. A \dfn{transmitting root} of the stratum $\Str(R)$ is a root $\beta \in \RR(\LL)$ that can be written in the form $-w^{-1}\alpha_i$ for some $w \in \Str(R)$ and $i \in I$. Any such root belongs to the set $\Sep(w) = R$. The transmitting roots of $R$ correspond to the one-way mirrors that allow light to leave $R$.

\begin{lemma}\label{lem:transmitting-wall}
    Every proper stratum has at least one transmitting root. 
\end{lemma}

\begin{proof}
Let $v \in \LL$ be arbitrary. Choose some $w \in \Str(R)$ that minimizes the length $\ell(vw^{-1})$. Since $w\neq v$, there exists $i \in I$ such that
$\ell(v w^{-1} s_i) = \ell(vw^{-1}) - 1$. We claim that $-w^{-1}\alpha_i$ is a transmitting root. To prove this, it is enough to show that $-w^{-1}\alpha_i \in \RR(\LL)$.

By the minimality of $\ell(vw^{-1})$, we have that $s_iw\not\in\Str(R)$, so $\Sep(s_i w) \neq \Sep(w)$. That is, there exists a root $\beta \in \RR(\LL)$ such that $s_i w \beta$ and $w \beta$ have different signs. It follows that $\beta = \pm w^{-1} \alpha_i$. 
Since $\ell(vw^{-1}s_i)<\ell(vw^{-1})$, we see that $s_i$ is a right inversion of $vw^{-1}$, so $vw^{-1}\alpha_i \in \Phi^-$. Since $v\beta \in \Phi^+$, we cannot have $\beta = w^{-1} \alpha_i$, so $\beta=-w^{-1}\alpha_i$. Therefore, $-w^{-1}\alpha_i \in \RR(\LL)$, as desired.
\end{proof} 

Let $u \in W$ and $i \in I$, and let $\beta$ be a transmitting root of $\Str(R)$, where $R = \Sep(u)$. Under certain conditions on $u$, $i$, and $\beta$, it is possible to prove that $\tau_i(u) = s_i u$. We will prove our strongest lemma of this form (\Cref{lem:super-strong-acute-angle}) first and then deduce from it two weaker lemmas whose hypotheses are easier to check (\Cref{lem:no-root-in-positive-span,lem:acute-angle}). We remind the reader that in this setting, we have $\beta \in R = \Sep(u)$, so $u\beta \in \Phi^-$.
\begin{lemma}\label{lem:super-strong-acute-angle}
   Let $u \in W$ and $i \in I$, and let $\beta$ be a transmitting root of $\Str(R)$, where $R = \Sep(u)$. Suppose that there exists a positive root $\gamma$ such that $\gamma \neq -u \beta$ and \[H_\gamma^+ \cap H^+_{u\beta} \cap H_{\alpha_i}^+ = \varnothing.\] Then $\tau_i(u) = s_iu$.
\end{lemma}

\begin{proof}
    Let us rephrase the hypotheses of the lemma as follows. We have that $u \in H^+_{u^{-1}\gamma}$, that $\beta \neq -u^{-1} \gamma$, and that \[H_{u^{-1}\gamma}^+ \cap H^+_{\beta} \cap H_{u^{-1}\alpha_i}^+ = \varnothing.\] See \cref{fig:G2_shading} for an illustration of this scenario.
    
    Assume for the sake of contradiction that $\tau_i(u) = u$. By the definition of a transmitting root, there exist $w \in \Str(R)$ and $i' \in I$ such that $-w \beta = \alpha_{i'}$. We will show that \[s_{i'}w \in H_{u^{-1}\gamma}^+ \cap H^+_{\beta} \cap H_{u^{-1}\alpha_i}^+,\] which will contradict the fact that $H_{u^{-1}\gamma}^+ \cap H^+_{\beta} \cap H_{u^{-1}\alpha_i}^+ = \varnothing$.
    
    First, we will show that $s_{i'} w \in H_{u^{-1}\gamma}^+$. Since $\tau_i(u) = u$, we have $\beta, u^{-1} \alpha_i \in \RR(\LL)$, so \[\LL \subseteq H^+_\beta \cap H_{u^{-1}\alpha_i}^+ \subseteq H_{-u^{-1}\gamma}^+.\] Thus, ${-u^{-1}\gamma \in \RR(\LL)}$. We also have $u(-u^{-1}\gamma) = -\gamma \in \Phi^-$, so $-u^{-1}\gamma \in \Sep(u) = R = \Sep(w)$. Therefore, $w(-u^{-1}\gamma) \in \Phi^-$, or equivalently, $wu^{-1}\gamma \in \Phi^+$. Since $\gamma \neq -u\beta$, we have $wu^{-1}\gamma \neq \alpha_{i'}$. It follows that $s_{i'} w u^{-1} \gamma \in \Phi^+$ as well. Hence, $s_{i'}w \in H_{u^{-1}\gamma}^+$.
    
    Second, we have $s_{i'} w \in H_{\beta}^{+}$ since $s_{i'} w \beta = s_{i'}(-\alpha_{i'}) = \alpha_{i'} \in \Phi^+$.
    
    Finally, we will show that $s_{i'} w \in H_{u^{-1}\alpha_i}^+$. We have $u^{-1}\alpha_i \in \RR(\LL)$ (since $\tau_i(u)=u$) and $u(u^{-1}\alpha_i) = \alpha_i \in \Phi^+$, so $u^{-1}\alpha_i \not \in \Sep(u) = R = \Sep(w)$. Therefore, $wu^{-1}\alpha_i \in \Phi^+$. Since $w^{-1} \alpha_{i'} = -\beta \not \in \RR(\LL)$, we have $w^{-1} \alpha_{i'} \neq u^{-1}\alpha_i$, so $wu^{-1}\alpha_i \neq \alpha_{i'}$. It follows that $s_{i'}wu^{-1}\alpha_i \in \Phi^+$, so $s_{i'}w \in H_{u^{-1}\alpha_i}^+$.
\end{proof}

\begin{figure}[ht]
 \begin{center}{\includegraphics[height=7.195cm]{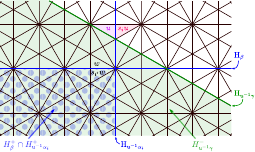}}
  \end{center}
\caption{An illustration of the proof of \cref{lem:super-strong-acute-angle}, drawn using the positive projectivization of the Tits cone and Coxeter arrangement of $\widetilde G_2$, whose Coxeter graph is $\begin{array}{l}\includegraphics[height=0.373cm]{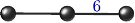}\end{array}.$ 
\hspace{.2cm}The assumption that $\tau_i(u)=u$ implies that ${\LL\subseteq H^+_{\beta}\cap H_{u^{-1}\alpha_i}^+\subseteq H_{u^{-1}\gamma}^-}$, which ends up contradicting the hypothesis that $\beta$ is a transmitting root of the stratum containing $u$.  
}\label{fig:G2_shading}
\end{figure} 

\begin{lemma}\label{lem:no-root-in-positive-span}
     Let $u \in W$ and $i \in I$, and let $\beta$ be a transmitting root of $\Str(R)$, where $R = \Sep(u)$. Suppose that there exist real numbers $a, a' > 0$ such that $a u \beta + a' \alpha_{i} \in \Phi$. Then $\tau_i(u) = s_i u$.
\end{lemma}
\begin{proof}
    First, suppose that $u\beta$ and $\alpha_i$ are linearly dependent. Since $\beta \in \Sep(u)$, this implies that $-u\beta = \alpha_i$. Since $\beta \in \RR(\LL)$, we have $u^{-1}\alpha_i = -\beta \not \in \RR(\LL)$, so $\tau_i(u) = s_i u$ by the definition of $\tau_i$.

    Next, suppose that $u\beta$ and $\alpha_i$ are linearly independent. Let $\gamma = - (au \beta + a' \alpha_i) \in \Phi$. We have $- u \beta \in \Phi^+$, so when $\gamma$ is written in the basis of simple roots, all the coefficients are nonnegative except possibly the coefficient of $\alpha_i$. But $\gamma \neq -\alpha_i$, so $\gamma \in \Phi^+$. Moreover, $\gamma \neq -u \beta$, and \[H_\gamma^+ \cap H^+_{u\beta} \cap H_{\alpha_i}^+ = \{w \in W : \{-awu\beta - a'w\alpha_i,\, wu\beta,\, w\alpha_i\} \subseteq \Phi^+\} = \varnothing.\] By \cref{lem:super-strong-acute-angle}, we have $\tau_i(u) = s_iu$.
\end{proof}

Informally, the following lemma says that a billiards trajectory cannot reflect off of any one-way mirror that forms an obtuse angle with the one-way mirror corresponding to a transmitting root.  (Recall that, since the bilinear form $B$ is $W$-equivariant, we have $B(u\beta, \alpha_i)=B(\beta, u^{-1}\alpha_i)$.)
\begin{lemma}\label{lem:acute-angle}
    Let $u \in W$ and $i \in I$, and let $\beta$ be a transmitting root of $\Str(R)$, where $R = \Sep(u)$. If $B(u\beta, \alpha_i) < 0$, then $\tau_i(u) = s_i u$.
\end{lemma}

\begin{proof}
    We have $s_i u \beta = u \beta - 2B(u\beta, \alpha_i) \alpha_i \in \Phi$. Now take $a = 1$ and $a' = -2B(u\beta, \alpha_i)$ in \Cref{lem:no-root-in-positive-span}.
\end{proof}

The following lemma will be useful in the proofs of \Cref{thm:rank_3,thm:complete,thm:right-angled}, where we will compute the set of small roots and use it to determine the structure of the strata.

\begin{lemma}\label{lem:transmitting-small}
    Let $\Str(R)$ be a stratum, and let $\beta$ be a transmitting root of $\Str(R)$. For every $u\in\Str(R)$, the root $-u\beta$ is small.
\end{lemma}
\begin{proof}
    We have $-u \beta \in \Phi^+$ because $\beta \in R = \Sep(u)$. Now, let $\gamma \in \Phi^+$, and suppose that 
    \begin{equation}\label{eq:dominance}-u \beta \dom \gamma.\end{equation} We must show that $\gamma = -u\beta$.
    
    We know by \eqref{eq:dominance} that $-\gamma \dom u \beta$. Applying $u^{-1}$ to each side yields $-u^{-1} \gamma \dom \beta$. Then ${\LL \subseteq H^+_{\beta} \subseteq H^+_{-u^{-1}\gamma}}$, so $-u^{-1}\gamma \in \RR(\LL)$. Moreover, $u (-u^{-1}\gamma) = -\gamma \in \Phi^-$, so ${-u^{-1}\gamma \in \Sep(u) = R}$.

     By the definition of a transmitting root, there exists $w \in \Str(R)$ such that $-w\beta$ is a simple root. Since $-u^{-1} \gamma \in R = \Sep(w)$, we have that $wu^{-1}\gamma \in \Phi^+$. Applying $wu^{-1}$ to each side of \eqref{eq:dominance} yields $-w\beta \dom wu^{-1}\gamma$. Since $-w\beta$ is simple, it is small, so $wu^{-1}\gamma = -w\beta$. Therefore $\gamma = -u \beta$, as desired.
\end{proof}

\subsection{Conjugate Coxeter Elements}\label{subsec:conjugate} 
Consider the Coxeter graph $\Gamma_W$ of $W$. An \dfn{acyclic orientation} of $\Gamma_W$ is a directed graph with no directed cycles that is obtained by orienting each edge of $\Gamma_W$. A \dfn{sink} (respectively, \dfn{source}) of an acyclic orientation is a vertex with out-degree (respectively, in-degree) $0$. If $i$ is a source or a sink of $\Gamma_W$, then we can \dfn{flip} $i$ by reversing the orientations of all edges incident to $i$; the result is a new acyclic orientation of $\Gamma_W$. We say two acyclic orientations of $\Gamma_W$ are \dfn{flip equivalent} if one can be obtained from the other via a sequence of flips. 

Let $c$ be a Coxeter element of $W$. Let us direct the edge between two adjacent vertices $i$ and $i'$ in $\Gamma_W$ from $i$ to $i'$ if $i$ appears to the right of $i'$ in some (equivalently, every) reduced word for $c$. Doing this for all edges in $\Gamma_W$ yields an acyclic orientation $\ao(c)$ of $\Gamma_W$. 

The map $\ao$ is a bijection from the set of Coxeter elements of $W$ to the set of acyclic orientations of $\Gamma_W$. Moreover, it is known \cite[Theorem~1.15]{Develin} that two Coxeter elements $c$ and $c'$ are conjugate in $W$ if and only if $\ao(c)$ and $\ao(c')$ are flip equivalent. 

\begin{proposition}\label{prop:conjugate_Coxeter}
Let $c$ and $c'$ be Coxeter elements of $W$ that are conjugate to each other. Then $c$ is futuristic if and only if $c'$ is futuristic. 
\end{proposition}

 \begin{proof}
According to \cite[Theorem~1.15]{Develin}, the acyclic orientations $\ao(c)$ and $\ao(c')$ are flip equivalent. Hence, we may assume that $\ao(c')$ is obtained from $\ao(c)$ via a single flip from a source to a sink. That is, there is a reduced word $i_ni_{n-1}\cdots i_1$ for $c$ such that $i_{n-1}\cdots i_1i_{n}$ is a reduced word for~$c'$. 

Fix an arbitrary nonempty convex set $\LL\subseteq W$. Suppose $c$ is futuristic. Let $u\in W$ be a periodic point of $\Pro_{c'}$. We have that $\Pro_c \circ \tau_{i_n} = \tau_{i_n} \circ \Pro_{c'}$, so $\tau_{i_n}(u)$ must be a periodic point of $\Pro_c$. Since $c$ is futuristic, this implies that $\tau_{i_n}(u)\in\LL$. Because $u$ is a periodic point of $\Pro_{c'}$, it can be obtained by applying a sequence of noninvertible Bender--Knuth toggles to $\tau_{i_n}(u)$. Hence, $u\in\LL$. This proves that $c'$ is futuristic. 

We have shown that if $c$ is futuristic, then $c'$ is futuristic. A similar argument proves the other direction.
 \end{proof}

It is well known that all acyclic orientations of a forest are flip equivalent, so we immediately obtain the following corollary. 

\begin{corollary}\label{cor:tree}
If the Coxeter graph of a Coxeter group $W$ is a forest, then $W$ is either futuristic or ancient.  
\end{corollary}

\subsection{Futuristicity is Hereditary}

We now prove \cref{prop:hereditary}, which states that the class of futuristic Coxeter groups and the class of non-ancient Coxeter groups are \emph{hereditary} in the sense that they are closed under taking standard parabolic subgroups. In fact, these results are immediate corollaries of the following stronger proposition. 

\begin{proposition}\label{prop:stronger_hereditary}
Let $J\subseteq I$. Let $c$ be a Coxeter element of $W$, and let $c'$ be a Coxeter element of the standard parabolic subgroup $W_J$ such that there exists a reduced word for $c$ that contains a reduced word for $c'$ as a subword. If $c$ is futuristic, then $c'$ is also futuristic. 
\end{proposition}

\begin{proof}
    Assume that $c$ is futuristic; we will show that $c'$ is futuristic. Let $\LL \subseteq W_J$ be a nonempty convex set; we will show that every periodic point of $\Pro_{c'}$ lies in $\LL$.

    Note that $W_J$ is precisely the set of elements of $W$ with no right inversions in $W \setminus W_J$, so $W_J$ is a convex subset of $W$, with $\RR(W_J)=\Phi^+ \setminus \Phi_J$.  Now, write $\RR_J(\LL)=\{\beta\in\Phi_J:\LL\subseteq H_\beta^+\}$. We have
    \[\RR(\LL) \cap \Phi_J=\RR_J(\LL),\]
    and the inclusion $\LL \subseteq W_J$ gives $\Phi^+ \setminus \Phi_J \subseteq \RR(\LL)$. At the same time, we have $w\beta\in\Phi^-$ for all $w\in\LL$ and $\beta\in\Phi^-\setminus\Phi_J$, so $\RR(\LL)$ is disjoint from $\Phi^-\setminus\Phi_J$. Putting this all together yields \begin{equation}\label{eq:RR'}
    \RR(\LL)=\RR_J(\LL)\cup(\Phi^+\setminus\Phi_J).
    \end{equation} Moreover, $\LL$ is the intersection of the half-spaces $H_\beta^+$ for $\beta \in \RR_J(\LL) \cup (\Phi^+ \setminus \Phi_J)$, so $\LL$ is convex as a subset of $W$.

    For $i\in I$, let $\tau_i\colon W\to W$ denote the noninvertible Bender--Knuth toggle on $W$ that is defined with respect to $\LL$. For $i\in J$, let $\tau^J_{i}\colon W_J\to W_J$ denote the noninvertible Bender--Knuth toggle on $W_J$ that is defined with respect to $\LL$.
    
    We claim that for all $i \in J$, the restriction of $\tau_i$ to $W_J$ is $\tau_i^J$. Moreover, we claim that for all $i \in I \setminus J$, the restriction of $\tau_i$ to $W_J$ is the identity function on $W_J$. Let $i \in I$ and $u \in W_J$. Then $\tau_i(u) = u$ if and only if $u^{-1} \alpha_{i} \in\RR(\LL)$. By \eqref{eq:RR'}, this holds if and only if $u^{-1}\alpha_i \in \RR_J(\LL)$ or $i \not \in J$; this proves both claims.

    It follows that the restriction of $\Pro_c$ to $W_J$ is $\Pro_{c'}$. Since $c$ is futuristic, every periodic point of $\Pro_{c}$ lies in $\LL$. Hence, every periodic point of $\Pro_{c'}$ lies in $\LL$, as desired.
\end{proof}

\subsection{Folding}\label{subsec:folding}

There is a common technique known as \emph{folding} that can be used to obtain new Coxeter groups from old ones \cite{Stembridge}. In this subsection, we discuss this technique and its relation to Bender--Knuth billiards. 
We were not able to find a discussion of folding root systems in the literature at the level of generality that we present here, so we will prove all of the relevant statements. However, on the level of Coxeter groups, these foldings are discussed in \cite{Dyer2, Muhlherr}.

Let $(W,S)$ be a Coxeter system whose simple reflections are indexed by a finite index set $I$, and let $\sigma\colon I\to I$ be an automorphism of the Coxeter graph $\Gamma_W$ such that each orbit (i.e., cycle) of $\sigma$ is an independent set (i.e., a set of pairwise nonadjacent vertices). Let $I^{\fold}$ be the set of orbits of $\sigma$.\footnote{The symbol \raisebox{0.2ex}{\resizebox{1.3\width}{0.65\height}{\rotatebox{90}{\textsf{w}}}} is pronounced ``fold''.} Consider a symbol $s_{\mathfrak o}$ for each orbit $\mathfrak o\in I^{\fold}$. Let $S^{\fold}=\{s_{\mathfrak o}:\mathfrak o\in I^{\fold}\}$, and let $(W^{\fold},S^{\fold})$ be the Coxeter system such that for all $\mathfrak o,\mathfrak o'\in I^{\fold}$, the order $m(\mathfrak o,\mathfrak o')$ of $s_{\mathfrak o}s_{\mathfrak o'}$ in $W^{\fold}$ is equal to the order of $\prod_{i\in \mathfrak o}s_i\prod_{i'\in \mathfrak o'}s_{i'}$ in $W$. Observe that there is a group homomorphism $\iota\colon W^{\fold}\to W$ determined by $\iota(s_{\oo})=\prod_{i\in \oo}s_i$. (The map $\iota$ is actually injective \cite{Muhlherr}, but we will not need this.) We call $W^{\fold}$ a \dfn{folding} of $W$. We also say $W^{\fold}$ is obtained from $W$ by \dfn{folding along} $\sigma$. 

\begin{example}\label{exam:folding}
The Coxeter graph of the affine Coxeter group $\widetilde E_6$ is 
\[\begin{array}{l}
\includegraphics[height=1.41cm]{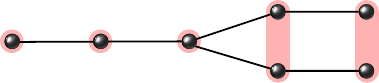}
\end{array}\,\,.\]
We have shaded in red the orbits of a Coxeter graph automorphism of order 2. By folding along this automorphism, we obtain the affine Coxeter group $\widetilde F_4$, whose Coxeter graph is 
\[\includegraphics[height=0.373cm]{BKTogglesPIC22}\,\,.\] 

The Coxeter graph of the affine Coxeter group $\widetilde E_7$ is \[\begin{array}{l}
\includegraphics[height=1.41cm]{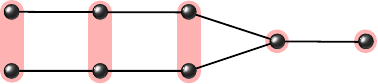}
\end{array}\,\,.\] As in the case of $\widetilde E_6$, we have shaded in red the orbits of a Coxeter graph automorphism of order 2. By folding along this automorphism, we again obtain $\widetilde F_4$. 
\end{example}

For each orbit $\oo\in I^{\fold}$, let us fix a particular index $i_{\oo}\in\oo$. For $\oo,\oo'\in I^{\fold}$, let $\deg_{\oo'}(\oo)$ denote the number of edges in $\Gamma_W$ of the form $\{i_{\oo},i'\}$ for $i'\in \oo'$. Let $\lab_{\oo'}(\oo)$ be the multiset of labels of the form $m(i_{\oo},i')$ with $i'\in\oo'$ and $m(i_{\oo},i')\geq 3$, where the multiplicity of a label $a$ in $\lab_{\oo'}(\oo)$ is equal to ${|\{i'\in\oo':m(i_{\oo},i')=a\}|}$. (Because $\sigma$ is a Coxeter graph automorphism, neither $\deg_{\oo'}(\oo)$ nor $\lab_{\oo'}(\oo)$ depends on the choice of $i_{\oo}$.) Thus, $\deg_{\oo'}(\oo)$ is the cardinality of $\lab_{\oo'}(\oo)$, counted with multiplicity. In particular, if $\deg_{\oo'}(\oo)=0$ (i.e., $\oo\cup\oo'$ is an independent set in $\Gamma_W$), then $\lab_{\oo'}(\oo)$ is empty. The following lemma allows us to determine the label $m(\oo,\oo')$ of the edge between $\oo$ and $\oo'$ in $\Gamma_{W^{\fold}}$ from the tuples $\lab_{\oo'}(\oo)$ and $\lab_{\oo}(\oo')$. 

\begin{lemma}\label{lem:folded_orders}
Let $\oo,\oo'$ be distinct orbits in $I^{\fold}$, and assume that $\deg_{\oo'}(\oo)\leq\deg_{\oo}(\oo')$. 
\begin{itemize}
\item If $\deg_{\oo'}(\oo)=0$, then $m(\oo,\oo')=2$. 
\item If $\deg_{\oo'}(\oo)=\deg_{\oo}(\oo')=1$, then $\lab_{\oo'}(\oo)=\lab_{\oo}(\oo')=\{m(\oo,\oo')\}$. 
\item If $\deg_{\oo'}(\oo)=1$ and $\lab_{\oo}(\oo')=\{3,3\}$, then $m(\oo,\oo')=4$. 
\item If $\deg_{\oo'}(\oo)=1$ and $\lab_{\oo}(\oo')=\{3,3,3\}$, then $m(\oo,\oo')=6$. 
\item If $\deg_{\oo'}(\oo)=1<\deg_{\oo}(\oo')$ and $\lab_{\oo}(\oo')$ is not $\{3,3\}$ or $\{3,3,3\}$, then $m(\oo,\oo')=\infty$. 
\item If $\deg_{\oo'}(\oo)\geq 2$, then $m(\oo,\oo')=\infty$. 
\end{itemize}
\end{lemma}

\begin{proof}
By passing to the parabolic subgroup $W_{\oo \cup \oo'}$, we may assume that $\oo \cup \oo' = I$.

It is immediate from the definitions that $m(\oo,\oo')=2$ if $\deg_{\oo'}(\oo)=0$. 

Suppose that $\deg_{\oo'}(\oo)=1$. The irreducible factors of $W$ are pairwise isomorphic Coxeter groups, and the Coxeter graph of each irreducible factor is a star graph whose edges all have the same label. To compute the order of $\prod_{i\in \mathfrak o}s_i\prod_{i'\in \mathfrak o'}s_{i'}$ in $W$, we compute the order after projecting to each irreducible factor, and we then take the least common multiple of the orders obtained this way (which will all be equal because the irreducible factors are isomorphic). We can check directly that this order is $a$ if $\lab_{\oo}(\oo')=\{a\}$, is $4$ if $\lab_{\oo}(\oo')=\{3,3\}$, is $6$ if $\lab_{\oo}(\oo')=\{3,3,3\}$, and is $\infty$ otherwise. 

Finally, suppose that $\deg_{\oo'}(\oo)\geq 2$. Let $\oo=\{i_1,\ldots,i_k\}$ and $\oo'=\{i_1',\ldots,i_{k'}'\}$, and let $b\geq 1$ be an integer. Upon inspection, we find that it is impossible to apply a sequence of commutation moves to the word $(i_1\cdots i_ki_1'\cdots i_{k'}')^b$ to obtain a word to which we can apply a nil move or a braid move that is not a commutation move. By Matsumoto's theorem (\cref{thm:matsumoto}), this implies that $(i_1\cdots i_ki_1'\cdots i_{k'}')^b$ is a reduced word. Hence, $\prod_{i\in\oo}s_i\prod_{i'\in\oo'}s_{i'}$ does not have order $b$. As $b$ was arbitrary, we conclude that $m(\oo,\oo')=\infty$. 
\end{proof}

Let us now discuss how root systems interact with folding. Recall that when defining the bilinear form in \eqref{eq:bilinear}, we had to choose a number $\mu_{\{i,i'\}}\geq 1$ whenever $m(i,i')=\infty$; we now stress that, in this subsection, we will sometimes need to choose a value for $\mu_{\{i,i'\}}$ that is strictly greater than $1$.

Let $V$ be the root space of $W$, and let $B$ be a bilinear form on $V$ induced by $W$. Let $V^{\fold}$ be the root space of $W^{\fold}$. Define a bilinear form $B^{\fold}\colon V^{\fold}\times V^{\fold}\to\mathbb R$ by setting
\begin{equation}\label{eq:folded_form_1}
B^{\fold}(\alpha_{\oo},\alpha_{\oo'})=-\sqrt{\left(\sum_{i'\in\oo'}B(\alpha_{i_\oo},\alpha_{i'})\right)\left(\sum_{i\in\oo}B(\alpha_{i_{\oo'}},\alpha_{i})\right)}.
\end{equation} 
Because $\sigma$ is a Coxeter graph automorphism, we can equivalently write
\begin{equation}\label{eq:folded_form_equivalent}B^{\fold}(\alpha_{\oo},\alpha_{\oo'})=\sqrt{\frac{\deg_{\oo}(\oo')}{\deg_{\oo'}(\oo)}}\sum_{i'\in\oo'}B(\alpha_{i_\oo},\alpha_{i'})=\sqrt{\frac{\deg_{\oo'}(\oo)}{\deg_{\oo}(\oo')}}\sum_{i\in\oo}B(\alpha_{i_{\oo'}},\alpha_{i}).
\end{equation} 
It will be useful to keep in mind that 
\begin{equation}\label{eq:degoo}
\sqrt{\frac{\deg_{\oo}(\oo')}{\deg_{\oo'}(\oo)}}=\sqrt{\frac{|\oo|}{|\oo'|}}
\end{equation}
(this is also a consequence of the fact that $\sigma$ is a Coxeter graph automorphism). 

\begin{lemma}\label{lem:form_induced}
The bilinear form $B^{\fold}$ is induced by $W^{\fold}$. 
\end{lemma} 
\begin{proof}
Fix $\oo,\oo'\in I^{\fold}$, and assume without loss of generality that $\deg_{\oo'}(\oo)\leq\deg_{\oo}(\oo')$. The proof follows from \cref{lem:folded_orders} and \eqref{eq:folded_form_equivalent}. For example, if $\deg_{\oo'}(\oo)=1$ and $\lab_{\oo}(\oo')=\{3,3,3\}$, then 
\[B^{\fold}(\alpha_{\oo},\alpha_{\oo'})=\sqrt{\frac{\deg_{\oo}(\oo')}{\deg_{\oo'}(\oo)}}\sum_{i'\in\oo'}B(\alpha_{i_{\oo}},\alpha_{i'})=\sqrt{\frac{3}{1}}(-\cos(\pi/3))=-\cos(\pi/6),\] and $m(\oo,\oo')=6$ by \cref{lem:folded_orders}. If $\deg_{\oo'}(\oo)\geq 2$, then \[B^{\fold}(\alpha_{\oo},\alpha_{\oo'})=\sqrt{\frac{\deg_{\oo}(\oo')}{\deg_{\oo'}(\oo)}}\sum_{i'\in\oo'}B(\alpha_{i_{\oo}},\alpha_{i'}) \leq \sum_{i'\in\oo'}B(\alpha_{i_{\oo}},\alpha_{i'})\leq 2(-\cos(\pi/3))=-1,\] and $m(\oo,\oo')=\infty$ by \cref{lem:folded_orders}. Similar arguments handle the other cases. 
\end{proof}

\Cref{lem:form_induced} allows us to consider the standard geometric representation of $W^{\fold}$ defined with respect to $B^{\fold}$. Define the linear transformation $\omega\colon V\to V^{\fold}$ by taking $\omega(\alpha_i)=\frac{1}{\sqrt{|\oo|}}\alpha_{\oo}$ for all $\oo\in I^{\fold}$ and all $i\in\oo$. We claim that 
\begin{equation}\label{eq:equivariance}
\omega(\iota(v)\beta)=v\omega(\beta)
\end{equation}
for all $v\in W^{\fold}$ and $\beta\in V$. It suffices to prove this when $\beta= \alpha_{i}$ is a simple root and $v = s_{\oo'}$ is a simple reflection. Let $\oo\in I^{\fold}$ be the orbit containing $i$. Then 
\begin{align*}
\omega(\iota(s_{\oo'})\alpha_{i})&=\omega\left(\alpha_{i}-\sum_{i'\in\oo'}B(\alpha_{i},\alpha_{i'})\alpha_{i'}\right) \\ 
&=\frac{1}{\sqrt{|\oo|}}\alpha_{\oo}-\sum_{i'\in\oo'}B(\alpha_i,\alpha_{i'})\frac{1}{\sqrt{|\oo'|}}\alpha_{\oo'} \\ 
&=\frac{1}{\sqrt{|\oo|}}\alpha_{\oo}-\frac{1}{\sqrt{|\oo|}}B^{\fold}(\alpha_{\oo},\alpha_{\oo'})\alpha_{\oo'},
\end{align*}
where the last equality follows from \eqref{eq:folded_form_equivalent} and \eqref{eq:degoo}. Because
\[
s_{\oo'}\omega(\alpha_i)=\frac{1}{\sqrt{|\oo|}}s_{\oo'}\alpha_{\oo}=\frac{1}{\sqrt{|\oo|}}\alpha_{\oo}-\frac{1}{\sqrt{|\oo|}}B^{\fold}(\alpha_{\oo},\alpha_{\oo'})\alpha_{\oo'},\]
this completes the proof of \eqref{eq:equivariance}. 

In what follows, we will write $\Phi$ and $\Phi^{\fold}$ for the root system of $W$ (defined with respect to $B$) and the root system of $W^{\fold}$ (defined with respect to $B^{\fold}$), respectively.  A version of this lemma (without root systems) also appears as \cite[Theorem 6.1(c)]{Dyer2}.

\begin{lemma}\label{lem:iota_easy}
Let $\oo\in I^{\fold}$, and let $i\in\oo$. For $v\in W^{\fold}$, we have $v\alpha_{\oo}\in(\Phi^{\fold})^+$ if and only if $\iota(v)\alpha_i\in\Phi^+$. 
\end{lemma}
\begin{proof}
According to \eqref{eq:equivariance}, we have $v\alpha_{\oo}=\sqrt{|\oo|}\omega(\iota(v)\alpha_i)$, so the proof follows from the fact that $\omega$ maps each simple root of $W$ to a positive scalar multiple of a simple root of $W^{\fold}$. 
\end{proof}

Let $\LL^{\fold}$ be a nonempty convex subset of $W^{\fold}$, and let $\LL$ be the convex hull of $\iota(\LL^{\fold})$ in $W$. 
For $i\in I$, let $\tau_i$ be the noninvertible Bender--Knuth toggle on $W$ defined with respect to $\LL$. For $\oo\in I^{\fold}$, let $\tau_{\oo}$ be the noninvertible Bender--Knuth toggle on $W^{\fold}$ defined with respect to $\LL^{\fold}$, and let $\overline{\tau}_{\oo}=\prod_{i\in\oo}\tau_i$, where $\prod$ denotes composition. The order of the composition does not matter because $\oo$ is an independent set of $\Gamma_W$. 

\begin{lemma}\label{lem:iota_commutes}
For every $\oo\in I^{\fold}$, we have 
$\iota\circ\tau_{\oo}=\overline{\tau}_{\oo}\circ\iota$.
\end{lemma}

\begin{proof}
Fix $\oo\in I^{\fold}$ and $x\in W^{\fold}$. We wish to prove that $\iota(\tau_{\oo}(x))=\overline{\tau}_{\oo}(\iota(x))$. To do so, it suffices to show that for each $i\in \oo$, we have $\tau_{\oo}(x)=x$ if and only if $\tau_i(\iota(x))=\iota(x)$. 

Fix $i\in\oo$. We have $\tau_{\oo}(x)=x$ if and only if $wx^{-1}\alpha_{\oo}\in(\Phi^{\fold})^+$ for all $w\in\LL^{\fold}$. According to \cref{lem:iota_easy}, this holds if and only if $\iota(w)(\iota(x)^{-1}\alpha_i)\in\Phi^+$ for all $w\in \LL^{\fold}$. Because $\LL$ is the convex hull of $\iota(\LL^{\fold})$, it follows that $\tau_{\oo}(x)=x$ if and only if $z(\iota(x)^{-1}\alpha_i)\in\Phi^+$ for all $z\in \LL$. This, in turn, holds if and only if $\tau_i(\iota(x))=\iota(x)$. 
\end{proof}

We are now prepared to prove \cref{prop:folding}, which states that a folding of a futuristic Coxeter group is futuristic. This result is immediate from the following stronger proposition. Note that if $c$ is a Coxeter element of $W^{\fold}$, then $\iota(c)$ is a Coxeter element of $W$. 

\begin{proposition}
Let $W^{\fold}$ be a folding of $W$, and let $c$ be a Coxeter element of $W^{\fold}$. If $\iota(c)$ is futuristic, then $c$ is futuristic.  
\end{proposition}
\begin{proof}
Preserve the notation from above. We prove the contrapositive. Suppose that $c$ is not futuristic. Then there exist a nonempty convex set $\LL^{\fold} \subseteq W^{\fold}$ and an element $u\in W^{\fold}\setminus\LL^{\fold}$ such that $\Pro_c^K(u)=u$ for some integer $K\geq 1$. Let $\LL$ be the convex hull in $W$ of $\iota(\LL^{\fold})$. It follows from \cref{lem:iota_commutes} that $\iota\circ\Pro_c=\Pro_{\iota(c)}\circ\iota$. Hence, $\Pro_{\iota(c)}^K(\iota(u))=\iota(\Pro_c^K(u))=\iota(u)$. 

We need to show that $\iota(u)\not\in\LL$. Because $u\not\in \LL^{\fold}$, there exists $\beta\in\Phi^{\fold}$ such that $u\beta\in(\Phi^{\fold})^-$ and $w\beta\in(\Phi^{\fold})^+$ for all $w\in\LL^{\fold}$. We can write $\beta=y\alpha_{\oo}$ for some $y\in W^{\fold}$ and $\oo\in I^{\fold}$. Choose $i\in \oo$. For each $v\in W^{\fold}$, we can use \cref{lem:iota_easy} to see that $v\beta\in(\Phi^{\fold})^+$ if and only if $\iota(v)(\iota(y)\alpha_i)\in\Phi^+$. It follows that $\iota(u)(\iota(y)\alpha_i)\in\Phi^-$ and that $\iota(w)(\iota(y)\alpha_i)\in\Phi^+$ for all $w\in \LL^{\fold}$. Since $\LL$ is the convex hull of $\iota(\LL^{\fold})$, we deduce that $\iota(u)\not\in\LL$. 
\end{proof}

\section{Finite Coxeter Groups}\label{sec:finite}  

Throughout this section, we assume that the Coxeter group $W$ is finite.

The machinery developed in \cref{subsec:strata} allows us to painlessly finish the proofs of \cref{thm:main1,thm:main2}. 

\begin{proof}[Proof of \Cref{thm:main1}]
    Let $\mathsf{w}_\circ$ be a reduced word for $w_\circ$. Let $u \in W$, and assume for the sake of contradiction that $\tau_{\mathsf{w}_\circ}(u) \not \in \LL$. Then there is some $\beta \in \Sep(\tau_{\mathsf{w}_\circ}(u))$. According to \Cref{lem:sep-containment}, we have $\beta \in \Sep(\wo^{-1}\tau_{\mathsf{w}_\circ}(u))$. But then the roots $\tau_{\mathsf{w}_\circ}(u) \beta$ and $\wo^{-1}\tau_{\mathsf{w}_\circ}(u)\beta$ are both negative, which contradicts the fact that $\wo \Phi^+ = \Phi^-$ (recall that $\wo^{-1}=\wo$). 
\end{proof}

In the following proof, recall the commutation equivalence relation $\equiv$ defined in \cref{sec:preliminaries}.

\begin{proof}[Proof of \cref{thm:main2}]
Let $\sfc$ be a reduced word for the Coxeter element $c$. It follows from \cite[Corollary~4.1]{Speyer} that there is a reduced word $\mathsf{w}_\circ$ for $\wo$ such that $\sfc^{\MM(c)}\equiv \y\mathsf{w}_\circ$ for some word $\y$. Appealing to \cref{thm:main1}, we find that \[\Pro_c^{\MM(c)}(W)=\tau_\y(\tau_{\mathsf{w}_\circ}(W))=\tau_\y(\LL)=\LL. \qedhere\]
\end{proof}

In the remainder of this section, we provide a (slightly informal) discussion of how one can compute $\MM(c)$ explicitly. 

Let $\x$ be a word. Let $i\in I$, and let $p_i(\x)$ be the number of occurrences of $i$ in $\x$. For $1\leq k\leq p_i(\x)$, let $i^{(k)}$ be the $k$-th occurrence of $i$ in $\x$ (counted from the right). It will be convenient to view $i^{(1)},\ldots,i^{(p_i(\x))}$ as distinct entities. The \dfn{heap} of $\x$ is a certain poset whose elements are the letters in $\x$ (which are seen as distinct from one another, even if they are equal as elements of $I$). The order relation is defined so that if $i_1^{(k_1)}$ and $i_2^{(k_2)}$ are two letters, then $i_1^{(k_1)}<i_2^{(k_2)}$ if and only if $i_1^{(k_1)}$ appears to the right of $i_2^{(k_2)}$ in every word that is commutation equivalent to $\x$. Two words are commutation equivalent if and only if they have the same heap. The Hasse diagram of the heap of $\x$, which we denote by $\Heap(\x)$, is called a \dfn{combinatorial AR quiver} \cite{WhyTheFuss}. 
We will draw $\Heap(\x)$ sideways so that larger letters appear to the left of smaller letters.\footnote{It is typical to draw $\Heap(\x)$ with larger letters to the right of smaller letters, but we have adopted the opposite convention in order to match our equally unorthodox convention for writing words.} 

There is an involution $\psi\colon I\to I$ given by $s_{\psi(i)}=\wo s_i\wo$. We can extend $\psi$ to an involution on the set of words over $I$ by letting $\psi(i_M\cdots i_1)=\psi(i_M)\cdots \psi(i_1)$. 

Recall that the \emph{Coxeter number} of $W$ is the quantity $h=|\Phi|/|S|$. Let $\sfc$ be a reduced word for a Coxeter element $c$. It is known that $\MM(c)\leq h$. In fact, it follows from \cite[Lemma~2.6.5]{WhyTheFuss} that \[\sfc^{h}\equiv\psi(\sort_{\sfc}(\wo))\sort_{\sfc}(\wo),\] where $\sort_{\sfc}(\wo)$ is a special reduced word called the \emph{$\sfc$-sorting word} for $\wo$ (see \cite{ReadingClusters}). One can compute $\sort_{\sfc}(\wo)$ by drawing the combinatorial AR quiver $\Heap(\sfc^h)$ and then ``cutting'' the diagram in half so that the left side can be obtained by applying $\psi$ to the right side. Then $\MM(c)$ is the smallest positive integer $k$ such that the right side fits inside of $\Heap(\sfc^k)$.

\begin{example}\label{exam:quiver}
Suppose that $W=\mathfrak S_7$, and let $s_i=(i\,\,i+1)$. The Coxeter number of $\mathfrak S_7$ is $h=7$. Consider the word $\sfc=213456$, and let $c=s_2s_1s_3s_4s_5s_6$ be the Coxeter element represented by~$\sfc$. \Cref{fig:ARQuiver} shows the combinatorial AR quiver of $\sfc^7$; it is formed by placing $7$ copies of $\Heap(\sfc)$ (indicated by bands shaded in yellow) in a row and adding (black) edges as appropriate. 

The involution $\psi\colon I\to I$ is given by $\psi(i)=7-i$. The thick black-and-blue piecewise linear curve in \cref{fig:ARQuiver} cuts the diagram into two pieces such that the piece on the left side of the cut can be obtained by applying $\psi$ to the piece on the right side of the cut. Thus, the piece on the right side of the cut is $\Heap(\sort_{\sfc}(\wo))$. The piece on the right of the cut uses letters from $5$ different yellow bands, so $\MM(c)=5$. 
\end{example}

\begin{figure}[ht]
  \begin{center}{\includegraphics[width=\linewidth]{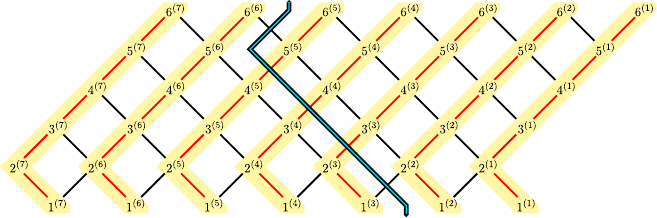}}
  \end{center}
  \caption{An illustration of \Cref{exam:quiver}. The black-and-blue piecewise linear curve cuts $\Heap(\sfc^7)$ into two pieces; the right piece is $\Heap(\sort_{\sfc}(\wo))$, while the left piece is $\Heap(\psi(\sort_{\sfc}(\wo)))$.}\label{fig:ARQuiver}
\end{figure}

\section{Affine Coxeter Groups}\label{sec:affine}
This section discusses general results about Bender--Knuth billiards in affine Coxeter groups. We will use these results to prove \cref{thm:affineAC}, which states that the affine Coxeter groups of types $\widetilde A$, $\widetilde C$, and $\widetilde G_2$ are futuristic. 

\subsection{Affine Coxeter Groups} 
Preserve the notation from \Cref{sec:preliminaries}, and assume in addition that the Coxeter system $(W, S)$ is affine. This means that the associated bilinear form $B \colon V \times V \to \mathbb{R}$ is positive semidefinite but not positive definite. We will also assume that $W$ is irreducible. Let \[V^W = \{\gamma \in V : \text{$u \gamma = \gamma$ for all $u \in W$}\}\] be the space of $W$-invariant vectors in $V$. Equivalently, $V^W$ is the kernel of $B$. The space $V^W$ is $1$-dimensional, and there exists $\delta \in V^W$ whose coefficients in the basis of simple roots are all positive. We call $\delta$ the \dfn{imaginary root}, though the reader should keep in mind that it is not a root of $W$.

Let $\overline{V} = V / V^W$, and let $\quot \colon V \to \overline{V}$ be the quotient map. Then $\overline{V}$ is a $W$-module, and the form $B$ descends to a $W$-invariant real inner product $\overline{B} \colon \overline{V} \times \overline{V} \to \mathbb{R}$.  Unlike $V$, the $W$-module $\overline{V}$ is not faithful; that is, the homomorphism $\rho_{\overline{V}} \colon W \to \mathrm{GL}(\overline{V})$ is not injective. Let $\overline{W}$ be the quotient group $W / \ker(\rho_{\overline{V}})$, and let $\quot \colon W \to \overline{W}$ be the quotient map. (It is a slight abuse of notation to use the letter $\quot$ to refer to two different quotient maps, but it will always be clear from context which is meant.) It is well known that $\overline{W}$ is a finite Coxeter group.

We will make use of the following result due to Speyer. (As noted in Speyer's article, this result was already proved for affine Coxeter groups in \cite{FominZelevinsky, Kuniba}.)
\begin{theorem}[{\cite[Theorem~1]{Speyer}}]\label{thm:power-of-coxeter-element-irreducible}
    Let $W$ be an infinite irreducible Coxeter group, and let $i_1,\ldots,i_n$ be an ordering of the index set $I$. For every $K\geq 1$, the word $(i_n\cdots i_1)^K$ is reduced.
\end{theorem}

To prove that the affine Coxeter groups of type~$\widetilde A$ are futuristic, we will use the following theorem, whose proof is type-uniform. 

\begin{theorem}\label{thm:completely_orthogonal_lemma}
Let $W$ be an irreducible affine Coxeter group. Let $\LL\subseteq W$ be a nonempty convex set, let $i_1,\ldots, i_n$ be an ordering of the index set $I$, and let $u_0,u_1,u_2,\ldots$ be a corresponding periodic billiards trajectory that is contained within a single proper stratum $\Str(R)$. Let $\beta$ be a transmitting root of $\Str(R)$. If $j$ is a positive integer such that $u_j=u_{j-1}$, then $B(u_{j-1}\beta, \alpha_{i_j}) = 0$.  
\end{theorem}

\begin{proof}
It may help to keep the following geometric intuition in mind. As in \Cref{fig:affineS3}, we can think of the Bender--Knuth billiards process as taking place within the Tits cone ${\BB W}$, or, equivalently, in the positive projectivization $\mathbb P({\BB W})=({\BB W}\setminus\{0\})/\mathbb{R}_{>0}$, which has the structure of an affine Euclidean space. By \Cref{lem:acute-angle}, the billiards trajectory can never reflect off of a one-way mirror that forms an obtuse angle with the hyperplane $\HH_\beta$. Thus, if the billiards trajectory is ever moving toward the hyperplane $\HH_\beta$, it can never be redirected to move away from $\HH_\beta$. Since the billiards trajectory is periodic, the only possibility is that it always moves parallel to $\HH_\beta$ and hits only one-way mirrors that are orthogonal to $\HH_\beta$. In other words, if $u_j = u_{j-1}$, then $B(u_{j-1}\beta, \alpha_{i_j}) = 0$.
    
Of course, the billiards trajectory $u_0, u_1, u_2, \ldots$ is really a sequence of elements of $W$, rather than an actual piecewise linear path inside $\mathbb P({\BB W})$. (But see \cref{sec:linear,quest:luminous}.) So, in order to turn our geometric intuition into a rigorous proof, we must formalize the notions of ``moving toward'' and ``moving away from'' $\HH_\beta$. We can do so as follows.

    Consider the Coxeter element $c=s_{i_n}\cdots s_{i_1}$. There is an integer $K\geq 1$ such that ${\Pro_c^K(u_0)=u_0}$ and $\quot(c)^K=\mathbbm{1}$. For $j \geq 1$, define
    \[
    \lambda_j = \begin{cases}
        2B(u_{j-1}\beta, \alpha_{i_j}) & \mbox{if $u_j = u_{j-1}$;} \\
        0 & \mbox{if $u_j = s_{i_j} u_{j-1}$.}
    \end{cases}
    \]
    By \Cref{lem:acute-angle}, we have $\lambda_j \geq 0$ for all $j$, and we wish to prove that $\lambda_j = 0$ for all $j$. Observe that we have defined $\lambda_j$ in such a way that
    \begin{equation}\label{eq:u_j-versus-u_j-1}u_j \beta = s_{i_j} u_{j-1}\beta + \lambda_j \alpha_{i_j}.\end{equation}

    For $j \geq 0$, define \[c_j = s_{i_{j+n}} \cdots s_{i_{j+2}} s_{i_{j+1}}\in W\quad\text{and}\quad \gamma_j = c_j^K u_j \beta - u_j \beta \in V.\] One may think of $c_j^K u_j = s_{i_{j+Kn}} \cdots s_{i_{j+2}}s_{i_{j+1}} u_j$ as the location that the billiards trajectory would reach if it were to start at $u_j$ and then progress $Kn$ steps, in the absence of any one-way mirrors. Roughly, $\gamma_j$ measures the angle between the hyperplane $\HH_\beta$ and the ``limiting direction'' of the billiards trajectory.

    Note that for all integers $j \geq 0$ and all vectors $\epsilon \in V$, we have \[\quot(c_j^K \epsilon - \epsilon) = \quot(c_j^K) \quot(\epsilon) - \quot(\epsilon).\] Since $c_j$ is conjugate to $c$ and we have $\quot(c)^K = \mathbbm{1}$, we must also have $\quot(c_j^K) = \mathbbm{1}$. It follows that $\quot(c_j^K \epsilon - \epsilon) = 0$, so $c_j^K \epsilon - \epsilon \in V^W$. Taking $\epsilon = u_{j} \beta$, we find that $\gamma_j \in V^W$ for all $j$. Taking $\epsilon = \alpha_{i_j}$ yields that $c_j^K \alpha_{i_j} - \alpha_{i_j} \in V^W$ for all $j$.
    
    Next, we will use \eqref{eq:u_j-versus-u_j-1} to relate $\gamma_j$ and $\gamma_{j-1}$. We have
    \begin{align}
        \gamma_j &= c_j^K (s_{i_j} u_{j-1}\beta + \lambda_j \alpha_{i_j}) - (s_{i_j} u_{j-1}\beta + \lambda_j \alpha_{i_j})\nonumber\\
        &= s_{i_j} (c_{j-1}^K u_{j-1} \beta - u_{j-1} \beta) + \lambda_j(c_j^K \alpha_{i_j} - \alpha_{i_j})\nonumber\\
        &= s_{i_j} \gamma_{j-1} + \lambda_j(c_j^K \alpha_{i_j} - \alpha_{i_j})\nonumber\\
        &= \gamma_{j-1} + \lambda_j(c_j^K \alpha_{i_j} - \alpha_{i_j});\label{eq:gamma-j-versus-gamma-j-1}
    \end{align}
    in the second equality, we used the fact that $c_js_{i_j} = s_{i_j} c_{j-1}$, and in the fourth equality, we used the fact that $\gamma_{j-1} \in V^W$.

    By \Cref{thm:power-of-coxeter-element-irreducible}, the word $(i_{j+n} \cdots i_{j+1})^K i_j$ is reduced. Hence, $c_j^K \alpha_{i_j} \in \Phi^+$, and $c_{j}^K \alpha_{i_j} \neq \alpha_{i_j}$ (by \eqref{eq:inversions_roots_referee}). It follows that when $c_{j}^K \alpha_{i_j} - \alpha_{i_j}$ is written in the basis of simple roots, at least one of the coefficients is positive. Since $c_j^K \alpha_{i_j} - \alpha_{i_j} \in V^W$, we find that $c_j^K \alpha_{i_j} - \alpha_{i_j}$ is a positive real multiple of the imaginary root $\delta$. 

    Using \eqref{eq:gamma-j-versus-gamma-j-1} together with the fact that $\lambda_j \geq 0$, we find that $\gamma_j - \gamma_{j-1}$ is a nonnegative real multiple of $\delta$ for all $j$. Since the sequence $\gamma_0, \gamma_1, \gamma_2, \ldots$ is periodic, it must actually be constant. This implies that $\lambda_j = 0$ for all $j$, as desired.
\end{proof}

For $\beta \in \Phi$, let $W^\beta$ be the subgroup of $W$ generated by the set \[\{r_{\beta'} : \text{$\beta' \in \Phi$ and $B(\beta, \beta')=0$}\}.\]

\begin{corollary}\label{thm:power-completely-orthogonal}
    Let $W$ be an irreducible affine Coxeter group, and let $c$ be a Coxeter element of $W$. If $c$ is not futuristic, then there exist an integer $K\geq 1$ and a root $\gamma \in \Phi$ such that $c^K\in W^\gamma$. 
\end{corollary}

\begin{proof}
    Let us write $c = s_{i_n} \cdots s_{i_1}$, where $i_1,\ldots,i_n$ is a fixed ordering of $I$. Suppose that $c$ is not futuristic. That is, there exist a nonempty convex set $\LL \subseteq W$ and an element $u_0 \in W \setminus \LL$ such that the corresponding billiards trajectory $u_0,u_1,u_2,\ldots$ is periodic. By \cref{lem:sep-decreasing}, the billiards trajectory is contained in a single proper stratum $\Str(R)$. 
    
    Choose a positive integer $K$ such that $\Pro_c^K(u_0)=u_0$ and $\quot(c)^K=\mathbbm{1}$. Let $\beta$ be an arbitrary transmitting root of $\Str(R)$, which exists by \Cref{lem:transmitting-wall}. We will prove that
    $c^K\in W^{u_0\beta}$. Clearly, $W^{u_0\beta} = u_0 W^\beta u_0^{-1}$, so it suffices to prove that
    \begin{equation}\label{eq:power-in-subgroup}u_0^{-1}c^K u_0 \in W^{\beta}.\end{equation}

    For each $j \geq 1$, define 
    \[t_j = u_{j}^{-1}s_{i_j} u_{j-1} = \begin{cases}
            r_{u_{j-1}^{-1}\alpha_{i_j}} & \mbox{if $u_j = u_{j-1}$;} \\
            \mathbbm{1} & \mbox{if $u_j = s_{i_j} u_{j-1}$.}
    \end{cases}\]
    In other words, if the billiards trajectory $u_0, u_1, u_2, \ldots$ hits a one-way mirror at step $j$, then $t_j$ is the reflection corresponding to that one-way mirror; otherwise, $t_j$ is the identity. It follows from \cref{thm:completely_orthogonal_lemma} that $t_j\in W^\beta$ for all $j\geq 1$. We have 
    \begin{align*}
        t_{Kn} t_{Kn-1} \cdots t_2 t_1 &= (u_{Kn}^{-1} s_{i_{Kn}} u_{Kn-1})(u_{Kn-1}^{-1} s_{i_{Kn-1}} u_{Kn-2}) \cdots (u_{2}^{-1} s_{i_{2}} u_{1})(u_{1}^{-1} s_{i_{1}} u_{0}) \\
        &= u_{Kn}^{-1} (s_{i_{Kn}} \cdots s_{i_1}) u_0 \\
        &= u_0^{-1} c^K u_0,
    \end{align*} which implies \eqref{eq:power-in-subgroup}.
\end{proof}

In the remainder of this section, we address the affine Coxeter groups of types $\widetilde A$, $\widetilde C$, and $\widetilde G_2$.

\subsection{Type~\texorpdfstring{$\widetilde A$}{\tilde{A}}}\label{subsec:affine_A}

Fix an integer $n \geq 3$. An \dfn{affine permutation} of size $n$ is a bijection $u \colon \mathbb Z \to \mathbb Z$ such that $u(j + n) = u(j) + n$ for all $j \in \mathbb Z$ and $\sum_{j=1}^{n} u(j) = n(n+1)/2$. Let $\widetilde A_{n-1}$ be the group of affine permutations of size $n$ under composition. For $a,b\in\mathbb Z$ with $a\not\equiv b\pmod{n}$, define 
\begin{equation}\label{eq:t}
t_{a, b} = \prod_{r \in \mathbb Z} (a + rn\,\,\,\, b + rn) \in \widetilde A_{n-1}.
\end{equation} Note that $t_{a, b} = t_{b, a} = t_{a + kn, b + kn}$ for all $k \in \mathbb{Z}$. Let $s_i=t_{i, i+1}$, and let $S=\{s_0,\ldots,s_{n-1}\}$. Then $(\widetilde A_{n-1},S)$ is a Coxeter system \cite[Section~8.3]{BjornerBrenti}. The reflections of $\widetilde A_{n-1}$ are precisely the elements $t_{a,b}$ defined in \eqref{eq:t}. The Coxeter graph of $\widetilde A_{n-1}$ is 
\begin{equation}\label{eq:affine_A_graph}	
\begin{array}{l}
\includegraphics[height=1.561cm]{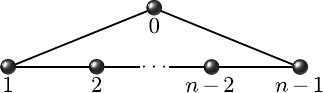}
 \end{array}. 
\end{equation}  

Let $U$ be the real vector space freely generated by $e_1,\ldots,e_{n},\delta$. We can define $e_j$ for all $j\in\mathbb Z$ by taking $e_{i+n}=e_i-\delta$. The root space of $\widetilde A_{n-1}$ is isomorphic to the subspace of $U$ spanned by $\{e_i - e_{i+1} : i \in \mathbb Z\}$. Under this isomorphism, the root system is \[\Phi=\{e_a-e_b:a,b\in\mathbb Z, \,\, a\not\equiv b\pmod{n}\},\] and the reflection corresponding to the root $e_a-e_b$ is $r_{e_a-e_b}=t_{a,b}$. We easily obtain the following lemma.

\begin{lemma}\label{lem:affine_A_completely}
Let $e_a-e_b$ and $e_{a'}-e_{b'}$ be two roots in the root system of type~$\widetilde A_{n-1}$. We have ${B(e_a - e_b, e_{a'} - e_{b'}) = 0}$ if and only if $a,b,a',b'$ belong to distinct residue classes modulo~$n$.
\end{lemma}

\begin{proposition}\label{prop:affine_A}
For $n\geq 3$, the Coxeter group $\widetilde A_{n-1}$ is futuristic. 
\end{proposition}

\begin{proof}
Let us consider acyclic orientations of the $n$-vertex cycle graph $\Gamma_{\widetilde A_{n-1}}$, as in \cref{subsec:conjugate}. For $1\leq d\leq n-1$, let $c_{(d)}$ be the Coxeter element $s_ds_{d+1}\cdots s_{n-1}s_{d-1}s_{d-2}\cdots s_0$. If we draw $\Gamma_{\widetilde A_{n-1}}$ in the plane as in \eqref{eq:affine_A_graph}, then the acyclic orientation $\ao(c_{(d)})$ has exactly $d$ counterclockwise edges. It is a simple exercise to show that two acyclic orientations of $\Gamma_{\widetilde{A}_{n-1}}$ are flip equivalent if and only if they have the same number of edges oriented counterclockwise. It follows that each Coxeter element of $\widetilde A_{n-1}$ is conjugate to exactly one of $c_{(1)},c_{(2)},\ldots,c_{(n-1)}$; hence, according to \cref{prop:conjugate_Coxeter}, we just need to show that each of the $n-1$ Coxeter elements in this list is futuristic. Fix $1\leq d\leq n-1$, and view $c_{(d)}$ as an affine permutation of size $n$. By \Cref{thm:power-completely-orthogonal}, it suffices to prove that $c_{(d)}^K \not \in W^\gamma$ for every integer $K \geq 1$ and every root $\gamma \in \Phi$. Fix such $K$ and $\gamma$, and write $\gamma=e_a-e_b$ for some integers $a$ and $b$ with $a\not\equiv b\pmod n$. According to \cref{lem:affine_A_completely}, the subgroup $W^\gamma$ is generated by reflections $t$ that satisfy $t(a) = a$ and $t(b) = b$. Consequently, we just need to show that $c_{(d)}^K(a)\neq a$. 

For $j\in\mathbb Z$, it is straightforward to compute that \[c_{(d)}(j)=\begin{cases} j-n+d-1 & \mbox{if } j\equiv 1\pmod n; \\ j-1 & \mbox{if } j\equiv 2,3,\ldots,d\pmod n; \\ j+1 & \mbox{if } j\equiv d+1,d+2,\ldots,n-1\pmod n; \\ n+d+1 & \mbox{if } j\equiv n\pmod n. \end{cases}\]
From this, one can readily deduce that $c_{(d)}^K(j)<j$ if $j\equiv 1,2,\ldots,d\pmod n$ and that $c_{(d)}^K(j)>j$ if $j\equiv d+1,d+2,\ldots,n\pmod n$. In particular, $c_{(d)}^K(a)\neq a$.  
\end{proof} 

\begin{remark}\label{rem:affine_A1}
The only affine Coxeter group of type~$\widetilde A$ that \cref{prop:affine_A} fails to address is $\widetilde A_1$, which has the Coxeter graph \[\begin{array}{l}
\includegraphics[height=0.343cm]{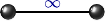}
\end{array}.\]  However, because $\widetilde A_1$ is right-angled, its futuristicity follows from \cref{thm:right-angled}. Alternatively, we could use \cref{prop:folding} and the fact that $\widetilde A_1$ is a folding of $\widetilde C_2$, which we will prove is futuristic in the next subsection. 
\end{remark}

\subsection{Type~\texorpdfstring{$\widetilde C$}{\tilde{C}}}\label{subsec:affine_C}

For $n\geq 3$, the Coxeter graph of the affine Coxeter group $\widetilde C_{n-1}$ is 
\[\begin{array}{l}
\includegraphics[height=0.373cm]{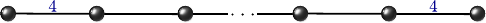}
\end{array},\]
where there are $n$ vertices in total. We can draw the Coxeter graph of $\widetilde A_{2n-3}$ as 
\[\begin{array}{l}
\includegraphics[height=1.4cm]{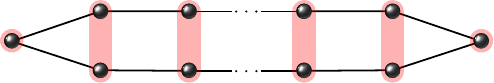}
\end{array},\]
where we have shaded in red the orbits of a Coxeter graph automorphism $\sigma$ of order $2$. Since $\widetilde C_{n-1}$ is obtained from $\widetilde A_{2n-1}$ by folding along $\sigma$, it follows from \cref{prop:folding} that $\widetilde C_{n-1}$ is futuristic.

\subsection{Type~\texorpdfstring{$\widetilde G_2$}{\tilde{G}}}

The affine Coxeter group $\widetilde G_2$ has Coxeter graph 
\[\begin{array}{l}
\includegraphics[height=0.668cm]{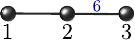}
\end{array}.\] 
The fact that $\widetilde G_2$ is futuristic is a special case of \cref{thm:rank_3}, which will be proved in \cref{sec:rank3}. Combining this fact with \cref{prop:affine_A}, \cref{rem:affine_A1}, and the folding argument in \cref{subsec:affine_C} proves \cref{thm:affineAC}.

\section{The Small-Root Billiards Graph}\label{sec:small}

We now proceed to the proofs of \Cref{thm:complete,thm:rank_3,thm:right-angled}. In each proof, our strategy will be as follows. Let $W$ be a Coxeter group that we wish to prove is futuristic. Let $\LL \subseteq W$ be a nonempty convex set, let $i_1, \ldots, i_n$ be an ordering of $I$, and let $u_0, u_1, u_2, \ldots$ be a corresponding periodic billiards trajectory. Assume for the sake of contradiction that the billiards trajectory is not contained in $\LL$. Then by \Cref{lem:sep-decreasing}, it is contained within a single proper stratum $\Str(R)$. By \Cref{lem:transmitting-wall}, there exists a transmitting root $\beta$ of $\Str(R)$. By \Cref{lem:transmitting-small}, each root $-u_j \beta$ is small. Let $\gamma_j = -u_j \beta$. The sequence $\gamma_0, \gamma_1, \gamma_2, \ldots$ of small roots is highly constrained. First, for each $j \geq 1$, we have $u_j \in \{u_{j-1}, s_{i_j} u_{j-1}\}$, so $\gamma_j \in \{\gamma_{j-1}, s_{i_j} \gamma_{j-1}\}$. Moreover, if there exists a positive root $\gamma' \neq \gamma_{j-1}$ such that $H^+_{\gamma'} \cap H^-_{\gamma_{j-1}} \cap H^+_{\alpha_{i_j}} = \varnothing$, then \cref{lem:super-strong-acute-angle} implies that $u_j = s_{i_j} u_{j-1}$ and (consequently) $\gamma_j = s_{i_j} \gamma_{j-1}$.

These considerations show that the distinct $\gamma_j$'s form a walk on a particular directed graph, which we now define. This graph is similar to the \emph{minimal root reflection table} of Casselman \cite{Casselman}, but it contains some additional information.
\begin{definition}\label{def:small-root-billiards-graph}
    The \dfn{small-root billiards graph} of $W$, denoted ${\bf G}_W$, is the following directed graph in which each edge is labeled by an element of $I$ and is either \dfn{solid} or \dfn{dotted}. The vertex set of the small-root billiards graph is $\Sigma \cup \{\ominus\}$, where $\Sigma$ is the set of small roots and $\ominus$ is a special symbol that represents negative roots. For each $\gamma \in \Sigma$ and $i \in I$ such that $s_i \gamma \in \Sigma$, there is an edge labeled $i$ from $\gamma$ to $s_i \gamma$ (this edge is a loop if $s_i\gamma=\gamma$). The edge from $\gamma$ to $s_i\gamma$ is solid if there exists a positive root $\gamma' \neq \gamma$ such that $H_{\gamma'}^+ \cap H_{\gamma}^- \cap H^+_{\alpha_i} = \varnothing$; otherwise, it is dotted. Additionally, for each $i \in I$, there is a solid edge labeled $i$ from $\alpha_i$ to $\ominus$. Let $\lambda(e)\in I$ denote the label of an edge $e$ of~${\bf G}_W$. 
\end{definition}

\begin{figure}[ht]
 \begin{center}\includegraphics[height=14.651cm]{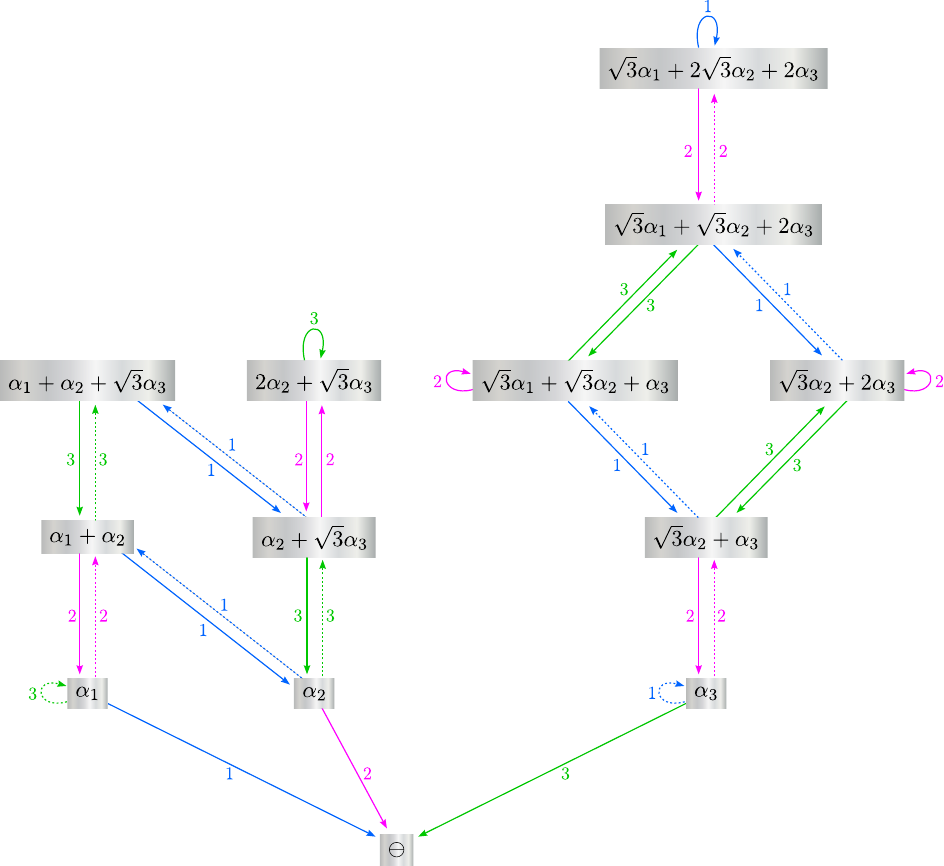}
  \end{center}
\caption{The small-root billiards graph of $\widetilde{G}_2$, whose Coxeter graph is \newline $\begin{array}{l}
\includegraphics[height=0.668cm]{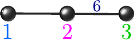}
\end{array}$.}\label{fig:small-root-billiards-graph-example}
\end{figure}

\Cref{fig:small-root-billiards-graph-example} shows the small-root billiards graph of $\widetilde G_2$.

We will always draw the small-root billiards graph with the higher roots nearer the top of the diagram. That is, if $i \in I$ and $\gamma \in \Sigma$ satisfy $s_i \gamma = \gamma - 2B(\gamma, \alpha_i) \alpha_i \in \Sigma$, then we will draw the vertex $s_i \gamma$ above (respectively, below) the vertex $\gamma$ if $B(\gamma, \alpha_i) < 0$ (respectively, $B(\gamma,\alpha_i)>0$). If $\gamma$ is drawn above  $s_i\gamma$, then the edge from $\gamma$ to $s_{i} \gamma$ is solid because we may take $\gamma'=s_i\gamma$. Thus, every edge in ${\bf G}_W$ that points downward is solid. However, it is also possible for a solid edge to point upward.

The following lemma summarizes the discussion thus far. The fact that the sequence $\gamma_0, \gamma_1, \gamma_2, \ldots$ is eventually periodic follows from the fact that $\Sigma$ is finite.

\begin{lemma}\label{lem:small-root-sequence}
    Let $i_1, \ldots, i_n$ be an ordering of $I$, and let $u_0, u_1, u_2, \ldots$ be a corresponding billiards trajectory that is contained in a single proper stratum. Then there is an eventually periodic sequence $\gamma_0, \gamma_1, \gamma_2, \ldots$ of small roots such that the following hold for all $j \geq 1$.
    \begin{enumerate}[(i)]
        \item \label{item:small-root-stay}If $u_j = u_{j-1}$, then $\gamma_j = \gamma_{j-1}$, and there is no solid edge of ${\bf G}_W$ labeled $i_j$ with source $\gamma_{j-1}$.
        \item \label{item:small-root-move}If $u_{j} = s_{i_j} u_{j-1}$, then there is an edge of ${\bf G}_W$ labeled $i_j$ from $\gamma_{j-1}$ to $\gamma_j$. (This edge may be either solid or dotted.)
    \end{enumerate}
    Moreover, if the sequence $u_0, u_1, u_2, \ldots$ is periodic, then so is $\gamma_0, \gamma_1, \gamma_2, \ldots$, and the period of $\gamma_0, \gamma_1, \gamma_2, \ldots$ divides the period of $u_0, u_1, u_2, \ldots$.
\end{lemma}

Our next goal is to reformulate \cref{lem:small-root-sequence} in a way that is easier to visualize. The idea of the reformulation is as follows. If $\gamma_0, \gamma_1, \gamma_2, \ldots$ is a sequence of roots satisfying the conditions of \cref{lem:small-root-sequence}, then we may remove the consecutive duplicates from the sequence $\gamma_0, \gamma_1, \gamma_2, \ldots$ to obtain a walk in the graph ${\bf G}_W$. This walk satisfies certain conditions, which we will describe in \cref{def:closed-billiards-walk} below. First, we must introduce some additional terminology.

Fix an ordering $i_1, \ldots, i_n$ of $I$. As usual, we may define an infinite periodic sequence $i_1, i_2, i_3, \ldots$ by setting $i_{j+n} = i_{j}$ for all $j\geq 1$. For any indices $i, i', i'' \in I$, let us say that $i'$ is \dfn{betwixt} $i$ and $i''$ if for all positive integers $j < j''$ with $i_j = i$ and $i_{j''} = i''$, there exists an integer $j'$ with $j < j' < j''$ and $i_{j'} = i'$. Note that this definition is not symmetric in $i$ and $i''$. For $j, j', j'' \in [n]$, we have that $i_{j'}$ is betwixt $i_{j}$ and $i_{j''}$ if and only if $j < j' < j''$, $j' < j'' \leq j$, or $j'' \leq j < j'$. 

A \dfn{closed walk} in a directed graph ${G}$ is a pair $\mathcal{C} = ((v_1, \ldots, v_\elll), (e_1, \ldots, e_\elll))$, where $\elll > 0$ is an integer, $v_1, \ldots, v_\elll$ are vertices of ${G}$, and for each $j\in[\elll]$, $e_j$ is an edge of ${G}$ from $v_j$ to $v_{j+1}$; here we use the convention $v_{\elll + 1} = v_1$. 

\begin{definition}\label{def:closed-billiards-walk}
    Let $\mathcal{C} = ((v_1, \ldots, v_\elll), (e_1, \ldots, e_\elll))$ be a closed walk in the small-root billiards graph ${\bf G}_W$. We say that $\mathcal{C}$ is \dfn{billiards-plausible} (with respect to the ordering $i_1, \ldots, i_n$ of $I$) if it satisfies the following two properties.
    \begin{enumerate}[(i)]
        \item\label{item:closed-billiards-walk-solid-edge-use} For each $j \in [\elll]$, there is no solid edge with source $v_j$ whose label is betwixt $\lambda(e_{j-1})$ and $\lambda(e_j)$. Here, we use the convention $e_0 = e_\elll$.
        \item\label{item:closed-billiards-walk-identity-product} We have $s_{\lambda(e_1)} \cdots s_{\lambda(e_\elll)} = \mathbbm{1}$.
    \end{enumerate}
\end{definition}

\begin{lemma}\label{lem:closed-billiards-walk-futuristic}
    If there does not exist a closed walk in ${\bf G}_W$ that is billiards-plausible with respect to the ordering $i_1,\ldots,i_n$, then the Coxeter element $c = s_{i_n} \cdots s_{i_1}$ is futuristic. 
\end{lemma}
\begin{proof}
    We prove the contrapositive. Assume that $c$ is not futuristic. Then there exist a nonempty convex subset $\LL\subseteq W$ and a periodic billiards trajectory $u_0, u_1, u_2, \ldots$ that is not contained in $\LL$. Let $K > 0$ be an integer such that $u_0 = u_{Kn}$.

    Let $j_1, \ldots, j_\elll$ be the elements of the set $\{j \in [Kn] : u_j = s_{i_j} u_{j-1}\}$ in increasing order. We have 
    \begin{equation}\label{eq:above_centered}\mathbbm{1} = u_0 u_{Kn}^{-1} = (u_0 u_1^{-1}) (u_1 u_2^{-1}) \cdots (u_{Kn - 2} u_{Kn-1}^{-1})(u_{Kn - 1} u_{Kn}^{-1}).
    \end{equation}
    For each $j \in [Kn]$, the factor $u_{j-1} u_{j}^{-1}$ equals $\mathbbm{1}$ if $u_j = u_{j-1}$ and equals $s_{i_j}$ if $u_j = s_{i_j} u_{j-1}$. Removing the identity terms from the product in \eqref{eq:above_centered} thus yields \begin{equation}\label{eq:simple-reflections-identity-product}s_{i_{j_1}} \cdots s_{i_{j_\elll}} = \mathbbm{1}.\end{equation}

    Now, let $\gamma_0, \gamma_1, \gamma_2, \ldots$ be the sequence from \Cref{lem:small-root-sequence}. For each $k \in [\elll]$, let $e_k$ be the edge of ${\bf G}_W$ labeled $i_{j_k}$ from $\gamma_{j_k - 1}$ to $\gamma_{j_k}$. By \Cref{lem:small-root-sequence}\eqref{item:small-root-stay}, the edges $e_1, \ldots, e_\elll$ form a closed walk $\mathcal{C}$ satisfying \cref{def:closed-billiards-walk}\eqref{item:closed-billiards-walk-solid-edge-use}. It follows from \eqref{eq:simple-reflections-identity-product} that $\mathcal{C}$ also satisfies \cref{def:closed-billiards-walk}\eqref{item:closed-billiards-walk-identity-product}, so $\mathcal{C}$ is billiards-plausible.
\end{proof}

For an example of how \cref{lem:closed-billiards-walk-futuristic} can be used to prove that several Coxeter groups are futuristic, see the proof of \cref{thm:rank_3} in \cref{sec:rank3}.

The proof of \cref{lem:closed-billiards-walk-futuristic} gives a procedure for constructing a billiards-plausible closed walk $\mathcal{C}$ from a periodic billiards trajectory $u_0, u_1, u_2, \ldots$ in a proper stratum. We say that the trajectory $u_0, u_1, u_2, \ldots$ is a \dfn{lift} of $\mathcal{C}$.

It is natural to ask if the converse of \cref{lem:closed-billiards-walk-futuristic} holds. 

\begin{question}\label{conj:closed_walks}
    Suppose ${\bf G}_W$ has a closed walk that is billiards-plausible with respect to the ordering $i_1,\ldots,i_n$. Is the Coxeter element $c = s_{i_n} \cdots s_{i_1}$ necessarily not futuristic?
\end{question}

\begin{lemma}\label{lem:closed-billiards-walk-superfuturistic}
    If there does not exist a closed walk in ${\bf G}_W$ that satisfies \cref{def:closed-billiards-walk}\eqref{item:closed-billiards-walk-solid-edge-use} with respect to the ordering $i_1, \ldots, i_n$, then the Coxeter element $c = s_{i_n} \cdots s_{i_1}$ is superfuturistic.
\end{lemma}
\begin{proof}
    We prove the contrapositive. Assume that $c$ is not superfuturistic. Then there exist a nonempty convex subset $\LL \subseteq W$ and a billiards trajectory $u_0, u_1, u_2, \ldots$ that does not eventually reach $\LL$. By \cref{lem:sep-decreasing}, we have \[\Sep(u_0) \supseteq \Sep(u_1) \supseteq \Sep(u_2) \supseteq \cdots.\] By \cref{lem:sep-finite}, the sequence $\Sep(u_0), \Sep(u_1), \Sep(u_2), \ldots$ is eventually constant. We may assume that this sequence is constant and the billiards trajectory $u_0, u_1, u_2, \ldots$ is contained within a single stratum $\Str(R)$.

    Let $\gamma_0, \gamma_1, \gamma_2, \ldots$ be the sequence from \Cref{lem:small-root-sequence}, and let $K > 0$ be an integer such that $Kn$ is divisible by the period of the sequence $\gamma_0, \gamma_1, \gamma_2, \ldots$. As in the proof of \cref{lem:closed-billiards-walk-futuristic}, let $j_1, \ldots, j_\elll$ be the elements of the set $\{j \in [Kn] : u_j = s_{i_j} u_{j-1}\}$ in increasing order. For each $k \in [\elll]$, let $e_k$ be the edge of ${\bf G}_W$ labeled $j_k$ from $\gamma_{j_k - 1}$ to $\gamma_{j_k}$. By \Cref{lem:small-root-sequence}\eqref{item:small-root-stay}, the edges $e_1, \ldots, e_\elll$ form a closed walk $\mathcal{C}$ satisfying \Cref{def:closed-billiards-walk}\eqref{item:closed-billiards-walk-solid-edge-use}.
\end{proof}

We now prove some lemmas that will help us determine whether the edges of $\mathbf{G}_W$ are solid or dotted.
\begin{lemma}\label{lem:edge-dotted}
    Let $\gamma \in \Sigma$ and $i \in I$. Suppose that when $\gamma$ is written in the basis of the simple roots, the coefficient of $\alpha_i$ is zero. If there is an edge labeled $i$ from $\gamma$ to $s_i \gamma$ in $\mathbf{G}_W$, then it is dotted.
\end{lemma}
\begin{proof}
    Assume for the sake of contradiction that there is a solid edge from $\gamma$ to $s_i \gamma$. Then there is a root $\gamma' \in \Phi^+ \setminus \{\gamma\}$ such that $H_{\gamma'}^+ \cap H_\gamma^- \cap H_{\alpha_i}^+ = \varnothing$.

    Let $t = r_\gamma$ and $t' = r_{\gamma'}$ be the reflections corresponding to $\gamma$ and $\gamma'$, respectively. We have $t \in H_\gamma^-$. Also, $t$ lies in the parabolic subgroup $W_{I \setminus \{i\}}$.  In particular, $t$ does not have $s_i$ as a right inversion, so $t \in H_{\alpha_i}^+$. Since $H_{\gamma'}^+ \cap H_\gamma^- \cap H_{\alpha_i}^+ = \varnothing$, we have $t \not \in H_{\gamma'}^+$. Hence, $t$ must have $t'$ as a right inversion. It follows that $t'$ also lies in the parabolic subgroup $W_{I \setminus \{i\}}$, so $\gamma'$ is a root of $W_{I \setminus \{i\}}$.

    The fact that $H_{\gamma'}^+ \cap H_\gamma^- \cap H_{\alpha_i}^+ = \varnothing$ tells us that $H_{\gamma'}^+\cap H_{\alpha_i}^+\subseteq H_\gamma^+$. Because $W_{I \setminus \{i\}}\subseteq H_{\alpha_i}^+$, we have \[H_{\gamma'}^+ \cap W_{I \setminus \{i\}} \subseteq H_\gamma^+ \cap W_{I \setminus \{i\}}.\] But $H_{\gamma'}^+ \cap W_{I \setminus \{i\}}$ and $H_\gamma^+ \cap W_{I \setminus \{i\}}$ are the half-spaces of $W_{I \setminus \{i\}}$ corresponding to $\gamma'$ and $\gamma$, respectively. Hence, $\gamma$ is not a small root of $W_{I \setminus \{i\}}$. This contradicts \cref{cor:small-roots-of-parabolic-subgroup}.
\end{proof}

\begin{lemma}\label{lem:edge-solid}
    Let $\gamma \in \Sigma$ and $i \in I$. Suppose that there exist real numbers $a, a' > 0$ such that $a \gamma - a' \alpha_i \in \Phi$. If there is an edge from $\gamma$ to $s_i \gamma$ in $\mathbf{G}_W$, then it is solid.
\end{lemma}
\begin{proof}
    Let $\gamma' = a \gamma - a' \alpha_i$. Then $H_{\gamma'}^+ \cap H_\gamma^- \cap H_{\alpha_i}^+ = \varnothing$, so the edge from $\gamma$ to $s_i \gamma$ is solid.
\end{proof}

\subsection{Small Roots of Rank-2 Parabolic Subgroups}
The proofs of \Cref{thm:right-angled,thm:complete,thm:rank_3} will involve explicitly computing ${\bf G}_W$ for various Coxeter groups $W$ and applying \Cref{lem:small-root-sequence} or \Cref{lem:closed-billiards-walk-futuristic}. In this subsection, we will show that ${\bf G}_W$ always has a certain easily described induced subgraph.

Let $i, i' \in I$ be distinct indices such that $m(i, i') < \infty$, and let $m = m(i, i')$. 
It follows from a routine computation that the roots of the standard rank-$2$ parabolic subgroup $W_{\{i, i'\}}$ are the vectors 
\[\alpha_{i, i'}(k) = \frac{\sin((k+1) \pi/m)}{\sin(\pi/m)} \alpha_i + \frac{\sin(k\pi/m)}{\sin(\pi/m)} \alpha_{i'} \in V,\]
where $k$ ranges over the integers.  The vector $\alpha_{i, i'}(k)$ depends only on the residue class of $k$ modulo $2m$, and distinct residue classes yield distinct vectors $\alpha_{i, i'}(k)$, so the number of roots of $W_{\{i,i'\}}$ is $2m$. Geometrically, the restriction of $B$ to the $2$-dimensional space $V_{\{i,i'\}} = \operatorname{span}\{\alpha_i, \alpha_{i'}\} \subseteq V$ is a real inner product, and $\alpha_{i, i'}(0), \alpha_{i, i'}(1), \ldots, \alpha_{i, i'}(2m-1)$ are $2m$ equally spaced vectors (in that order) on the unit circle of $V_{\{i,i'\}}$. The simple roots of $W_{\{i,i'\}}$ are $\alpha_{i, i'}(0) = \alpha_{i}$ and $\alpha_{i, i'}(m-1) = \alpha_{i'}$. The action of $W_{\{i,i'\}}$ on the roots is given by $s_i \alpha_{i, i'}(k) = \alpha_{i, i'}(m - k)$ and $s_{i'} \alpha_{i, i'}(k) = \alpha_{i, i'}(m-2 - k)$.

Let $\Sigma_{\leq 2}$ denote the set 
\begin{equation}\label{eq:sigma-leq-2}\{\alpha_i : i \in I\} \cup \{\alpha_{i, i'}(k) : \text{$i, i' \in I$ are distinct},\,\, m(i, i') < \infty,\,\, 0 < k < m(i, i')-1\}.\end{equation}
Note that since $\alpha_{i, i'}(k) = \alpha_{i', i}(m(i, i') - 1 - k)$, we have listed each element of $\Sigma_{\leq 2} \setminus \{\alpha_i : i \in I\}$ twice in \eqref{eq:sigma-leq-2}. Let ${\bf G}_{W, \leq 2}$ denote the induced subgraph of ${\bf G}_W$ on the vertex set $\Sigma_{\leq 2} \cup \{\ominus\}$. 

\begin{lemma}\label{lem:rank-2-small-root-graph}
    The set $\Sigma_{\leq 2}$ consists of all small roots that can be written as linear combinations of at most two simple roots. Moreover, the edges of ${\bf G}_{W, \leq 2}$ are as follows. 
    \begin{enumerate}[(i)]
        \item For each pair $i, i' \in I$ of distinct indices with $m(i, i') < \infty$ and for each integer $k$ satisfying $0 < k < m(i, i') -1$, there is a solid edge labeled $i$ from $\alpha_{i, i'}(k)$ to $\alpha_{i, i'}(m(i, i')-k)$.
        \item For each pair $i, i' \in I$ of distinct indices with $m(i, i') < \infty$, there is a dotted edge labeled $i$ from $\alpha_{i'}$ to $\alpha_{i, i'}(1)$. (If $m(i, i') = 2$, then this edge is a loop.)
        \item For each triple $i, i',i'' \in I$ of distinct indices with $2 = m(i, i') = m(i, i'') < m(i', i'') < \infty$ and for each integer $k$ satisfying $0 < k < m(i', i'') -1$, there is a dotted edge labeled $i$ from $\alpha_{i', i''}(k)$ to itself.
        \item For each $i \in I$, there is a solid edge labeled $i$ from $\alpha_i$ to $\ominus$.
    \end{enumerate}
\end{lemma}
\begin{proof}
    The first statement in the lemma follows from \cref{cor:small-roots-of-parabolic-subgroup}.

    Let $\gamma \in \Sigma_{\leq 2}$ and $i \in I$. Upon inspection, we find that $s_i \gamma \in \Sigma_{\leq 2}$ if and only if one of the following conditions holds:
    \begin{enumerate}[(i)]
        \item \label{item:solid-dihedral} We have $\gamma = \alpha_{i, i'}(k)$, where $i' \in I$ satisfies $m(i, i') < \infty$ and $k$ is an integer such that $0 < k < m(i, i') - 1$. In this case, $s_i \gamma = \alpha_{i, i'}(m(i, i') - k)$.
        \item \label{item:dotted-dihedral} We have $\gamma = \alpha_{i'}$, where $i' \in I$ satisfies $m(i, i') < \infty$. In this case, $s_i \gamma = \alpha_{i, i'}(1)$. 
        \item \label{item:additional-edges} We have $\gamma = \alpha_{i', i''}(k)$, where $i', i'' \in I$ satisfy $2 = m(i, i') = m(i, i'') < m(i', i'') < \infty$ and $k$ is an integer such that $0 < k < m(i', i'') - 1$. In this case, $s_i \gamma = \gamma$.
    \end{enumerate}
    Together with the edge from $\alpha_i$ to $\ominus$ for each $i$, this completes the description of the set of edges of ${\bf G}_{W, \leq 2}$. We use \cref{lem:edge-dotted,lem:edge-solid} to determine which edges are solid and which edges are dotted.
\end{proof}

Let us use \cref{lem:rank-2-small-root-graph} to briefly discuss the graph structure of ${\bf G}_{W, \leq 2}$, ignoring for now whether edges are solid or dotted. As mentioned above, $\alpha_{i, i'}(k) = \alpha_{i', i}(m(i, i') - 1 - k)$. Let $i, i' \in I$ be distinct indices.  If $m = m(i, i')$ is an odd integer, then ${\bf G}_{W, \leq 2}$ has a bidirectional path \[\alpha_i = \alpha_{i, i'}(0) \underset{i'}{\longleftrightarrow} \alpha_{i, i'}(m - 2) \underset{i}{\longleftrightarrow} \alpha_{i, i'}(2) \underset{i'}{\longleftrightarrow} \cdots \underset{i'}{\longleftrightarrow} \alpha_{i, i'}(1) \underset{i}{\longleftrightarrow} \alpha_{i, i'}(m - 1) = \alpha_{i'}\] with $m$ vertices and $m - 1$ edges in each direction. If instead $m = m(i, i')$ is an even integer, then we obtain the two disjoint bidirectional paths
\begin{align*}
    \alpha_i &= \alpha_{i, i'}(0) \underset{i'}{\longleftrightarrow} \alpha_{i, i'}(m - 2) \underset{i}{\longleftrightarrow} \alpha_{i, i'}(2) \underset{i'}{\longleftrightarrow} \cdots \longleftrightarrow \alpha_{i, i'}(2 \lfloor m / 4 \rfloor) \quad \text{and} \\
    \alpha_{i'} &= \alpha_{i, i'}(m - 1) \underset{i}{\longleftrightarrow} \alpha_{i, i'}(1) \underset{i'}{\longleftrightarrow} \alpha_{i, i'}(m-3) \underset{i}{\longleftrightarrow} \cdots \longleftrightarrow \alpha_{i, i'}(2 \lfloor (m - 2) / 4 \rfloor + 1),
\end{align*}
each of which has $m / 2$ vertices and $m / 2 - 1$ edges in each direction; there are also loops at the vertices $\alpha_{i, i'}(2 \lfloor m / 4 \rfloor)$ and $\alpha_{i, i'}(2 \lfloor (m - 2) / 4 \rfloor + 1)$. (The labels of these loops and of the rightmost edges in these bidirectional paths depend on the parity of $m / 2$.) Regardless of the parity of $m$, each vertex in one of these paths also has a loop labeled $i''$ for each index $i'' \in I$ such that ${m(i, i'') = m(i', i'') = 2}$.

\section{Right-Angled Coxeter Groups}\label{sec:right-angled} 
The Coxeter group $W$ is called \dfn{right-angled} if $m(i,i')\in\{2,\infty\}$ for all distinct $i,i'\in I$. We now prove \cref{thm:right-angled}, which states that every right-angled Coxeter group is superfuturistic. 

\begin{proof}[Proof of \cref{thm:right-angled}]
Let $W$ be a right-angled Coxeter group. By \Cref{prop:small-roots-description}, the set $\Sigma$ of small roots of $W$ is equal to the set $\{\alpha_i : i \in I\}$ of simple roots. In the small-root billiards graph ${\bf G}_W$, there is a dotted edge labeled $i$ from $\alpha_{i'}$ to itself for all $i, i' \in I$ with $m(i, i') = 2$. There is also a solid edge labeled $i$ from $\alpha_i$ to $\ominus$ for all $i \in I$. There are no other edges. See \Cref{fig:small-root-billiards-graph-right-angled} for an example.
\begin{figure}
\begin{center}
\includegraphics[height=4cm]{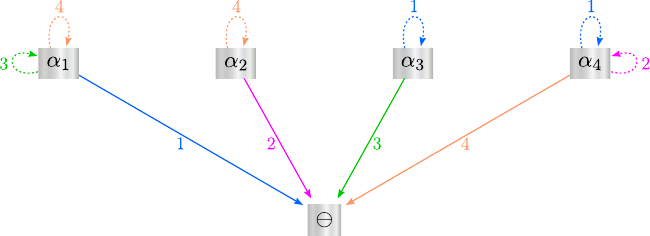}
\end{center}
\caption{The small-root billiards graph of the right-angled Coxeter group whose Coxeter graph is $\begin{array}{l}\includegraphics[height=0.639cm]{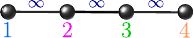}\end{array}$.\label{fig:small-root-billiards-graph-right-angled}}
\end{figure}
The fact that $W$ is superfuturistic now follows directly from \cref{lem:closed-billiards-walk-superfuturistic}.
\end{proof}

\section{Coxeter Groups with Complete Coxeter Graphs}\label{sec:complete} 
\Cref{thm:complete} states that if the Coxeter graph of a Coxeter group $W$ is complete, then $W$ is futuristic. To prove this, we will show that the small-root billiards graph ${\bf G}_{W}$ is actually equal to the graph ${\bf G}_{W, \leq 2}$ that we described in \Cref{lem:rank-2-small-root-graph} and the subsequent paragraphs. See \Cref{fig:complete-coxeter-graph-small-root-billiards-graph} for an example.
\begin{lemma}\label{lem:complete_small_root_graph}
    If the Coxeter graph of $W$ is complete, then $\Sigma_{\leq 2} = \Sigma$.
\end{lemma}
\begin{proof}
    By \Cref{prop:small-roots-description}, it suffices to prove that for all $i \in I$ and all $\gamma \in \Sigma_{\leq 2} \setminus \{\alpha_i\}$, we have either $s_i \gamma \in \Sigma_{\leq 2}$ or $B(\gamma, \alpha_i) \leq -1$. For this, there are two inequalities to check. First, we must show that for all $i, i' \in I$ with $m(i, i') = \infty$, we have $B(\alpha_i, \alpha_{i'}) \leq -1$. Second, we must show that for all distinct $i, i', i'' \in I$ with $m(i, i') < \infty$ and for all $k$ with $0 < k < m(i, i') - 1$, we have $B(\alpha_{i, i'}(k), \alpha_{i''}) \leq -1$. The first inequality follows from the definition of the bilinear form $B$. For the second, we have
    \begin{align*}
    B(\alpha_{i, i'}(k), \alpha_{i''}) &= -\frac{\sin((k+1)\pi/m(i, i'))}{\sin(\pi/m(i, i'))} \cos(\pi / m(i, i'')) - \frac{\sin(k\pi/m(i, i'))}{\sin(\pi/m(i, i'))} \cos(\pi / m(i', i'')) \\
    &\leq -\cos(\pi / m(i, i'')) - \cos(\pi/m(i', i'')) \\
    &\leq -1,
    \end{align*}
    where we have used the inequalities $0 < k < m(i, i') - 1$ and $m(i, i''), m(i', i'') \geq 3$.
\end{proof}

\begin{figure}
\begin{center}
\includegraphics[height=12.195cm]{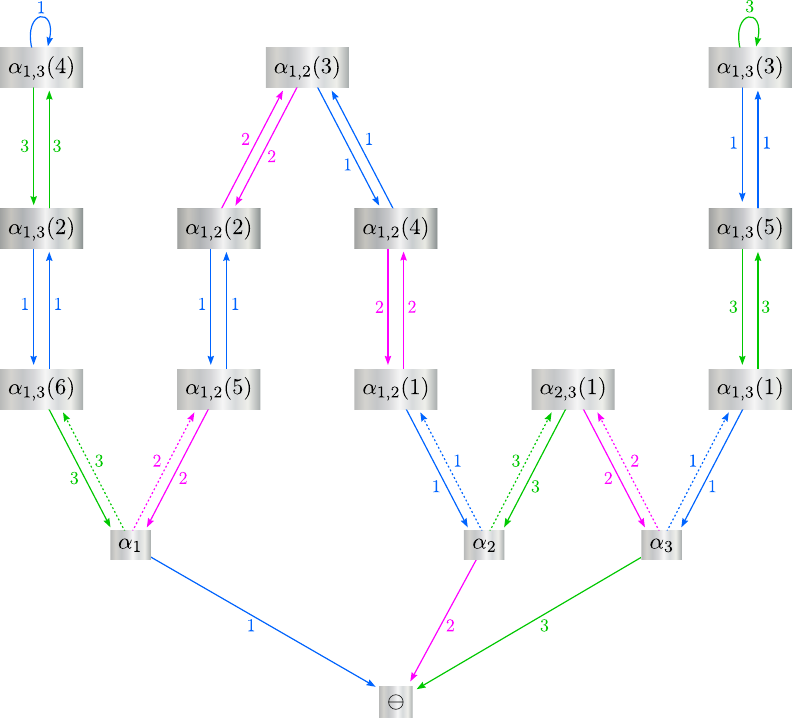}
\end{center}
\caption{The small-root billiards graph of the Coxeter group whose Coxeter graph is $\begin{array}{l}\includegraphics[height=1.177cm]{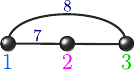}\end{array}$.\label{fig:complete-coxeter-graph-small-root-billiards-graph}}
\end{figure}

To prove \Cref{thm:complete}, all that remains is to carefully apply \Cref{lem:closed-billiards-walk-futuristic}.

\begin{proof}[Proof of \Cref{thm:complete}]
    Let $W$ be a Coxeter group with a complete Coxeter graph. Let $i_1, \ldots, i_n$ be an ordering of $I$, and assume for the sake of contradiction that $c = s_{i_n} \cdots s_{i_1}$ is not futuristic. By \Cref{lem:closed-billiards-walk-futuristic}, there is a billiards-plausible closed walk $\mathcal{C} = ((v_1, \ldots, v_\elll), (e_1, \ldots, e_\elll))$ in ${\bf G}_W$.
    
    By \Cref{def:closed-billiards-walk}\eqref{item:closed-billiards-walk-identity-product}, we have $s_{\lambda(e_1)} \cdots s_{\lambda(e_\elll)} = \mathbbm{1}$. Therefore, by Matsumoto's theorem (\cref{thm:matsumoto}), it is possible to apply a nil move or a braid move to the word $\lambda(e_1) \cdots \lambda(e_\elll)$.

    Suppose first that it is possible to apply a nil move to the word $\lambda(e_1) \cdots \lambda(e_\elll)$. Then there are two consecutive edges $e_{j-1}, e_j$ of the walk $\mathcal{C}$ such that $\lambda(e_{j-1}) = \lambda(e_j)$. Every element of $I \setminus \{\lambda(e_j)\}$ is betwixt $\lambda(e_{j-1})$ and $\lambda(e_j)$. Therefore, by \Cref{def:closed-billiards-walk}\eqref{item:closed-billiards-walk-solid-edge-use}, there are no solid edges in ${\bf G}_W$ with source $v_j$, except possibly $e_j$. 

    By \cref{lem:complete_small_root_graph}, either $v_j = \alpha_{i}$ for some $i \in I$, or $v_j = \alpha_{i, i'}(k)$, where $i, i' \in I$ satisfy $m(i, i') < \infty$ and $k$ is an integer such that $0 < k < m(i, i') - 1$. If $v_j = \alpha_i$ for some $i \in I$, then there is a solid edge from $v_j$ to $\ominus$. If instead $v_j = \alpha_{i, i'}(k)$, then by \cref{lem:rank-2-small-root-graph}, there are two solid edges with source $v_j$: one labeled $i$ and one labeled $i'$. In either case, we obtain a contradiction with the fact that $e_j$ is the only possible solid edge in ${\bf G}_W$ with source $v_j$. 

    Now suppose that it is possible to apply a braid move to the word $\lambda(e_1) \cdots \lambda(e_\elll)$. Then there exist indices $i, i' \in I$ such that the closed walk $\mathcal{C}$ has $m(i, i')$ consecutive edges whose labels alternate between $i$ and $i'$. However, by \cref{lem:rank-2-small-root-graph,lem:complete_small_root_graph}, any walk in ${\bf G}_W$ that does not use the vertex $\ominus$ and whose edge labels alternate between $i$ and $i'$ must have at most $m(i, i') - 1$ edges. Again, we obtain a contradiction. 
\end{proof}

\section{Coxeter Groups of Rank at Most 3}\label{sec:rank3}
We now have all the tools we need to prove \Cref{thm:rank_3}, which states that all Coxeter groups of rank at most 3 are futuristic.

\begin{proof}[Proof of \cref{thm:rank_3}]
    All Coxeter groups of rank at most 2 are futuristic by \Cref{thm:finite_futuristic,thm:right-angled}. Let $W$ be a Coxeter group of rank $3$. We wish to show that $W$ is futuristic. Without loss of generality, we may assume that the index set $I$ is $\{1, 2, 3\}$. Let $p = m(1, 3)$, $q = m(1, 2)$, and $r = m(2, 3)$, and assume without loss of generality that $p \leq q \leq r$. We now consider several cases depending on the values of $p$, $q$, and $r$.

\medskip

    \noindent {\bf Case 1.} ($p = q = 2$)
    
    In this case, $W$ is reducible, and its irreducible factors are futuristic because they each have rank at most $2$. Therefore, $W$ is futuristic (see \cref{rem:irreducible_futuristic}). 

\medskip

    \noindent {\bf Case 2.} ($p = 2$, $q = 3$, and $r \leq 5$)
    
    In this case, it follows from the classification of finite Coxeter groups \cite[Appendix~A1]{BjornerBrenti} that $W$ is finite, so it is futuristic by \Cref{thm:finite_futuristic}.

\medskip

    \noindent {\bf Case 3.} ($p = 2$, $q = 3$, and $r = 6$)

    In this case, $W$ is the affine Coxeter group $\widetilde G_2$. The small-root billiards graph ${\mathbf G}_W$ is depicted in \cref{fig:small-root-billiards-graph-example}. We will now show that $c = s_{i_3}s_{i_2}s_{i_1}$ is futuristic, where $i_1 = 1$, $i_2 = 2$, and $i_3 = 3$. By inspection, every closed walk in ${\bf G}_W$ that satisfies \cref{def:closed-billiards-walk}\eqref{item:closed-billiards-walk-solid-edge-use} uses the sequence of vertices
    \[\sqrt{3} \alpha_2 + \alpha_3,\,\sqrt{3} \alpha_1 + \sqrt{3} \alpha_2 + \alpha_3,\,\sqrt{3} \alpha_1 + \sqrt{3} \alpha_2 + \alpha_3,\,\sqrt{3} \alpha_1 + \sqrt{3} \alpha_2 + 2\alpha_3,\,\sqrt{3} \alpha_2 + 2 \alpha_3,\, \sqrt{3}\alpha_2 + 2 \alpha_3,\]
    cyclically shifted and repeated $K$ times for some $K > 0$. Such a walk cannot satisfy \cref{def:closed-billiards-walk}\eqref{item:closed-billiards-walk-identity-product}. Indeed, using Matsumoto's theorem (\cref{thm:matsumoto}), one can check that the word $(123123)^K$ is reduced, so \[(s_1 s_2 s_3 s_1 s_2 s_3)^K \neq \mathbbm{1}.\] 
    By \cref{lem:closed-billiards-walk-futuristic}, $s_3s_2s_1$ is futuristic. By \cref{cor:tree}, $W$ is futuristic.

    \medskip
    
    \noindent {\bf Case 4.} ($p = 2$, $q = 3$, and $7 \leq r < \infty$)
\begin{figure}
\begin{center}
\includegraphics[height=11.384cm]{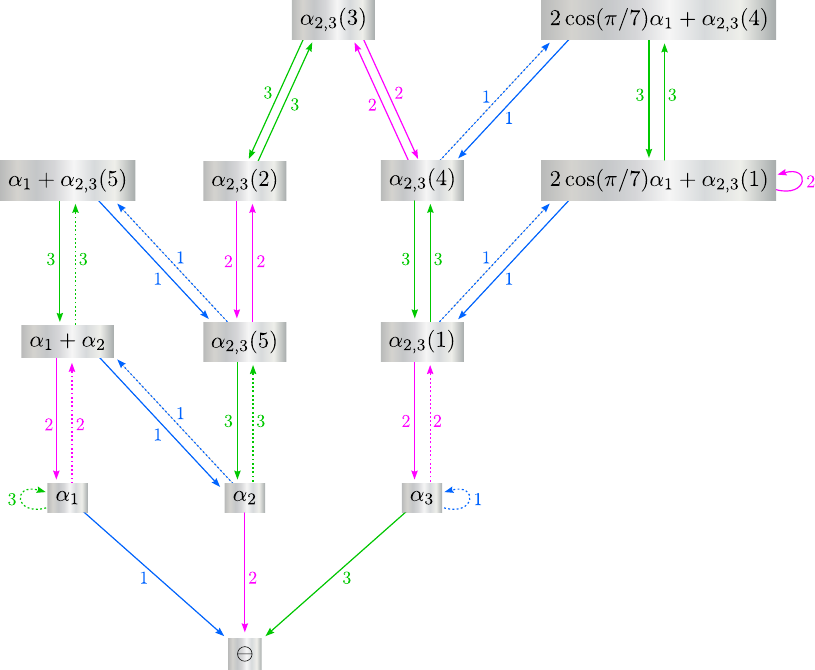}
\end{center}
\caption{The small-root billiards graph of the Coxeter group whose Coxeter graph is $\begin{array}{l} \includegraphics[height=0.668cm]{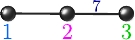}\end{array}$.}\label{fig:small-root-billiards-graph-237}
\end{figure}

\begin{figure}
\begin{center}
\includegraphics[height=11.384cm]{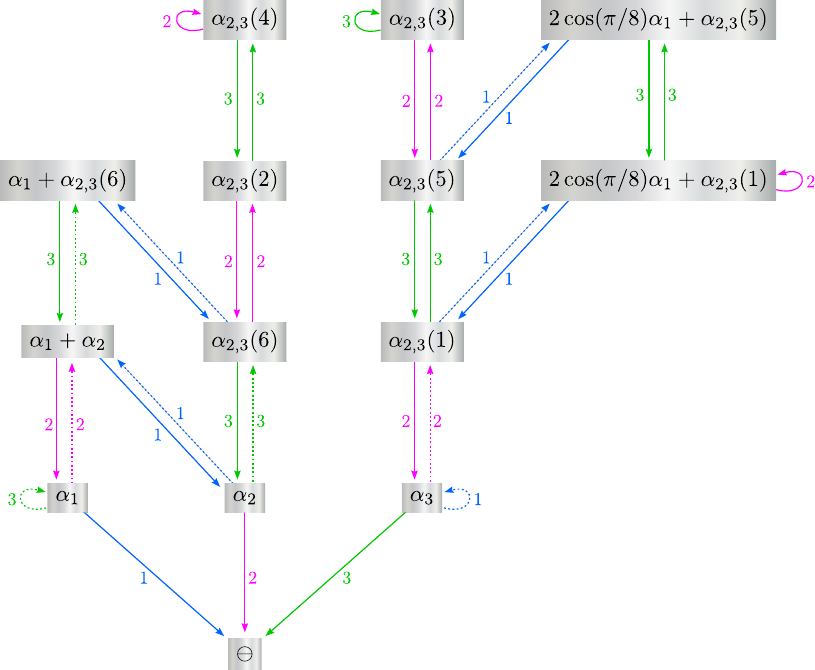}
\end{center}
\caption{The small-root billiards graph of the Coxeter group whose Coxeter graph is $\begin{array}{l} \includegraphics[height=0.668cm]{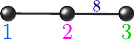}\end{array}$.}\label{fig:small-root-billiards-graph-238}
\end{figure}
    
We claim that
    \[
        \Sigma = \Sigma_{\leq 2} \cup \{\alpha_1 + \alpha_{2, 3}(r-2),\,\,2\cos(\pi/r) \alpha_1 + \alpha_{2, 3}(1),\,\,2\cos(\pi/r)\alpha_1 + \alpha_{2, 3}(r-3)\}.
    \]
    Moreover, we claim that ${\bf G}_W$ has 11 edges aside from the ones in ${\bf G}_{W, \leq 2}$:
    \begin{itemize}
        \item a solid edge labeled $3$ from $\alpha_1 + \alpha_{2, 3}(r-2)$ to $\alpha_1 + \alpha_2 = \alpha_{1, 2}(1)$ and a dotted edge labeled $3$ in the other direction;
        \item a solid edge labeled $1$ from $\alpha_1 + \alpha_{2, 3}(r-2)$ to $\alpha_{2, 3}(r-2)$ and a dotted edge labeled $1$ in the other direction;
        \item a solid edge labeled $1$ from $2\cos(\pi/r) \alpha_1 + \alpha_{2, 3}(1)$ to $\alpha_{2, 3}(1)$ and a dotted edge labeled $1$ in the other direction;
        \item a solid edge labeled $1$ from $2\cos(\pi/r) \alpha_1 + \alpha_{2, 3}(r-3)$ to $\alpha_{2, 3}(r-3)$ and a dotted edge labeled $1$ in the other direction;
        \item a solid edge labeled $3$ from $2\cos(\pi/r) \alpha_1 + \alpha_{2, 3}(r-3)$ to $2\cos(\pi/r) \alpha_1 + \alpha_{2, 3}(1)$ and another solid edge labeled $3$ in the other direction;
        \item a solid loop labeled $2$ from $2\cos(\pi/r) \alpha_1 + \alpha_{2, 3}(1)$ to itself.
    \end{itemize}
    See \cref{fig:small-root-billiards-graph-237,fig:small-root-billiards-graph-238} for examples. By the discussion after the proof of \cref{lem:rank-2-small-root-graph}, the graph will look substantially different depending on the parity of $r$.

    To verify our claim, we use \cref{prop:small-roots-description}. We have
    \begin{alignat*}{3}
        \alpha_1 + \alpha_{2, 3}(r-2) &= s_1 \alpha_{2, 3}(r-2) \qquad && \text{with $B(\alpha_{2, 3}(r-2), \alpha_1) = -\frac{1}{2} > -1$,}\\
        2 \cos(\pi/r) \alpha_1 + \alpha_{2, 3}(1) &= s_1 \alpha_{2, 3}(1) \qquad && \text{with $B(\alpha_{2, 3}(1), \alpha_1) = -\cos(\pi/r) > -1,$}\\
        2 \cos(\pi/r) \alpha_1 + \alpha_{2, 3}(r-3) &= s_1 \alpha_{2, 3}(r-3) \qquad && \text{with $B(\alpha_{2, 3}(r-3), \alpha_1) = -\cos(\pi/r) > -1$},
    \end{alignat*}
    so $\alpha_1 + \alpha_{2, 3}(r-2)$, $2\cos(\pi/r) \alpha_1 + \alpha_{2, 3}(1)$, and $2\cos(\pi/r)\alpha_1 + \alpha_{2, 3}(r-3)$ are small roots. To find the edges of ${\bf G}_W$ incident to these three vertices, we apply each of the simple reflections $s_1$, $s_2$, and $s_3$ to each of the roots $\alpha_1 + \alpha_{2, 3}(r-2)$, $2\cos(\pi/r) \alpha_1 + \alpha_{2, 3}(1)$, and $2\cos(\pi/r)\alpha_1 + \alpha_{2, 3}(r-3)$. We omit the details; the only nontrivial verification is that the edge labeled $2$ from $2\cos(\pi/r) \alpha_1 + \alpha_{2, 3}(1)$ to itself is solid, which follows from the fact that 
    \begin{align*}
    H_{\alpha_1}^+ \cap H_{2\cos(\pi/r) \alpha_1 + \alpha_{2, 3}(1)}^- \cap H^+_{\alpha_2} &= H_{s_1s_2s_1\alpha_2}^- \cap H_{s_1s_2s_1\alpha_3}^- \cap H^-_{s_1s_2s_1\alpha_1} \\
    &= (H_{\alpha_2}^- \cap H_{\alpha_3}^- \cap H_{\alpha_1}^-)s_1s_2s_1\\
    &= \varnothing.
    \end{align*}
    
    We must show that these are the only small roots. That is, we must show that $\Sigma = \Sigma'$, where \[\Sigma' = \Sigma_{\leq 2} \cup \{\alpha_1 + \alpha_{2, 3}(r-2),\,\,2\cos(\pi/r) \alpha_1 + \alpha_{2, 3}(1),\,\,2\cos(\pi/r)\alpha_1 + \alpha_{2, 3}(r-3)\}.\] It suffices to prove that for each $i \in I$ and $\gamma \in \Sigma' \setminus \{\alpha_i\}$, we have either $s_i \gamma \in \Sigma'$ or $B(\gamma, \alpha_i) \leq -1$. This reduces to proving three inequalities:
    \begin{align*}
        B(\alpha_1 + \alpha_{2, 3}(r-2), \alpha_2) &\leq -1; \\
        B(2\cos(\pi/r)\alpha_1 + \alpha_{2, 3}(r-3), \alpha_2) &\leq -1; \\
        B(\alpha_{2, 3}(k), \alpha_1) &\leq -1 \qquad \text{for $2 \leq k \leq r - 4$.}
    \end{align*}
    Each of these inequalities is easy to verify from the definition of the bilinear form $B$.
    
    We will now show that $c = s_{i_3}s_{i_2}s_{i_1}$ is futuristic, where $i_1 = 1$, $i_2 = 2$, and $i_3 = 3$. By \cref{lem:closed-billiards-walk-futuristic}, it suffices to show that ${\bf G}_W$ has no billiards-plausible closed walk.

    On one hand, we claim that if $r$ is odd, then no closed walk in ${\bf G}_W$ can even satisfy \cref{def:closed-billiards-walk}\eqref{item:closed-billiards-walk-solid-edge-use}. Indeed, if a walk satisfying \cref{def:closed-billiards-walk}\eqref{item:closed-billiards-walk-solid-edge-use} enters one end of the bidirectional path between $\alpha_{2,3}(r-2)$ and $\alpha_{2,3}(r-3)$, then it must traverse the entire path and come out the other end because of all the solid edges labeled $2$ and $3$. It follows that the length of this path does not affect whether or not there exists a closed walk in ${\bf G}_W$ satisfying \cref{def:closed-billiards-walk}\eqref{item:closed-billiards-walk-solid-edge-use}. Verifying the claim is now a matter of checking finitely many possibilities. It follows that if $r$ is odd, then $s_3s_2s_1$ is futuristic.    
    
    On the other hand, if $r$ is even, then a similar argument (to reduce to checking finitely many possibilities) implies that every closed walk in ${\bf G}_W$ satisfying \cref{def:closed-billiards-walk}\eqref{item:closed-billiards-walk-solid-edge-use} must use the sequence of vertices  
    \begin{multline*}
    \alpha_{2, 3}(1),\,\, 2\cos(\pi/r)\alpha_1 + \alpha_{2, 3}(1),\,\, 2\cos(\pi/r)\alpha_1 + \alpha_{2, 3}(1),\,\, 2\cos(\pi/r)\alpha_1 + \alpha_{2, 3}(r-3), \\
    \alpha_{2, 3}(r-3), \alpha_{2, 3}(3), \alpha_{2, 3}(r-5), \alpha_{2, 3}(5), \ldots, \alpha_{2, 3}(5), \alpha_{2, 3}(r-5), \alpha_{2, 3}(3), \alpha_{2, 3}(r-3),
    \end{multline*}
    cyclically shifted and repeated $K$ times for some $K \geq 1$. Such a walk cannot satisfy \cref{def:closed-billiards-walk}\eqref{item:closed-billiards-walk-identity-product}. Indeed, using Matsumoto's theorem (\cref{thm:matsumoto}), one can check that the word $(1231(23)^{r/2-2})^K$ is reduced, so \[(s_1 s_2 s_3 s_1 (s_2 s_3)^{r/2 - 2})^K \neq \mathbbm{1}.\] 
    It follows that if $r$ is even, then $s_3s_2s_1$ is futuristic. By \cref{cor:tree}, $W$ is futuristic.
    \medskip

    \noindent {\bf Case 5.} ($p = 2$ and $4 \leq q \leq r < \infty$) 

    \begin{figure}
\begin{center}
\includegraphics[height=8.606cm]{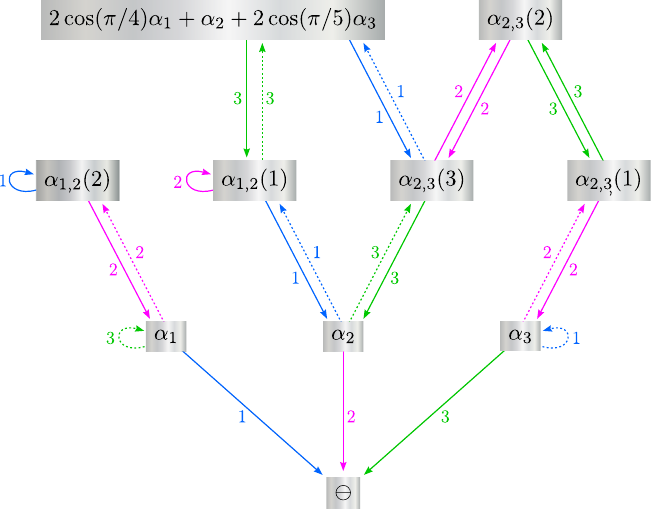}
\end{center}
\caption{The small-root billiards graph of the Coxeter group whose Coxeter graph is $\begin{array}{l} \includegraphics[height=0.668cm]{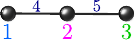}\end{array}$.}\label{fig:small-root-billiards-graph-245}
\end{figure}
    
    We claim that $\Sigma = \Sigma_{\leq 2} \cup \{2 \cos(\pi/q) \alpha_1 + \alpha_2 + 2 \cos(\pi/r) \alpha_3\}$. Moreover, we claim that $\mathbf{G}_W$ has 4 edges aside from the ones in $\mathbf{G}_{W, \leq 2}$:
    \begin{itemize}
        \item a solid edge labeled $1$ from $2 \cos(\pi/q) \alpha_1 + \alpha_2 + 2 \cos(\pi/r) \alpha_3$ to $\alpha_{2, 3}(r-2)$ and a dotted edge labeled $1$ in the other direction;
        \item a solid edge labeled $3$ from $2 \cos(\pi/q) \alpha_1 + \alpha_2 + 2 \cos(\pi/r) \alpha_3$ to $\alpha_{1, 2}(1)$ and a dotted edge labeled $3$ in the other direction.
    \end{itemize}

    See \cref{fig:small-root-billiards-graph-245} for an example. To verify our claim, we use \cref{prop:small-roots-description}. We have \[2 \cos(\pi/q) \alpha_1 + \alpha_2 + 2 \cos(\pi/r) \alpha_3 = s_3 \alpha_{1, 2}(1)\] with $B(\alpha_{1, 2}(1), \alpha_3) = -\cos(\pi/q) > -1$, so $2 \cos(\pi/q) \alpha_1 + \alpha_2 + 2 \cos(\pi/r) \alpha_3$ is a small root. To find the edges of $\mathbf{G}_W$ incident to this vertex, we apply each of the simple reflections $s_1$, $s_2$, and $s_3$ to $2 \cos(\pi/q) \alpha_1 + \alpha_2 + 2 \cos(\pi/r) \alpha_3$. We omit the details; it is possible to determine whether each edge is solid or dotted using \cref{lem:edge-dotted,lem:edge-solid}.

    We must show that these are the only small roots. That is, we must show that $\Sigma = \Sigma'$, where
    \[\Sigma' = \Sigma_{\leq 2} \cup \{2 \cos(\pi/q) \alpha_1 + \alpha_2 + 2 \cos(\pi/r) \alpha_3\}.\]
    It suffices to prove that for all $i \in I$ and all $\gamma \in \Sigma' \setminus \{\alpha_i\}$, we have either $s_i \gamma \in \Sigma'$ or $B(\gamma, \alpha_i) \leq -1$. This reduces to proving three inequalities:
    \begin{align*}
        B(2 \cos(\pi/q) \alpha_1 + \alpha_2 + 2 \cos(\pi/r) \alpha_3, \alpha_2) &\leq -1; \\
        B(\alpha_{1, 2}(k), \alpha_3) &\leq -1 \qquad \text{for $2 \leq k \leq q - 2$}; \\
        B(\alpha_{2, 3}(k), \alpha_1) &\leq -1 \qquad \text{for $1 \leq k \leq r - 3$.}
    \end{align*}
    Each of these inequalities is easy to verify from the definition of the bilinear form $B$.
    
    We will now show that $c = s_{i_3}s_{i_2}s_{i_1}$ is futuristic, where $i_1 = 1$, $i_2 = 2$, and $i_3 = 3$. If $q$ is odd or $r$ is odd, then (along the lines of the argument from Case~4) no closed walk in ${\bf G}_W$ can satisfy \cref{def:closed-billiards-walk}\eqref{item:closed-billiards-walk-solid-edge-use}. If $q$ and $r$ are both even, then every closed walk in ${\bf G}_W$ that satisfies \cref{def:closed-billiards-walk}\eqref{item:closed-billiards-walk-solid-edge-use} uses the sequence of vertices 
    \begin{align*}
    \alpha_2, \alpha_{1, 2}(1), \alpha_{1, 2}(q - 3), \alpha_{1, 2}(3)&, \ldots, \alpha_{1, 2}(3), \alpha_{1, 2}(q - 3), \alpha_{1, 2}(1), \\
    &2 \cos(\pi/q) \alpha_1 + \alpha_2 + 2 \cos(\pi/r) \alpha_3, \\
    &\alpha_{2, 3}(r-2), \alpha_{2, 3}(2), \alpha_{2, 3}(r-4), \ldots, \alpha_{2, 3}(r-4), \alpha_{2, 3}(2), \alpha_{2, 3}(r-2),
    \end{align*}
    cyclically shifted and repeated $K$ times for some $K \geq 1$. Such a walk cannot satisfy \cref{def:closed-billiards-walk}\eqref{item:closed-billiards-walk-identity-product}. Indeed, using Matsumoto's theorem (\cref{thm:matsumoto}), one can check that the word $((12)^{q/2 - 1} 31 (23)^{r/2 - 1})^K$ is reduced, so \[((s_1 s_2)^{q/2 - 1} s_3 s_1 (s_2 s_3)^{r/2 - 1})^K \neq \mathbbm{1}.\] 
    By \cref{lem:closed-billiards-walk-futuristic}, $s_3s_2s_1$ is futuristic. By \cref{cor:tree}, $W$ is futuristic.
    
    \medskip
    
    \noindent {\bf Case 6.} ($p = 2$ and $q < r = \infty$)
    
    We claim that $\Sigma = \Sigma_{\leq 2}$. For this, it suffices to prove that for each $i \in I$ and all $\gamma \in \Sigma_{\leq 2} \setminus \{\alpha_i\}$, we have either $s_i \gamma \in \Sigma_{\leq 2}$ or $B(\gamma, \alpha_i) \leq -1$. This reduces to proving that $B(\alpha_3, \alpha_2) \leq -1$ and that $B(\alpha_{1, 2}(k), \alpha_3) \leq -1$ for $1 \leq k \leq q-1$; both inequalities are are easy to verify from the definition of the bilinear form $B$. It is straightforward to check that no closed walk in ${\bf G}_W$ can satisfy \cref{def:closed-billiards-walk}\eqref{item:closed-billiards-walk-solid-edge-use} (for any ordering $i_1, i_2, i_3$ of the indices $1, 2, 3$), so $W$ is futuristic by \cref{lem:closed-billiards-walk-futuristic}.
    \medskip

    \noindent {\bf Case 7.} ($p = 2$ and $q=r = \infty$)
    
    In this case, $W$ is futuristic because it is right-angled (\cref{thm:right-angled}).

    \medskip
    
    \noindent {\bf Case 8.} ($p \geq 3$) 
    
    In this case, $W$ is futuristic because its Coxeter graph is complete (\Cref{thm:complete}).
\end{proof}

\section{Ancient Coxeter Groups}\label{sec:ancient}

We now shift gears and construct examples of ancient Coxeter groups. 

Let us outline the strategy that led us to the arguments presented below. The Coxeter graph of each Coxeter group $W$ considered in this section is a tree, so in order to prove that $W$ is ancient, it suffices (by \cref{cor:tree}) to prove that a single fixed Coxeter element of $W$ is not futuristic. To do so, we first construct a candidate periodic sequence $u_0, u_1, u_2, \ldots$. Next, we attempt to find a nonempty convex subset $\LL \subseteq W$, not containing $u_0$, such that $u_0, u_1, u_2, \ldots$ is a periodic billiards trajectory with respect to $\LL$. The machinery of \cref{sec:small} tells us that we should take $u_0, u_1, u_2, \ldots$ to be a lift of a billiards-plausible closed walk in the corresponding small-root billiards graph.

It will be helpful to keep in mind an alternative definition of noninvertible Bender--Knuth toggles that uses reflections rather than roots. Let $\LL\subseteq W$ be a nonempty convex set, and let $i\in I$ and $u\in W$. Then $\tau_i(u)=u$ if either every element of $\LL\cup\{u\}$ has $u^{-1}s_iu$ as a right inversion or no element of $\LL\cup\{u\}$ has $u^{-1}s_iu$ as a right inversion. Otherwise, $\tau_i(u)=s_iu$. For each reflection $t$, the hyperplane $\HH_{\beta_t}$ is a one-way mirror if and only if either every element of $\LL$ has $t$ as a right inversion or no element of $\LL$ has $t$ as a right inversion.

If $J\subseteq I$ is such that the standard parabolic subgroup $W_J$ is finite, then we write $\wo(J)$ for the long element of $W_J$.

\begin{proposition}\label{prop:tilde_Bn_ancient}
Fix $n\geq 3$, and let $b,b'\geq 3$ be integers. The Coxeter group $W$ with Coxeter graph
\[\begin{array}{l}
\includegraphics[height=1.56cm]{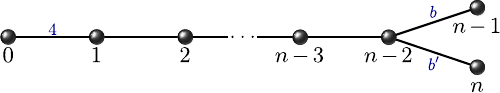}
\end{array}\]
is ancient. 
\end{proposition}

\begin{proof}
The Coxeter graph of $W$ is a tree, so by \cref{cor:tree}, it suffices to prove that the Coxeter element $c=s_0s_1\cdots s_{n-1}s_n$ is not futuristic. 

Let $J=\{1,\ldots,n-2\}$. We can view the standard parabolic subgroup $W_J$ of $W$ as the symmetric group $\mathfrak S_{n-1}$ (where $s_i$ is the transposition $(i\,\,i+1)$). For $0\leq j\leq n-3$, let \[z_j=(s_{j}s_{j+1})(s_{j-1}s_j)\cdots (s_2s_3)(s_1s_2) \in W_J.\] Note that $z_0=\id$. Using commutation moves, we can also write 
\begin{equation}\label{eq:z_j_alternative}
z_{j}=(s_j\cdots s_2s_1)(s_{j+1}\cdots s_3s_2).
\end{equation}
The one-line notation of the permutation $z_{n-3}s_1$ is $(n-1,n-2,1,2,3,\ldots,n-3)$. Consequently, the right inversions of $z_{n-3}s_1$ are $s_1$ and the transpositions of the form $(i\,\,j)$ with $1\leq i\leq 2<j\leq n-1$. 

Consider the reflection $t^*=s_0s_1s_0$. Let \[\mathscr{K}=\{s_{1},\,z_{n-3}s_1,\,\wo(\{n-2,n-1\})z_{n-3},\,\wo(\{n-2,n\})z_{n-3}\},\] and write $\mathscr Kt^*=\{vt^*:v\in\mathscr K\}$. Let us fix our convex set $\LL$ to be the convex hull of $\mathscr K\cup\mathscr Kt^*$. 
If we write $\LL t^*=\{vt^*:v\in\LL\}$, then $\LL=\LL t^*$. Note that if $t \neq s_1$ is a right inversion of some $w\in\mathscr K\cup\mathscr Kt^*$, then $\HH_{\beta_t}$ is a window because $w\in H_{\beta_t}^-$ and $s_{1}\in H_{\beta_t}^+$. For each $v \in \mathscr K$, one can check that $\ell(vt^*)=\ell(v)+\ell(t^*)$; this implies that 
\begin{equation}\label{eq:Invt*}
\Inv(vt^*)=\Inv(t^*)\cup\{t^*tt^*:t\in\Inv(v)\}=\{s_0,t^*,s_1s_0s_1\}\cup\{t^*tt^*:t\in\Inv(v)\}.
\end{equation} 
Recall the notation $[\xi\mid\xi']_d$ from \eqref{eq:mid_notation}. It will be helpful to keep in mind that we have the reduced expressions
\begin{align}
\wo(\{{n-2},{n-1}\})z_{n-3}&=[s_{n-1}\mid s_{n-2}]_{b-2}(s_{n-3}s_{n-4}\cdots s_1)(s_{n-1}s_{n-2}\cdots s_1) \label{eq:woz1}\\ &=[s_{n-2}\mid s_{n-1}]_{b-2}(s_{n-2}s_{n-3}\cdots s_1)(s_{n-1}s_{n-2}\cdots s_2) \label{eq:woz1'}
\end{align}
and 
\begin{align}
\wo(\{{n-2},{n}\})z_{n-3}&=[s_{n}\mid s_{n-2}]_{b'-2}(s_{n-3}s_{n-4}\cdots s_1)(s_{n}s_{n-2}\cdots s_1) \label{eq:woz2} \\ &=[s_{n-2}\mid s_{n}]_{b'-2}(s_{n-2}s_{n-3}\cdots s_1)(s_{n}s_{n-2}\cdots s_2) \label{eq:woz2'}
\end{align}
(with parentheses added for clarity). This can be seen by applying braid moves. For example, if $n=6$ and $b=5$, then \eqref{eq:woz1} follows from the computation \begin{align*}
\wo(\{{n-2},{n-1}\})z_{n-3}&=s_4s_5s_4s_5{\color{darkred}s_4s_3s_4}s_2s_3s_1s_2 \\ 
&=s_4s_5s_4s_5s_3s_4{\color{darkred}s_3s_2s_3}s_1s_2 \\ 
&=s_4s_5s_4s_5s_3s_4s_2s_3{\color{darkred}s_2s_1s_2} \\ 
&=s_4s_5s_4s_5s_3s_4s_2{\color{darkred}s_3s_1}s_2s_1 \\ 
&=s_4s_5s_4s_5s_3{\color{darkred}s_4s_2s_1}s_3s_2s_1 \\
&=s_4s_5s_4{\color{darkred}s_5s_3s_2s_1}s_4s_3s_2s_1 \\
&=s_4s_5s_4s_3s_2s_1s_5s_4s_3s_2s_1,
\end{align*}
while \eqref{eq:woz1'} follows from the computation 
\begin{align*}
\wo(\{{n-2},{n-1}\})z_{n-3}&=s_5s_4s_5s_4s_5s_3s_4s_2{\color{darkred}s_3s_1}s_2 \\
&=s_5s_4s_5s_4s_5s_3{\color{darkred}s_4s_2s_1}s_3s_2 \\
&=s_5s_4s_5s_4{\color{darkred}s_5s_3s_2s_1}s_4s_3s_2 \\
&=s_5s_4s_5s_4s_3s_2s_1s_5s_4s_3s_2.
\end{align*}

The reflection $s_1$ is a right inversion of both $s_1$ and $z_{n-3}s_1$. It follows from \eqref{eq:woz1} and \eqref{eq:woz2} that $s_1$ is also a right inversion of $\wo(\{{n-2},{n-1}\})z_{n-3}$ and of $\wo(\{{n-2},{n}\})z_{n-3}$. For each $v\in\mathscr K$, we have just shown that $s_1\in\Inv(v)$, so it follows from \eqref{eq:Invt*} that ${s_1=t^*s_1t^*\in\Inv(vt^*)}$. We deduce that $\mathscr K\cup\mathscr Kt^*\subseteq H_{\alpha_1}^-$ and hence $\LL\subseteq H_{\alpha_1}^-$. Note that $\id$ belongs to a proper stratum since $\id\in H_{\alpha_1}^+$. We will show that $\Pro_c^{2n-2}(\id)=\id$, which will prove that $c$ is not futuristic. Because $\LL=\LL t^*$, it follows from the definition of the noninvertible Bender--Knuth toggles (\cref{def:toggles}) that $\tau_i(u t^*) = \tau_i(u) t^*$ for all $u\in W$ and $i\in I$. This implies that $\Pro_c^k(u)u^{-1}=\Pro_c^k(ut^*)(ut^*)^{-1}$ for all $u\in W$ and $k\geq 0$. Setting $u=\id$ and $k=n-1$, we find that $\Pro_c^{n-1}(\id)t^*=\Pro_c^{n-1}(t^*)$. We will actually prove that 
\begin{align}
&\Pro_c(z_j)=z_{j+1} \quad\text{for all }0\leq j\leq n-4, \label{eq:BnStep1} \\ 
&\Pro_c(z_{n-3})=s_ns_{n-1}\cdots s_3s_2s_0, \label{eq:BnStep2} \\ 
&\Pro_c(s_ns_{n-1}\cdots s_3s_2s_0)=t^*; \label{eq:BnStep3}
\end{align}
since $\id=z_0$, this will immediately imply that $\Pro_c^{n-1}(\id)=t^*$, from which it will follow that \[\Pro_c^{2n-2}(\id)=\Pro_c^{n-1}(t^*)=\Pro_c^{n-1}(\id)t^*=(t^*)^2=\id,\] as desired. (The reader may find it helpful to consult \cref{exam:tilde_Bn_ancient} while reading the remainder of the proof.)

We begin by proving \eqref{eq:BnStep2} and \eqref{eq:BnStep3}.
Let \[\gamma_1=z_{n-3}^{-1}s_{n}s_{n-1}\alpha_{n-2}.\] Because $s_ns_{n-1}\alpha_{n-2}\in\Phi^+\setminus\Phi_J$ and $z_{n-3}^{-1}\in W_J$, we know that $\gamma_1\in\Phi^+$. Note that each element of $\mathscr K\cup\mathscr Kt^*$ belongs to either $W_{I\setminus\{n-1\}}$ or $W_{I\setminus\{n\}}$; since the reflection $r_{\gamma_1}=z_{n-3}^{-1}s_ns_{n-1}s_{n-2}s_{n-1}s_nz_{n-3}$ does not belong to either of these parabolic subgroups, it is not a right inversion of any element of $\mathscr K\cup\mathscr Kt^*$. This shows that $\mathscr K\cup\mathscr Kt^*\subseteq H_{\gamma_1}^+$, so 
\begin{equation}
\LL\subseteq H_{\gamma_1}^+.
\end{equation}

Using \eqref{eq:z_j_alternative}, we find that $z_{n-3}^{-1}\alpha_{n}=s_2s_3\cdots s_{n-2}\alpha_{n}$. This implies that when we apply $\tau_n$ to $z_{n-3}$, we cross through $\HH_{s_2s_3\cdots s_{n-2}\alpha_n}$, which is a window because it corresponds to a right inversion of the element $\wo(\{{n-2},n\})z_{n-3}$ (by \eqref{eq:woz2'}). Likewise, when we apply $\tau_{n-1}$ to $s_nz_{n-3}$, we cross through $\HH_{s_2s_3\cdots s_{n-2}\alpha_{n-1}}$, which is a window because it corresponds to a right inversion of the element $\wo(\{{n-2},{n-1}\})z_{n-3}$ (by \eqref{eq:woz1'}). When we apply $\tau_{n-2}$ to $s_{n-1}s_nz_{n-3}$, we hit the mirror $\HH_{\gamma_1}$. When we next apply $\tau_{n-3},\tau_{n-4},\ldots,\tau_1$ (this sequence is empty if $n=3$), we cross through $\HH_{s_1s_2\cdots s_{n-3}\alpha_{n-2}},\HH_{s_1s_2\cdots s_{n-4}\alpha_{n-3}},\ldots,\HH_{s_1\alpha_2}$, which are windows because they correspond to the right inversions $(1\,\,n-1),(1\,\,n-2),\ldots,(1\,\,3)$ of $z_{n-3}s_1$. Thus, \[\tau_1\tau_2\cdots\tau_n(z_{n-3})=s_1s_2\cdots s_{n-3}s_{n-1}s_{n}z_{n-3}=s_ns_{n-1}\cdots s_3s_2.\]  When we apply $\tau_0$ to $s_ns_{n-1}\cdots s_3s_2$, we cross through $\HH_{\alpha_0}$, which is a window because it corresponds to the right inversion $s_0$ of $s_1t^*$. This proves \eqref{eq:BnStep2}.   

When we apply $\tau_n$ to $s_ns_{n-1}\cdots s_3s_2s_0$, we cross through $\HH_{s_2s_3\cdots s_{n-2}\alpha_n}$, which is a window because it corresponds to a right inversion of $\wo(\{{n-2},n\})z_{n-3}$ (by \eqref{eq:woz2'}). When we then apply $\tau_{n-1}$ to $s_{n-1}\cdots s_3s_2s_0$, we cross through $\HH_{s_2s_3\cdots s_{n-2}\alpha_{n-1}}$, which is a window because it corresponds to a right inversion of $\wo(\{{n-2},{n-1}\})z_{n-3}$ (by \eqref{eq:woz1'}). When we next apply $\tau_{n-2},\tau_{n-3},\ldots,\tau_2$ (this sequence is empty if $n=3$), we cross through $\HH_{s_2s_3\cdots s_{n-3}\alpha_{n-2}},\HH_{s_2s_3\cdots s_{n-4}\alpha_{n-3}},\ldots,\HH_{s_2\alpha_3},\HH_{\alpha_2}$, which are windows because they correspond to the right inversions $(2\,\,n-1),(2\,\,n-2),\ldots,(2\,\,3)$ of $z_{n-3}s_1$. Thus, \[\tau_2\tau_3\cdots\tau_n(s_ns_{n-1}\cdots s_3s_2s_0)=s_0.\] When we apply $\tau_1$ to $s_0$, we cross through $\HH_{s_0\alpha_1}$, which is a window because it corresponds to the right inversion $s_0s_1s_0$ of $s_1t^*$. When we apply $\tau_0$ to $s_1s_0$, we cross through $\HH_{s_0s_1\alpha_0}$, which is a window because it corresponds to the right inversion $s_1s_0s_1$ of $s_1t^*$. This proves \eqref{eq:BnStep3}. 

It remains to prove \eqref{eq:BnStep1}. Because \eqref{eq:BnStep1} is vacuously true if $n=3$, we may assume in what follows that $n\geq 4$. Let $\gamma_{2}=s_{2}s_{1}\alpha_0\in\Phi^+$, and for $3\leq i\leq n$, let $\gamma_i=\alpha_i\in\Phi^+$. We will prove that 
\begin{equation}\label{eq:Kgamma}
\mathscr K\cup\mathscr Kt^*\subseteq H_{\gamma_2}^+\cap H_{\gamma_3}^+\cap\cdots\cap H_{\gamma_n}^+.
\end{equation}
Suppose $2\leq i\leq n$ and $v\in\mathscr K$; we wish to show that $r_{\gamma_i}\not\in\Inv(v)\cup\Inv(vt^*)$. One may readily check that $t^*r_{\gamma_i}t^*=r_{\gamma_i}$ and that $r_{\gamma_i}\not\in\Inv(t^*)$. Appealing to \eqref{eq:Invt*}, we find that $r_{\gamma_i}$ is in $\Inv(v)$ if and only if it is in $\Inv(vt^*)$. Hence, it suffices to show that $r_{\gamma_i}\not\in\Inv(v)$. We may verify this directly if $3\leq i\leq n$; if $i=2$, then we simply need to observe that $v\in W_{I\setminus\{0\}}$ and $r_{\gamma_2}\not\in W_{I\setminus\{0\}}$. This proves \eqref{eq:Kgamma}. Because $\LL$ is the convex hull of $\mathscr K\cup\mathscr Kt^*$, it follows that 
\[\LL\subseteq H_{\gamma_2}^+\cap H_{\gamma_3}^+\cap\cdots\cap H_{\gamma_n}^+.\]

Suppose $0\leq j\leq n-4$. Let us compute $\Pro_c(z_j)$ by starting at $z_j$ and applying $\tau_n,\tau_{n-1},\ldots,\tau_0$ in that order. When we apply $\tau_{n},\tau_{n-1},\ldots,\tau_{j+3}$, we hit the mirrors $\HH_{\gamma_n},\HH_{\gamma_{n-1}},\ldots,\HH_{\gamma_{j+3}}$, respectively. When we apply $\tau_{j+2}$ to $z_j$, we cross through the hyperplane $\HH_{z_j^{-1}\alpha_{j+2}}=\HH_{s_2s_3\cdots s_{j+1}\alpha_{j+2}}$, which is a window because it corresponds to the right inversion $(2\,\,\, j+3)$ of $z_{n-3}s_1$. When we next apply $\tau_{j+1}$ to $s_{j+2}z_j$, we cross through the hyperplane $\HH_{z_j^{-1}s_{j+2}\alpha_{j+1}}=\HH_{s_1s_2\cdots s_{j+1}\alpha_{j+2}}$, which is a window because it corresponds to the right inversion $(1\,\,\,j+3)$ of $z_{n-3}s_1$. This shows that \[\tau_{j+1}\tau_{j+2}\cdots\tau_n(z_j)=s_{j+1}s_{j+2}z_j=z_{j+1}.\] If we now apply $\tau_{j},\tau_{j-1},\ldots,\tau_0$ to $z_{j+1}$, we hit the mirrors $\HH_{\gamma_{j+2}},\HH_{\gamma_{j+1}},\ldots,\HH_{\gamma_2}$, respectively. This establishes \eqref{eq:BnStep1} and completes the proof that $W$ is ancient. 
\end{proof}

\begin{example}\label{exam:tilde_Bn_ancient}
    Let us illustrate the proof of \cref{prop:tilde_Bn_ancient} when $n=5$ (for arbitrary integers $b,b'\geq 3$). In this case, 
we have \[\mathscr{K}=\{s_1,\,{\color{Traj4}z_{2}s_1},\,{\color{Traj1}\wo(\{3,4\})z_2},\,{\color{Traj5}\wo(\{{3},{5}\})z_2}\}\] and \[\mathscr{K}t^*=\{{\color{red}s_1t^*},\,z_{2}s_1t^*,\,\wo(\{3,4\})z_2t^*,\,\wo(\{3,5\})z_2t^*\},\] where $z_2=s_3s_4s_2s_3$ and $t^*=s_0s_1s_0$. Then $\LL$ is the convex hull of $\mathscr K\cup\mathscr Kt^*$. Fix the ordering $5,4,3,2,1,0$ of $I$. The beginning of the billiards trajectory starting at $\id$ is  
\[\begin{array}{l}\includegraphics[height=2.443cm]{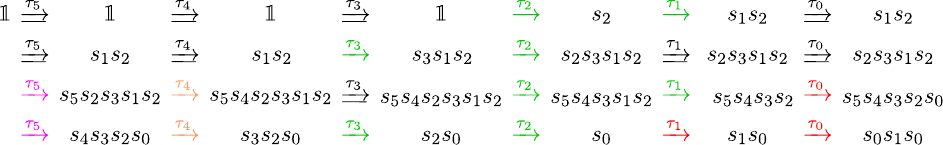}.
\end{array}\]
We have underlined each arrow associated to a toggle that hits a mirror (this is meant to resemble an equals sign). We have colored each arrow that passes through a window $\HH_{\beta_t}$ with the same color as one of the elements $w\in\mathscr K\cup\mathscr Kt^*$ such that $t\in\Inv(w)$. For example, the arrow corresponding to $\tau_3$ in the second row is colored {\color{Traj4}green} because it passes through the window $\HH_{\beta_{t}}$, where ${t=s_2s_3s_2\in\Inv({\color{Traj4}z_2s_1})}$. 
\end{example}  

\begin{proposition}\label{prop:tilde_Dn_ancient}
Fix $n\geq 3$, and let $a,a',b,b'\geq 3$ be integers. The Coxeter group $W$ with Coxeter graph
\[\begin{array}{l}
\includegraphics[height=1.564cm]{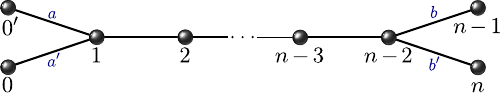}
\end{array}\]
is ancient. 
\end{proposition}

\begin{proof}
The proof is similar to that of \cref{prop:tilde_Bn_ancient}, so we will omit some details. The Coxeter graph of $W$ is a tree, so it suffices (by \cref{cor:tree}) to prove that the Coxeter element ${c=s_ns_{n-1}\cdots s_1s_0s_{0'}}$ is not futuristic. 

For $0\leq j\leq n-3$, let \[z_j=(s_js_{j+1})\cdots (s_2s_3)(s_1s_2).\] Let us fix our convex set $\LL$ to be the convex hull of the set \[\mathscr{K}=\{s_{1},\,z_{n-3}s_1,\,\wo(\{{n-2},{n-1}\})z_{n-3},\,\wo(\{{n-2},{n}\})z_{n-3},\,\wo(\{{0},1\}),\,\wo(\{{0'},1\})\}.\] 
If $t \neq s_1$ is a right inversion of some $w\in\mathscr K$, then $\HH_{\beta_t}$ is a window because $w\in H_{\beta_t}^-$ and $s_1\in H_{\beta_t}^+$. For example, $\HH_{\alpha_0}$ and $\HH_{\alpha_{0'}}$ are windows because $s_0$ and $s_{0'}$ are right inversions of $\wo(\{{0},1\})$ and $\wo(\{{0'},1\})$, respectively. 

We can readily check that $\LL\subseteq H_{\alpha_1}^-$. Because $\id\in H_{\alpha_1}^+$, we know that $\id$ belongs to a proper stratum. We will show that $\Pro_c^{n-1}(\id)=\id$, which will prove that $c$ is not futuristic. To do so, we will actually prove that
\begin{align}
&\Pro_c(z_j)=z_{j+1} \quad\text{for all }0\leq j\leq n-4, \label{eq:DnStep1} \\ 
&\Pro_c(z_{n-3})=s_ns_{n-1}\cdots s_3s_2s_0s_{0'}, \label{eq:DnStep2} \\ 
&\Pro_c(s_ns_{n-1}\cdots s_3s_2s_0s_{0'})=\id. \label{eq:DnStep3}
\end{align}

An argument very similar to the one used to prove \eqref{eq:BnStep2} in the proof of \cref{prop:tilde_Bn_ancient} yields the identity $\tau_0\tau_1\cdots\tau_{n-1}\tau_n(z_{n-3})=s_ns_{n-1}\cdots s_3s_2s_0$. Moreover, if we apply $\tau_{0'}$ to $s_ns_{n-1}\cdots s_3s_2s_0$, then we cross through the window $\HH_{\alpha_{0'}}$ to reach $s_ns_{n-1}\cdots s_3s_2s_0s_{0'}$. This proves \eqref{eq:DnStep2}. 

Let us now prove \eqref{eq:DnStep3}. Note that we still have the reduced expressions \eqref{eq:woz1'} and \eqref{eq:woz2'} for $\wo(\{{n-2},{n-1}\})z_{n-3}$ and $\wo(\{{n-2},{n}\})z_{n-3}$, respectively. When we apply $\tau_n$ to the element $s_ns_{n-1}\cdots s_3s_2s_0s_{0'}$, we cross through $\HH_{s_2s_3\cdots s_{n-2}\alpha_n}$, which is a window because it corresponds to a right inversion of $\wo(\{{n-2},n\})z_{n-3}$ (by \eqref{eq:woz2'}). When we then apply $\tau_{n-1}$ to $s_{n-1}\cdots s_3s_2s_0s_{0'}$, we cross through $\HH_{s_2s_3\cdots s_{n-2}\alpha_{n-1}}$, which is a window because it corresponds to a right inversion of $\wo(\{{n-2},{n-1}\})z_{n-3}$ (by \eqref{eq:woz1'}). When we next apply $\tau_{n-2},\tau_{n-3},\ldots,\tau_2$ (this sequence is empty if $n=3$), we cross through $\HH_{s_2s_3\cdots s_{n-3}\alpha_{n-2}},\HH_{s_2s_3\cdots s_{n-4}\alpha_{n-3}},\ldots,\HH_{s_2\alpha_3},\HH_{\alpha_2}$, which are windows because they correspond to the right inversions $(2\,\,n-1),(2\,\,n-2),\ldots,(2\,\,3)$ of $z_{n-3}s_1$ (we are again viewing $W_{\{1,\ldots,n-2\}}$ as $\mathfrak S_{n-1}$). Thus, \[\tau_2\tau_3\cdots\tau_n(s_ns_{n-1}\cdots s_3s_2s_0s_{0'})=s_0s_{0'}.\] 
Notice that $\LL \subseteq H_{s_{0'} s_0 \alpha_1}^+$ since each element of $\mathscr{K}$ is in either $W_{I\setminus\{0\}}$ or $W_{I\setminus\{0'\}}$. Moreover, $s_0s_{0'} \in H_{s_{0'} s_0 \alpha_1}^+$ since $s_{0'} s_0 s_1 s_0 s_{0'}$ is not a right inversion of $s_0 s_{0'}$.  Hence, when we apply $\tau_1$ to $s_0s_{0'}$, we hit $\HH_{s_{0'}s_0\alpha_1}$, which is a mirror. When we apply $\tau_0$ to $s_0s_{0'}$, we cross through the window $\HH_{\alpha_0}$ to reach $s_{0'}$. Finally, when we apply $\tau_{0'}$ to $s_{0'}$, we cross through the window $\HH_{\alpha_{0'}}$ to reach $\id$. This proves \eqref{eq:DnStep3}.

It remains to prove \eqref{eq:DnStep1}. Because \eqref{eq:DnStep1} is vacuously true if $n=3$, we may assume in what follows that $n\geq 4$. 
Consider the positive roots \[\gamma_{2}=s_{2}s_{1}\alpha_0\quad\text{and}\quad\gamma_2'=s_2s_1\alpha_{0'},\] and for $3\leq i\leq n$, let $\gamma_i=\alpha_i$. An argument very similar to the one used in the proof of \cref{prop:tilde_Bn_ancient} shows that \[\LL\subseteq H_{\gamma_2}^+\cap H_{\gamma_{2}'}^+\cap H_{\gamma_3}^+\cap\cdots\cap H_{\gamma_n}^+.\] 
By employing an argument very similar to the one used to establish \eqref{eq:BnStep1} in the proof of \cref{prop:tilde_Bn_ancient}, we find that $\tau_0\tau_1\cdots\tau_{n-1}\tau_n(z_j)=z_{j+1}$ for all $0\leq j\leq n-4$. Moreover, if we apply $\tau_{0'}$ to $z_{j+1}$, we hit the mirror $\HH_{\gamma_2'}$. This establishes \eqref{eq:DnStep1} and completes the proof that $W$ is ancient. 
\end{proof}

\begin{example}\label{exam:tilde_Dn_ancient}
Let us illustrate the proof of \cref{prop:tilde_Dn_ancient} when $n=5$ (for arbitrary integers $a,a',b,b'\geq 3$). In this case, $\LL$ is the convex hull of the set \[\mathscr{K}=\{s_1,\,{\color{Traj4}z_{2}s_1},\,{\color{Traj1}\wo(\{3,4\})z_2},\,{\color{Traj5}\wo(\{{3},{5}\})z_2},\,{\color{Traj2}\wo(\{0,1\})},\,{\color{Traj3}\wo(\{{0'},1\})}\},\] where $z_2=s_3s_4s_2s_3$. Fix the ordering $5,4,3,2,1,0,0'$ of $I$. The billiards trajectory starting at $\id$ begins with 
\[\begin{array}{l}\includegraphics[width=.98\linewidth]{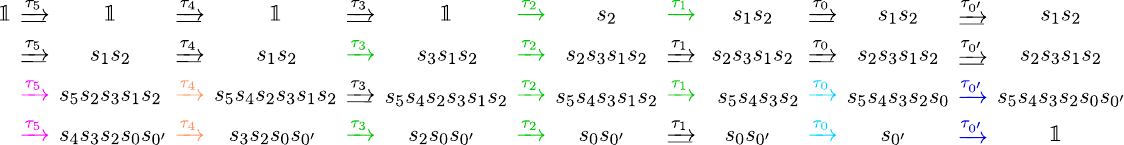}
\end{array}\]
and then continues to repeat periodically. We have underlined each arrow associated to a toggle that hits a mirror. We have colored each arrow that passes through a window $\HH_{\beta_t}$ with the same color as one of the elements $w\in\mathscr K$ such that $t\in\Inv(w)$. 
\end{example}  

\begin{proposition}
The Coxeter group with Coxeter graph 
\[\begin{array}{l}\includegraphics[height=1.56cm]{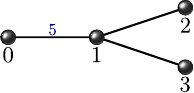}\end{array}\]
is ancient. 
\end{proposition}

\begin{proof}
By \cref{cor:tree}, it suffices to show that the Coxeter element $c=s_0s_1s_2s_3$ is not futuristic. Consider the reflection $t^*=s_0s_1s_0$. Let \[\mathscr K=\{s_1,\,\wo(\{0,1,2\}),\,\wo(\{0,1,3\})\},\] and write $\mathscr Kt^*=\{vt^*:v\in\mathscr K\}$. Let us fix our convex set $\LL$ to be the convex hull of $\mathscr K\cup\mathscr Kt^*$. One can check directly (by hand or by computer) that $\id\not\in\LL$ and that $\Pro_c^4(\id)=\id$. Hence, $c$ is not futuristic. 
\end{proof}

In the following proposition, we consider the affine Coxeter groups $\widetilde E_6$, $\widetilde E_7$, $\widetilde E_8$, $\widetilde F_4$, whose Coxeter graphs are 
\[\includegraphics[height=1.273cm]{BKTogglesPIC24}\,\, ,\]

\[\includegraphics[height=1.273cm]{BKTogglesPIC23}\,\, ,\]

\[\includegraphics[height=1.273cm]{BKTogglesPIC45}\,\,,\]

\[\includegraphics[height=0.373cm]{BKTogglesPIC22}\,\, ,\]
respectively. 

\begin{proposition} 
The Coxeter groups $\widetilde E_6$, $\widetilde E_7$, $\widetilde E_8$, $\widetilde F_4$ are ancient. 
\end{proposition}

\begin{proof}
According to \cref{cor:tree}, we need only show that these Coxeter groups are not futuristic. We saw in \cref{exam:folding} that $\widetilde F_4$ is a folding of $\widetilde E_6$ and is also a folding of $\widetilde E_7$. Therefore, by \cref{prop:folding}, it suffices to show that $\widetilde E_8$ and $\widetilde F_4$ are not futuristic. 

Let us start with $\widetilde F_4$. Let $0,1,2,3,4$ be the vertices of the Coxeter graph of $\widetilde F_4$, listed from left to right (in the figure drawn above), and let $c=s_0s_1s_2s_3s_4$.
Consider the positive roots
\[ \beta = s_2\alpha_3,\quad \gamma_1 = s_0s_1\alpha_2,\quad \gamma_2 = s_1s_3\alpha_2,\quad \gamma_3 = s_4s_3\alpha_2, \]
and let 
\[ \LL = H_{\beta}^- \cap H_{\gamma_1}^+ \cap H_{\gamma_2}^+ \cap H_{\gamma_3}^+. \]
Evidently, $\id$ is not in $\LL$ since $\id$ is not in $H_{\beta}^-$. One can check (by hand or by computer) that 
$\Pro_c^{3}(\id)=\id$.
Hence, $c$ is not futuristic.   

We now prove that $\widetilde E_8$ is not futuristic. Let us identify the vertices of the Coxeter graph with $0,1,\ldots,8$ as follows: \[\begin{array}{l}\includegraphics[height=1.595cm]{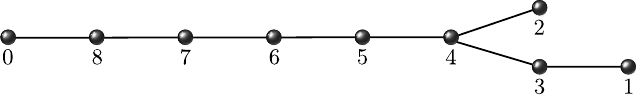}\end{array}.\] 
Let $c'=s_0s_1s_2s_3s_4s_5s_6s_7s_8$. Consider the positive roots
\[ \beta = s_7s_6s_4s_3s_2s_5\alpha_4,\quad \gamma_1=\alpha_0,\quad \gamma_2=\alpha_2,\quad \gamma_3=s_1s_3\alpha_4,\quad \gamma_4 = s_5s_6\alpha_7\]
\[  \gamma_5 = s_3s_4\alpha_5,\quad \gamma_6 = s_6s_7\alpha_8, \quad \gamma_7 = s_4s_5\alpha_6,\quad \gamma_8=s_7\alpha_8, \]
and let
\[ \LL = H_{\beta}^- \cap H_{\gamma_1}^+ \cap H_{\gamma_2}^+ \cap H_{\gamma_3}^+ \cap H_{\gamma_4}^+ \cap H_{\gamma_5}^+ \cap H_{\gamma_6}^+ \cap H_{\gamma_7}^+ \cap H_{\gamma_8}^+. \]
Evidently, $\id$ is not in $\LL$ since $\id$ is not in $H_{\beta}^-$. One can check (by hand or by computer) that 
$\Pro_{c'}^{3}(\id)=\id$.
Hence, $c'$ is not futuristic.  
\end{proof}

\section{Classical Billiards (Light Beams)}\label{sec:linear}
We now briefly discuss how our Bender--Knuth billiards systems relate to more classical billiards systems studied in dynamics (as in \cite{Hasselblatt, Kozlov, McMullen, Tabachnikov}).

As before, let $(W,S)$ be a Coxeter system whose simple reflections are indexed by a finite set $I$; we will assume that $|I|\geq 2$. Recall that $\BB W$ is the Tits cone of $W$. A \dfn{ray} is a continuous function $\rr \colon \mathbb{R}_{\geq 0} \to \BB W$ such that $\rr(t) \neq 0$ for all $t \in \mathbb{R}_{\geq 0}$, the positive projectivization $\overline{\rr} \colon \mathbb{R}_{\geq 0} \to \mathbb{P}(\BB W)$ is locally injective, and the image of $\rr$ is contained in a 2-dimensional subspace of $V^*$.

\begin{example}
    If $W$ is finite, then $\BB W = V^*$. In this case, for any two linearly independent vectors $\gamma, \gamma' \in V^*$, there is a ray $\rr \colon \mathbb{R}_{\geq 0} \to \BB W$ defined by \[\rr(t) = \cos(t) \gamma + \sin(t) \gamma'.\]
\end{example}
\begin{definition}
Let us say that a continuous, piecewise linear curve $\qq \colon \mathbb{R}_{\geq 0} \to \BB W$ is \dfn{cordial} if it satisfies the following conditions:
\begin{enumerate}[(i)]
\item There are infinitely many real numbers $t \in \mathbb{R}_{\geq 0}$ such that $\qq(t)\in\bigcup_{\HH\in\mathcal H_W}\HH$.
\item For any fixed real number $b$, there are only finitely many real numbers $t$ with $0 \leq t \leq b$ such that $\qq(t)\in\bigcup_{\HH\in\mathcal H_W}\HH$. 
\item \label{item:pass-through-only-orthogonal} For all $\beta, \beta' \in \Phi$ such that $B(\beta, \beta') \neq 0$, the image of $\qq$ does not intersect $\HH_\beta \cap \HH_{\beta'}$.
\end{enumerate}
\end{definition}

Let $\rr \colon \mathbb{R}_{\geq 0} \to \BB W$ be a cordial ray such that $\rr(0)\not\in\bigcup_{\HH\in\mathcal H_W}\HH$. Consider the open regions of the Coxeter arrangement through which the curve $\rr$ passes. We obtain a sequence $v_0, v_1, v_2, \ldots$ of elements of $W$ and a sequence $0 = t_0 < t_1 < t_2 < \cdots$ of real numbers with the property that for all nonnegative integers $k$ and all real numbers $t$ with $t_{k} < t < t_{k+1}$, we have $\qq(t) \in \BB^\circ v_k$, where $\BB^\circ$ denotes the interior of $\BB$. Although $\rr$ is a continuous curve, the regions $\BB v_{k-1}$ and $\BB v_{k}$ can be nonadjacent in the Coxeter arrangement since $\rr(t_{k})$ can lie on multiple hyperplanes of $\mathcal{H}_W$. Put differently, it is possible that $v_{k} v_{k-1}^{-1}$ is not a simple reflection. However, \eqref{item:pass-through-only-orthogonal} guarantees that for each integer $k \geq 1$, the hyperplanes in~$\mathcal{H}_W$ that contain $\rr(t_{k})$ are pairwise orthogonal. Since $\qq(t_{k})$ lies in both $\BB v_{k-1}$ and $\BB v_{k}$, we conclude that $v_{k} v_{k-1}^{-1}$ is a product of distinct simple reflections that commute pairwise. Let $J_k(\rr)$ denote the set of the indices of these simple reflections. Then $J_k(\rr)$ is an independent set (i.e., a set of pairwise nonadjacent vertices) of the Coxeter graph $\Gamma_W$.

\begin{definition}
    Let $i_1, \ldots, i_n$ be an ordering of the elements of $I$, and let $i_1, i_2, i_3, \ldots$ be the infinite sequence satisfying $i_{j+n} = i_j$ for every positive integer $j$. We say that the ordering $i_1, \ldots, i_n$ is \dfn{luminous} if there exists a cordial ray $\rr \colon \mathbb{R}_{\geq 0} \to \BB W$ such that the sequence $i_1, i_2, i_3, \ldots$ consists of the elements of $J_1(\rr)$ in some order, followed by the elements of $J_2(\rr)$ in some order, followed by the elements of $J_3(\rr)$ in some order, and so on. In this case, we say that $\rr$ \dfn{certifies} that $i_1, \ldots, i_n$ is luminous.
\end{definition}

Let $i_1, \ldots, i_n$ be a luminous ordering of the elements of $I$, and let $i_1, i_2, i_3, \ldots$ be the infinite sequence satisfying $i_{j+n} = i_j$ for every positive integer $j$. Let $\rr$ be a cordial ray that certifies that $i_1, \ldots, i_n$ is luminous, and let $J_k = J_k(\rr)$. That is, for every $k \geq 1$, we have \[J_k = \{i_{j+1}, \ldots, i_{j + |J_{k}|}\},\] where $j = |J_1| + \cdots + |J_{k-1}|$. Let $\LL \subseteq W$ be a convex set, and let $u_0, u_1, u_2, \ldots$ be a billiards trajectory defined with respect to the ordering $i_1, \ldots, i_n$. We will now describe how to ``fold'' the ray $\rr$ into a cordial curve $\pp \colon \mathbb{R}_{\geq 0} \to \BB W$, which we call a \dfn{light beam}. The light beam $\pp$ will reflect off of the one-way mirrors of the convex set $\LL$ as in classical billiards, and we can think of the Bender--Knuth billiards trajectory $u_0, u_1, u_2, \ldots$ as a discretization of $\pp$.

For each independent set $J = \{j_1, \ldots, j_d\} \subseteq I$ of the Coxeter graph $\Gamma_W$, define 
\[s_J = s_{j_1} \cdots s_{j_d}\qquad \text{and} \qquad \tau_J = \tau_{j_1} \cdots \tau_{j_d};\]
since $J$ is independent, the order of the multiplication used to define $s_J$ does not matter, nor does the order of the composition used to define $\tau_J$.

After multiplying $\rr$ on the right by a suitable element of $W$, we may assume that $\rr(0)$ lies in the region $\BB^\circ u_0$. Recall that there exist a sequence $v_0, v_1, v_2, \ldots$ of elements of $W$ and a sequence $0 = t_0 < t_1 < t_2 < \cdots$ of positive real numbers such that $\qq(t) \in \BB^\circ v_k$ for all nonnegative integers $k$ and all real numbers $t$ with $t_{k} < t < t_{k+1}$. We have \[v_k = s_{J_k} s_{J_{k-1}} \cdots s_{J_1} u_0.\] Therefore, the ray $\rr$ passes consecutively through the open regions \[\BB^\circ u_0,\,\BB^\circ s_{J_1} u_0,\,\BB^\circ s_{J_2} s_{J_1} u_0,\,\BB^\circ s_{J_3} s_{J_2} s_{J_1} u_0, \ldots\]
of the Coxeter arrangement. The light beam $\pp$ will be constructed so that it instead passes consecutively through the open regions \[\BB^\circ u_0,\,\BB^\circ \tau_{J_1}(u_0),\,\BB^\circ \tau_{J_2} \tau_{J_1}(u_0),\,\BB^\circ \tau_{J_3} \tau_{J_2} \tau_{J_1}(u_0), \ldots.\]
Note that the elements $u_0,  \tau_{J_1}(u_0), \tau_{J_2} \tau_{J_1}(u_0), \tau_{J_3} \tau_{J_2} \tau_{J_1}(u_0), \ldots$ form a subsequence of the billiards trajectory $u_0, u_1, u_2, \ldots$.

Define the function $\pp \colon \mathbb{R}_{\geq 0} \to \BB W$ via \[\pp(t) = \rr(t) (s_{J_k} \cdots s_{J_1} u_0)^{-1} (\tau_{J_k} \cdots \tau_{J_1}(u_0))\] for all integers $k \geq 0$ and all real numbers $t$ with $t_k \leq t < t_{k+1}$. It is not difficult to check that $\pp$ is a continuous, piecewise linear, and cordial curve. The positive projectivization $\overline{\pp} \colon \mathbb{R}_{\geq 0} \to \mathbb{P}(\BB W)$ is a curve that travels along a linear path unless it hits the reflective side of a one-way mirror of $\LL$, at which point it reflects off of that one-way mirror. It is possible to equip $\mathbb{P}(\BB W)$ with a metric \cite[Section~2]{McMullen} in such a way that when the light beam $\overline{\pp}$ reflects off of a one-way mirror, the angle of incidence equals the angle of reflection. With this metric, $\mathbb{P}(\BB W)$ is isometric to a sphere if $W$ is finite, to a Euclidean space if $W$ is affine, and to a hyperbolic space $\mathbb{H}^{n - 1}$ if $W$ is hyperbolic (i.e., if the bilinear form $B \colon V \times V \to \mathbb{R}$ has signature $(n-1, 1)$, where $n = |I|$ is the rank of $W$).  For examples of light beams reflecting in this manner, see the thin cyan curves in \cref{fig:affineS3,fig:hyperbolic}.\footnote{When the thin cyan curve in \cref{fig:spherical} reflects off the hyperplane $\HH_{e_1 - e_3}$, the angle of incidence does not appear to be equal to the angle of reflection. This is because \cref{fig:spherical} is drawn not-to-scale and  angles are distorted.}

Let us highlight two settings where it is known that an ordering is luminous. Both of these settings are discussed in \cite{McMullen}.

The first setting is that in which $\Gamma_W$ is bipartite. Let $I=Q\sqcup Q'$ be a bipartition of $\Gamma_W$; we may assume that $Q$ and $Q'$ are both nonempty. Let $i_1,\ldots,i_q$ be an ordering of $Q$, and let $i_{q+1},\ldots,i_n$ be an ordering of $Q'$. The base region $\BB$ is a cone whose facets correspond to the elements of $I$. For $i\in I$, let $F_i$ be the facet of $\BB$ that separates $\BB$ from $\BB s_i$. Let $\mathcal F_Q=\bigcap_{i\in Q}F_i$ and $\mathcal F_{Q'}=\bigcap_{i\in Q'}F_i$. We can choose $z\in\mathbb P(\mathcal F_Q)$ and $z'\in\mathbb P(\mathcal F_{Q'})$ so that the line $\mathfrak l$ passing through $z$ and $z'$ is orthogonal to both $\mathcal F_Q$ and $\mathcal F_{Q'}$ (here, orthogonality is defined with respect to the Hilbert metric on $\mathbb P(\BB W)$, as defined in \cite{McMullen}). Let $x_0$ be a point on the line $\mathfrak l$ that lies in the interior of $\mathbb P(\BB)$. Let $\rr\colon\mathbb R_{\geq 0}\to \BB W$ be a ray whose positive projectivization $\overline\rr$ satisfies $\overline\rr(0)=x_0$, $\overline\rr(1)=z$, and $\overline\rr(\epsilon)\in\mathbb P(\BB)$ for all $0\leq \epsilon\leq 1$. Then $\rr$ certifies that $i_1, \ldots, i_n$ is a luminous ordering of $I$.

\begin{figure}[ht]
 \begin{center}{\includegraphics[width=0.65\linewidth]{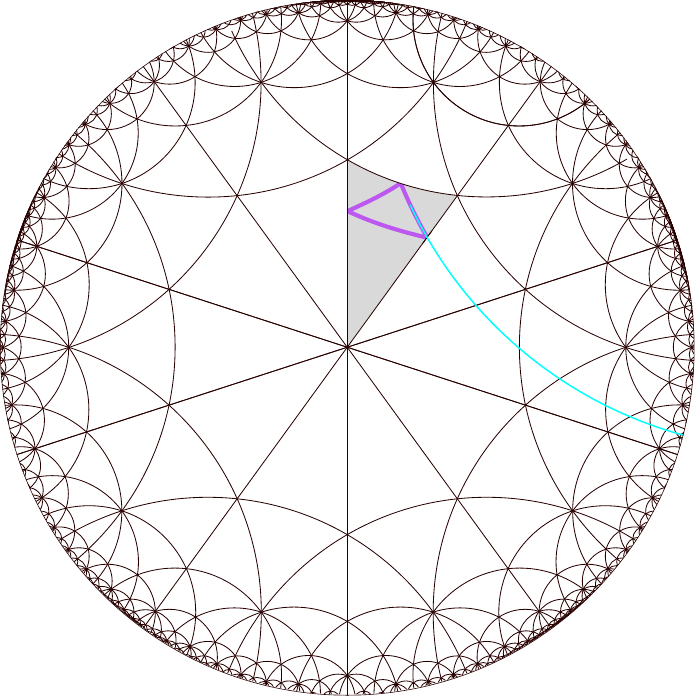}}
  \end{center}
\caption{
The Tits cone and Coxeter arrangement of the Coxeter group with Coxeter graph \!\!$\begin{array}{l} \includegraphics[height=0.659cm]{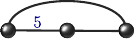}\end{array}$\!. We have passed to the positive projectivization $\mathbb{P}({\BB W})$, which is a hyperbolic plane, and then drawn the hyperbolic plane using the Poincar\'e disk model. We have shaded $\mathbb P(\BB)$ in light gray and drawn the pedal triangle of $\mathbb P(\BB)$ in purple. We have used the pedal triangle to construct a ray $\rr$; the image of $\overline \rr$ is drawn in cyan.  
}\label{fig:pedal}
\end{figure}

The second setting is that in which $W$ has rank $3$. In this case, we will show that each ordering of $I$ is a cyclic shift of a luminous ordering, which will justify why we were able to draw the thin cyan light beams in \cref{fig:affineS3,fig:spherical,fig:affineS3_infinite,fig:hyperbolic}. If $m(i,i')=2$ for some $i,i'\in I$, then every ordering of $I$ is a cyclic shift of an ordering coming from a bipartition of $\Gamma_W$, so the desired result follows from the preceding paragraph. Therefore, we may assume that $m(i,i')\geq 3$ for all distinct $i,i'\in I$. The positive projectivization $\mathbb P(\BB)$ of the base region $\BB$ is a hyperbolic triangle with (acute) angles $\pi/m(i_1,i_2), \pi/m(i_1,i_3), \pi/m(i_2,i_3)$. For each $i\in I$, let $F_i$ be the side of $\mathbb P(\BB)$ that separates $\mathbb P(\BB)$ from $\mathbb P(\BB s_i)$. The \dfn{pedal triangle} of $\mathbb P(\BB)$ is the triangle whose vertices are the feet of the altitudes of $\mathbb P(\BB)$. Let ${\bf v}_i$ be the vertex of the pedal triangle of $\mathbb P(\BB)$ that lies in $F_i$. It is known \cite{McMullen} that the pedal triangle is a closed billiards path (in the classical sense) inside of $\mathbb P(\BB)$. This means that we can extend one of the sides of the pedal triangle of $\mathbb P(\BB)$ into a ray that induces the sequence $i_1, i_2, i_3, \ldots$. To be more precise, let $\mathfrak l$ be the line passing through ${\bf v}_{i_3}$ and ${\bf v}_{i_1}$. Let $x_0$ be a point on $\mathfrak l$ that lies in the interior of $\mathbb P(\BB)$. Let $\rr\colon\mathbb R_{\geq 0}\to \BB W$ be a ray whose positive projectivization $\overline\rr$ satisfies $\overline \rr(0)=x_0$, $\overline \rr(1)={\bf v}_1$, and $\overline \rr(a)\in\mathbb P(\BB)$ for all $0\leq a\leq 1$. Then $\rr$ certifies that $i_1, i_2, i_3$ is a luminous ordering of $I$. See \cref{fig:pedal}.

\section{Futuristic Directions}\label{sec:future}  

\subsection{Characterizations}\label{subsec:e8tilde}
The most outstanding open problem arising from this paper is that of characterizing futuristic Coxeter groups and ancient Coxeter groups. Let us mention some smaller questions that might be more tractable. 


Say a Coxeter group $W$ is \dfn{contemporary} if it is neither futuristic nor ancient. In other words, $W$ is contemporary if there exist Coxeter elements $c$ and $c'$ of $W$ such that $c$ is futuristic while $c'$ is not. 

\begin{question}
Do contemporary Coxeter groups exist? 
\end{question}

Note that \cref{cor:tree} tells us that a Coxeter group whose Coxeter graph is a forest cannot be contemporary. 

Every minimally non-futuristic Coxeter group that we have found is also minimally ancient. Moreover, every minimally ancient Coxeter group that we have found has only finite edge labels in its Coxeter graph. This leads us naturally to the following questions. 

\begin{question}
Does there exist a minimally ancient Coxeter group $W$ such that at least one edge of $\Gamma_W$ is labeled $\infty$? 
\end{question}

\begin{question}
Does there exist a minimally non-futuristic Coxeter group $W$ such that at least one edge of $\Gamma_W$ is labeled $\infty$? 
\end{question}

\subsection{Periodic Points}

Let $c$ be a Coxeter element of $W$, and let $\LL \subseteq W$ be a convex set that is not heavy with respect to $c$. By definition, $\Pro_c$ has a periodic point outside of $\LL$. It would be interesting to gain a better understanding of these periodic points. 

\subsection{Billiards-Plausible Closed Walks}
In \cref{conj:closed_walks}, we asked whether the converse to \cref{lem:closed-billiards-walk-futuristic} holds. In other words, if there exists a closed walk in ${\bf G}_W$ that is billiards-plausible with respect to an ordering $i_1,\ldots,i_n$ of $I$, does it follow that the Coxeter element $c = s_{i_n} \cdots s_{i_1}$ is not futuristic? (The stronger statement that every billiards-plausible closed walk in ${\bf G}_W$ lifts to a billiards trajectory is false; there is a counterexample with $W=\widetilde E_8$.)

\subsection{Decidability}\label{subsec:decidability}

Is there an algorithm for deciding whether or not a Coxeter element is futuristic?  Such an algorithm could conceivably be based on the results of \Cref{sec:small}, which imply that a Coxeter element fails to be futuristic if and only if there is some billiards-plausible closed walk that lifts to a billiards trajectory. Given a billiards-plausible closed walk, one can algorithmically determine whether it lifts to a billiards trajectory. One can also iterate over the billiards-plausible closed walks of a bounded length. So one route to showing that futuristicity is decidable is to find an upper bound for the minimal length of a periodic billiards trajectory or, equivalently, for the minimal length of a billiards-plausible closed walk that lifts.

We remark that there is an algorithm for determining whether or not a given finite convex set $\LL$ is heavy with respect to a given Coxeter element $c$.  In order to describe this algorithm, we require a bit of terminology.
For $Q \subseteq W$, let $\NN(Q)$ denote the directed labeled graph with vertex set $Q \cup \{\tau_i(u) : u \in Q,\, i \in I\}$ that has a directed edge labeled $i$ from $u$ to $\tau_i(u)$ for all $u \in Q$ and all $i \in I$. We say that $Q, Q' \subseteq W$ are \dfn{$\tau$-equivalent} if $\NN(Q)$ and $\NN(Q')$ are isomorphic as labeled directed graphs (that is, there exists a graph isomorphism from $\NN(Q)$ to $\NN(Q')$ that preserves the direction and label of each edge).

Let $\LL$ be a finite convex set, and let $c$ be a Coxeter element.  Recall from \cref{lem:referee} that the sizes of the strata are bounded above by $2^{|\Phi \setminus (\RR(\LL) \cup (-\RR(\LL)))|}$.
It follows that there are only finitely many distinct proper strata up to $\tau$-equivalence.
It is possible to algorithmically find one representative from each $\tau$-equivalence class of proper strata. It then suffices to search for periodic billiards trajectories in these representatives. Since each stratum is finite, this search can be completed in a finite amount of time.

\subsection{Operators from Long Elements}

Let $\mathsf{w}_\circ$ be a reduced word for the long element of a finite Coxeter group $W$. Fix a nonempty convex subset $\LL$ of $W$ as before. We proved (\cref{thm:main1}) that $\tau_{\mathsf{w}_\circ}(W)=\LL$. It could be interesting to study the fibers of $\tau_{\mathsf{w}_\circ}$, even in the special case where $W=\mathfrak S_n$ and $\tau_{\mathsf{w}_\circ}$ is the extended evacuation operator $\Ev$.

One could also study the dynamics of the restriction of $\tau_{\mathsf{w}_\circ}$ to $\LL$ (which is a bijection). When $W=\mathfrak S_n$ and $\tau_{\mathsf{w}_\circ}=\Ev$, the restriction of $\tau_{\mathsf{w}_0}$ to $\LL$ is Sch\"utzenberger's evacuation map, which is well known to be an involution. In general, however, the restriction of $\tau_{\mathsf{w}_\circ}$ to $\LL$ need not be an involution. 

\subsection{Sorting Times}
Let $c$ be a Coxeter element of $W$, and let $\LL$ be a nonempty finite convex set that is heavy with respect to $c$. By multiplying $\LL$ on the right by an element of $W$ if necessary, we may assume without loss of generality that $\id\in\LL$. 
For $u \in W$, let $\TT(u)$ be the smallest nonnegative integer $K$ such that $\Pro_c^K(u) \in \LL$. We call $\TT(u)$ the \dfn{sorting time} of $u$. It is natural to ask about the asymptotic growth rate of $\TT(u)$ as $\ell(u)$ grows.  (Recall from \Cref{sec:preliminaries} that $\ell(u)$ denotes the length of $u$.)  

Since each noninvertible Bender--Knuth toggle can decrease the length of an element by at most $1$, it is clear that $\TT(u) \geq \ell(u)/|I|-O(1)$. The finiteness of $\LL$ ensures that all strata are finite and that there are only finitely many $\tau$-equivalence classes of strata (see \cref{subsec:decidability}).  Since $\LL$ is heavy with respect to $c$, there exists an integer $k\geq 1$ such that $\Sep(\Pro_c^k(v))\subsetneq\Sep(v)$ for all $v\in W\setminus\LL$. So, $\TT(u)\leq k|\Sep(u)|\leq k\ell(u)$. This shows that $\TT(u)$ grows linearly with $\ell(u)$. 

Even for specific choices of $W$, $c$, and $\LL$, it could be interesting to compute or estimate the quantity \[C_0(c,\LL):=\limsup_{\ell(u) \to \infty} \TT(u)/\ell(u).\] Is there an upper bound on $C_0(c, \LL)$ that depends only on $W$? One could also study the quantity \[\limsup\limits_{\ell(u) \to \infty} \TT(u)/|\Sep(u)|,\] which is closely related to $C_0(c,\LL)$.

\subsection{Luminous Orderings}
In \cref{sec:linear}, we discussed two settings where we know that an ordering $i_1,\ldots,i_n$ of $I$ is luminous.

\begin{question}\label{quest:luminous}
For which Coxeter groups is it the case that every ordering of $I$ is a cyclic shift of a luminous ordering? 
\end{question}

\subsection{Closed and Antisymmetric Sets of Roots}\label{subsec:closed_roots}

One can generalize the definition of the noninvertible Bender--Knuth toggles in \cref{def:toggles} by replacing the set $\RR(\LL)$ with an arbitrary closed and antisymmetric set $\RR\subseteq\Phi$. Analogues of \cref{lem:no-root-in-positive-span,lem:acute-angle,thm:completely_orthogonal_lemma,thm:power-completely-orthogonal} remain true in this more general setting, but \cref{lem:super-strong-acute-angle,lem:transmitting-small} become false. 

One special case of this more general setting is that in which $\RR=\Phi^-$. In this case, the resulting noninvertible Bender--Knuth toggles generate ${\bf H}_W(0)$, the $0$-Hecke monoid of $W$, and their action on $W$ defines the standard action of ${\bf H}_W(0)$ on $W$ (see \cite{Hivert, Kenney}). Said differently, $\tau_{\mathsf w}(\id)$ is the Demazure product of a word $\mathsf{w}$. 

\subsection{Other Toggle Sequences} 

Let $\LL$ be a nonempty convex subset of $W$. Given any infinite sequence ${\bf i}=(i_1,i_2,\ldots)$ of elements of $I$ and any starting point $u_0\in W$, we can construct the sequence $u_0,u_1,u_2,\ldots$ by the recurrence relation $u_j=\tau_{i_j}(u_{j-1})$ for all $j\geq 1$. In this article, we have been primarily concerned with the setting in which ${\bf i}$ arises by repeating some ordering of $I$. However, it could be fruitful to consider other sequences, such as sequences induced by light beams. We could, for instance, say that ${\bf i}$ is \dfn{futuristic} if for every nonempty finite convex set $\LL$ and every starting point $u_0$, the resulting sequence $u_0,u_1,u_2,\ldots$ eventually reaches $\LL$. 
If $\bf i$ is formed by repeating some ordering $i_1,\ldots,i_n$ of $I$, then ${\bf i}$ is futuristic if and only if the Coxeter element $c=s_{i_n}\cdots s_{i_1}$ is futuristic (by \cref{prop:finite-to-infinite}). 

It could also be interesting to consider the setting in which the noninvertible Bender--Knuth toggles are applied in a random order. This produces a Markov chain with state space $W$ in which each transition is given by applying $\tau_i$, where $i\in I$ is chosen randomly (according to some probability distribution on $I$). This random process was studied by Lam \cite{Lam} in the case where $W$ is an affine Weyl group and the noninvertible Bender--Knuth toggles are defined via the set $\RR=\Phi^-$ (as discussed in \cref{subsec:closed_roots}).

\section*{Acknowledgments}
Grant Barkley was supported in part by the National Science Foundation grant DMS-1854512. Colin Defant was supported by the National Science Foundation under Award No.\ 2201907 and by a Benjamin Peirce Fellowship at Harvard University. Eliot Hodges was supported by Jane
Street Capital, the National Security Agency, the National Science Foundation (grants 2140043
and 2052036), and the Harvard College Research Program. Noah Kravitz was supported in part by a National Science Foundation Graduate Research
Fellowship (grant DGE--2039656). This work initiated while Colin Defant, Eliot Hodges, Noah Kravitz, and Mitchell Lee were visiting and/or working at the University of Minnesota Duluth Mathematics REU in 2023 with support from Jane Street Capital; we thank Joe Gallian for providing this wonderful opportunity. We are grateful to Jon McCammond, Curtis McMullen, and Ilaria Seidel for several helpful conversations. We thank the anonymous referee for several helpful comments.

\end{document}